\documentclass[10pt,english]{article}

\usepackage{csquotes}
\usepackage[dvipsnames]{xcolor}
\usepackage[backend=biber,style=numeric,sortcites,natbib=true,sorting=nty,hyperref=true,backref=true]{biblatex}
\usepackage{authblk}

\usepackage[a4paper]{geometry}
\usepackage[english]{babel}
\usepackage[toc,page]{appendix}

\usepackage{mathtools}
\usepackage{amssymb,amsfonts}
\usepackage{amsthm}
\usepackage{aliascnt}
\usepackage{dsfont}
\usepackage[all]{xy}
\usepackage[scr=boondox, cal=esstix]{mathalpha}

\usepackage{float}
\usepackage{subfig}
\usepackage{graphicx}
\usepackage[font=small,labelfont=bf]{caption}
\usepackage{tikz}
\usetikzlibrary{arrows.meta,calc,positioning}

\newtheorem{theorem}{Theorem}[section]

\newaliascnt{lemma}{theorem}
\newtheorem{lemma}[lemma]{Lemma}
\aliascntresetthe{lemma}

\newaliascnt{proposition}{theorem}
\newtheorem{proposition}[proposition]{Proposition}
\aliascntresetthe{proposition}

\newaliascnt{corollary}{theorem}
\newtheorem{corollary}[corollary]{Corollary}
\aliascntresetthe{corollary}

\theoremstyle{definition}

\newaliascnt{definition}{theorem}
\newtheorem{definition}[definition]{Definition}
\aliascntresetthe{definition}

\theoremstyle{remark}

\newaliascnt{remark}{theorem}
\newtheorem{remark}[remark]{Remark}
\aliascntresetthe{remark}

\theoremstyle{plain}
\newtheorem{maintheorem}{Theorem}

\numberwithin{equation}{section}

\usepackage{hyperref}
\usepackage[capitalise,nameinlink,sort]{cleveref}

\crefname{theorem}{theorem}{theorems}
\Crefname{theorem}{Theorem}{Theorems}

\crefname{lemma}{lemma}{lemmas}
\Crefname{lemma}{Lemma}{Lemmas}

\crefname{proposition}{proposition}{propositions}
\Crefname{proposition}{Proposition}{Propositions}

\crefname{corollary}{corollary}{corollaries}
\Crefname{corollary}{Corollary}{Corollaries}

\crefname{definition}{definition}{definitions}
\Crefname{definition}{Definition}{Definitions}

\crefname{remark}{remark}{remarks}
\Crefname{remark}{Remark}{Remarks}

\crefname{maintheorem}{theorem}{theorems}
\Crefname{maintheorem}{Theorem}{Theorems}

\crefname{subsection}{subsection}{subsections}
\Crefname{subsection}{Subsection}{Subsections}

\newtheorem{claim}[theorem]{Claim}
\crefname{claim}{claim}{claims}
\Crefname{claim}{Claim}{Claim}

\theoremstyle{plain}

\newtheorem*{notation}{Notation}

\newcommand{\norm}[1]{\left\Vert#1\right\Vert}
\newcommand{\conorm}[1]{\mathrm{m}\left(#1\right)}

\newcommand{\Real}{\mathbb R}
\newcommand{\R}{\Real}
\newcommand{\abs}[1]{\left\vert#1\right\vert}
\newcommand{\Tor}{\mathbb{T}}

\newcommand{\Pp}{\mathbb P}
\newcommand{\Nat}{\mathbb N}

\newcommand{\N}{\mathbb{N}}
\newcommand{\oo}{\infty}

\newcommand{\vol}{\operatorname{vol}}
\newcommand{\pr}{\mathcal{P}}
\newcommand{\disu}[2]{\operatorname{d}_u{(#1},{#2})}
\newcommand{\diam}{\operatorname{diam}}
\newcommand{\dd}{\operatorname{d}\!}

\newcommand{\Fs}{\ensuremath{\mathcal{W}^s}}

\newcommand{\Fu}{\ensuremath{\mathcal{W}^u}}

\newcommand{\Ws}[1]{\ensuremath{W^s(#1)}}
\newcommand{\Wu}[1]{\ensuremath{W^u(#1)}}

\newcommand{\hWu}[1]{\ensuremath{\widehat{W}^u(#1)}}

\newcommand{\der}{\operatorname{d}\hspace{-2pt}}
\newcommand{\Leb}{\mathrm{Leb}}
   
\DeclarePairedDelimiterXPP\clo[1]{\mathtt{cl}}(){}{#1}  

\newcommand{\hu}{\mathrm{h}^u}

\newcommand{\hM}{\widehat{M}}

\newcommand{\hpi}{\widehat{\pi}}
\newcommand{\hE}{\widehat{E}}

\newcommand{\hf}{\hat{f}}

\newcommand{\hmu}{\hat{\mu}}
\newcommand{\hm}{\hat{m}}

\newcommand{\hnu}{\hat{\nu}}
\newcommand{\hp}[1][p]{\hat{#1}}
\newcommand{\hq}{\hat{q}}
\newcommand{\Gibbs}[1][f]{\mathrm{Gibbs}^u(#1)}
\newcommand{\Gibbse}[1][f]{\mathrm{Gibbs}^u_{\mathrm{erg}}(#1)}

\newcommand{\PH}[1][r]{\mathcal{PH}^{#1}(M)}
\newcommand{\diffM}[1][r]{\operatorname{Diff}^{#1}(M)}

\newcommand{\CEA}{\mathrm{C}_{\scriptscriptstyle{EA}}}

\newcommand{\Gap}{\mathrm{Gap}(f)}

\allowdisplaybreaks

\newcommand{\PHb}[1][r]{\mathscr{PH}^{#1}_{\mathrm{uCB}}(M)}
\newcommand{\PHbvol}[1][r]{\mathscr{PH}^{#1}_{\mathrm{uCB},\mathrm{vol}}(M)}

\newcommand{\diag}{\mathtt{diag}}
\newcommand{\ep}{\epsilon}
\newcommand{\al}{\alpha}
\newcommand{\eps}{\varepsilon}

\newcommand{\lam}{\lambda}

\newcommand{\Wuloc}[1]{W^{u}_{\mathrm{loc}}(#1)}

\newcommand{\hWuloc}[1]{\hat{W}^{u}_{\mathrm{loc}}(#1)}

\newcommand{\cE}{\mathcal E}
\newcommand{\cR}{\mathcal R}
\newcommand{\mc}[1]{\mathcal #1}

\definecolor{lapislazuli}{rgb}{0.15, 0.38, 0.61}

\hypersetup{
    colorlinks=true,
    linkcolor=lapislazuli,
    citecolor=lapislazuli,
    urlcolor=lapislazuli
    }

\makeatletter
\newcommand{\subjclass}[2][2020]{%
  \let\@oldtitle\@title%
  \gdef\@title{\@oldtitle\footnotetext{#1 \emph{MSC2020 classification:} #2}}%
}
\newcommand{\keywords}[1]{%
  \let\@@oldtitle\@title%
  \gdef\@title{\@@oldtitle\footnotetext{\emph{Keywords:} #1.}}%
}
\makeatother

\begin{document}

 \title{Expanding on Average  Diffeomorphisms}
 \author{
Pablo D. Carrasco \quad
Davi Obata \quad
Radu Saghin \quad
Jiagang Yang
}

  \date{\today}

  \maketitle
  \begin{abstract}
We study the statistical properties of partially hyperbolic diffeomorphisms with a two-dimensional center exhibiting mixed behavior, without assuming a dominated splitting inside the center. We introduce the notion of an \emph{expanding on average} for partially hyperbolic diffeomorphism and, under a center jacobian assumption, we prove the existence of finitely many ergodic hyperbolic SRB measures, and show the strong positive recurrence property, which implies that each SRB measure, up to a period, satisfies exponential decay of correlations,  among other strong statistical laws. We apply our results to several examples. 
\end{abstract}

\tableofcontents
\selectlanguage{english}

\section{Introduction} 
\label{sec:introduction_and_results}


\subsection{Motivation and main result}


In this article we are interested in the statistical properties of differentiable dynamical systems. For a smooth diffeomorphism $f:M\to M$, where from now on $M$ denotes a closed manifold, we want to understand the distribution of the $f$-orbits. To make the question more precise one can do the following: given $p\in M$ we look at the possible accumulation points of the empirical measures
\[
   \delta^{(n)}_{p}=\frac{1}{n}\sum_{i=0}^{n-1} \delta_{f^i(p)},
\]
where $\delta_x$ is the Dirac measure at a point $x$, and we want to characterize the properties of these limits. Given an ergodic\footnote{$\mu$ is ergodic for $f$ if every $f$-invariant measurable set has $\mu$-measure $0$ or $1$.} $f$-invariant measure $\mu$, its \emph{basin} is defined as
\[
\mathcal{B}(\mu) \coloneq \{p\in M: \lim_{n\to +\infty} \delta_x^{(p)} = \mu\}.
\]
We say that $\mu$ is \emph{physical}, if $\mathcal{B}(\mu)$ has positive volume. For most dynamical systems, volume is not an invariant measure. Physical measures capture the asymptotic statistical behavior for a set of points relevant for the natural measure class on the manifold. 

 For systems displaying hyperbolic behavior,  physical measures are often \emph{hyperbolic SRB measures}, introduced by Sinai, Ruelle and Bowen in \cite{Sinai_1972,SRBattractor, Bowenbook} (see  \Cref{sec:Preliminaries} for the precise definition).  A classical result in smooth dynamics states that for a $C^2$ Anosov diffeomorphism, there are finitely many SRB measures whose union of their basins has full volume.  Palis conjectured that this result holds for a typical dynamical system \cite{Palisconj}.  Since then, finding conditions for the existence and finiteness (or uniqueness) of  SRB physical measures has become an active area of research in smooth dynamics. 
 
 A generalization of uniformly hyperbolic theory is the nonuniformly hyperbolic theory. Existence of SRB measures in this setting include works such as \cite{BenOvadiaSRB,BenedicksYoung,ClimenhagaLuzzattoPesin2022, Young1998}. Some important recent  advances were made for the existence of SRB measures for $C^\infty$-diffeomorphisms having positive Lyapunov exponents on a set of positive volume \cite{Burguetsrb,BCSSRB, snirburguet}.

Another generalization of Anosov diffeomorphism are the so-called \emph{partially hyperbolic systems}, where besides the existence of  expanding and contracting directions there is an additional distribution with some intermediate behavior. The precise definition is recalled in \Cref{sec:Preliminaries}. Denote by $\PH$ the set of partially hyperbolic $C^r$-diffeomorphisms of $M$. Partially hyperbolic systems form a robust and flexible class of dynamics and arise naturally in many settings, including geometry, mathematical physics, and homogeneous dynamics.

There are important works due to Bonatti--Viana \cite{Bonatti2000}, and Alves--Bonatti--Viana \cite{SRBMostexp}, which show the existence and finiteness of  SRB physical measures for partially hyperbolic systems that are mostly contracting, and mostly expanding, respectively. In the mostly contracting case, the center direction has a nonuniformly contracting behavior for volume almost every point, whereas the mostly expanding case, has a similar definition with a nonuniform expansion along the center. 

A major remaining difficulty is to treat partially hyperbolic systems whose center dynamics is genuinely mixed: the center direction may exhibit both expansion and contraction at the level of Lyapunov exponents, without assuming a dominated splitting inside the center or a uniform sign condition on all center exponents. This is the setting considered in this paper. Let us mention that we do not study the case where there is a zero exponent in the center.

For this purpose we introduce the concept of \emph{expanding on average} partially hyperbolic diffeomorphisms and show that this property implies the existence of finitely many ergodic SRB measures whose basins cover a full volume set.  Not only that, it also implies many good statistical properties for the orbits.  This property is inspired in the expanding on average (also known as uniform expansion) property in random dynamics (see for instance \cite{Dolgopyat2007, deWittDolgo2024, Chung2020, Liuthesis}).  Informally, our condition requires that every center direction, when transported along unstable leaves by the center holonomies, experiences positive average exponential growth along every unstable disk. In what follows, see \Cref{sec:Preliminaries} for the definitions and background material about partially hyperbolic dynamics. 

Let $f:M \to M$ be a $C^2$ partially hyperbolic diffeomorphism with a splitting $TM = E^s \oplus E^c \oplus E^u$ such that $\dim(E^c) = 2$, which verifies a technical condition called \emph{$u-$  center bunching}, see \Cref{def:bunching} for the precise definition.  We denote by $\PHb[2]$ the set of  such diffeomorphisms. Write $\Fu=\{\Wu{p}\}_{p\in M}$ for the unstable foliation of $f$.  The $u$-center bunching condition implies the existence of linear holonomies along unstable manifolds, that is, for $q\in \Wu{p}$, there is a holonomy map $\hu_{q,p}: E^c_p \to E^c_q$ that permites to transport center vectors, and depends continuously from $p,q$. In particular, given some non-zero vector $v\in E^c_p$ one can extend it to a unit vector field along $\Wu{p}$, by
\[
\chi_{q,p}(v):=\frac{\hu_{q,p}(v)}{\norm{\hu_{q,p}(v)}}.
\]

\begin{definition}\label{def:AUEP}
    Under the assumptions above on $f$, we say that $f$ is \emph{expanding on average} if  there exists $C>0$ so that: for every  disk contained in an unstable manifold centered at $p$, $D_p$,  and every unit vector $v\in E^c_p$, it holds
    \begin{equation}\label{eq.inequalityea}
    \liminf_{n\to\oo}\frac{1}{n\Leb^u(D_p)}\int_{D_p}\log \norm{D_qf^n\chi_{q,p}(v)} \der\Leb^u(q)\geq C,
    \end{equation}
    where $\Leb^u$ is the Lebesgue measure along $\Wu{p}$ induced by the riemaniann metric.
\end{definition}

In \Cref{sec:Somebasicconsequences}, we define the notion of an optimal constant for \Cref{def:AUEP}. Using this optimal constant, we define a quantity that we call \emph{exponent gap} which is denoted by $\Gap$, see \eqref{def:Gap}. This quantity gives a lower bound on the difference of the center Lyapunov exponents for $u$-Gibbs measures.  Our main result is the following.

\begin{maintheorem}\label{mainthm:TheoremA}
If $f\in \PHb[2]$ is expanding on average and satisfies $\Gap>0$, then there are finitely many ergodic  hyperbolic SRB physical measures $\mu_1, \cdots, \mu_k$, such that
\[
\displaystyle \vol \left( \bigcup_{i=1}^k\mathcal{B}(\mu_i)\right)=1,
\]
where $\vol$ is the normalized Lebesgue measure on $M$. 
\end{maintheorem}

As we explain below, the expanding on average property implies strong limit laws for the resulting SRB measure.

\subsection{Further properties of expanding on average systems}

In the partially hyperbolic setting, physical measures, whenever they exist, belong to a special class of measures called \emph{$u$-Gibbs measures}. Roughly speaking, an invariant measure $\mu$ is $u$-Gibbs if it admits conditional measures along unstable manifolds that are absolutely continuous (see \Cref{def:ugibbs} for the precise definition). Given $f\in \PH$, denote by $\Gibbs$ the set of such measures. The set of $u$-Gibbs measures is a compact, convex and nonempty subset of the set of invariant measures of $f$. As we will see, the expanding on average condition actually is equivalent to hyperbolicity and a certain geometrical condition for the $u$-Gibbs measures (see Theorem \hyperlink{theoremD}{D}).

Our next result is about continuity of Lyapunov exponents for such measures. Before stating it, let us recall the definition of Lyapunov exponents. For any $p\in M$, the top center Lyapunov exponent of $f$ at $p$ is defined by
\begin{align}\label{eq:topexp}
\lambda^{+c}(p)=\lambda^{+c}(p,f) \coloneqq \lim_{n\to + \infty} \frac{1}{n} \log\norm{D_pf^n|_{E^c}},
\end{align}
whenever the limit exits. Similarly, the bottom center Lyapunov exponent is defined by 
\begin{align}\label{eq:botexp}
\lambda^{-c}(p)=\lambda^{-c}(p,f) \coloneqq \lim_{n\to + \infty} -\frac{1}{n} \log\norm{D_pf^{-n}|_{E^c}}, 
\end{align}
provided that the limit exists. It is by now a classical result of Furstenberg and Kesten that given an ergodic measure $\mu$ there are numbers $\lambda^{+c}(\mu,f)$ and $\lambda^{-c}(\mu,f)$ such that for $\mu$-almost every $p$, $\lambda^{+c}(p)= \lambda^{+c}(\mu)$ and $ \lambda^{-c}(p)=\lambda^{-c}(\mu)$ (see \Cref{sec:Preliminaries}). When $\mu$ is not ergodic, one considers the average exponents, given by $\lambda^{*c}(\mu) = \int \lambda^{*c}d\mu$, for $*=-,+$.

An important problem in the theory of Lyapunov exponents is to understand the regularity of the functions $(\mu,f) \mapsto \lambda^{*c}(\mu,f)$, for $*=-,+$.

\begin{maintheorem}\label{mainthm:TheoremB}
 Let $f\in \PHb[2]$ be an expanding on average diffeomorphism that satisfies $\Gap>0$. Let $(f_n)_n\subset \PHb[2]$ be a sequence of diffeomorphisms converging to $f$ in the $C^2$-topology. For each $n$ suppose that $\mu_n$ is a $u$-Gibbs measure for $f_n$ and that $\mu_n$ converges to some measure $\mu$ in the weak* topology. Then 
 \[
 \displaystyle \lim_{n\to +\infty} \lambda^{-c}(\mu_n,f_n) = \lambda^{-c}(\mu,f) \textrm{ and }  \lim_{n\to +\infty} \lambda^{+c}(\mu_n,f_n) = \lambda^{+c}(\mu,f).
 \]
\end{maintheorem}

Recently, continuity  of Lyapunov exponents was used to obtain strong statistical properties of the system. This is done with a notion of \emph{Strong Positive Recurrence} (SPR) introduced by Buzzi--Crovisier--Sarig \cite{BCSSPR}, see \Cref{Sec:SPR}  for the precise definitions. In this paper, we show that the expanding on average property with the positive exponent gap condition implies the SPR property. In particular, this gives several strong statistical properties as stated in the next result.

\begin{maintheorem}\label{mainthm:TheoremC}
If $f\in \PHb[2]$ is expanding on average and satisfies $\Gap>0$, for every $p\in M$. Then,  for any ergodic SRB measure $\mu$ of $f$, there exist $k\in \Nat$ and measures $\mu_1, \cdots, \mu_k$ such that
\[
\mu = \frac{1}{k} \sum_{j=1}^k \mu_j,
\]
for each $j=1, \cdots, k$, $f_* \mu_j = \mu_{j+1}$, with the convention that $k+1 = 1$, and $(\mu_j, f^k)$ verifies exponential decay of correlations for H\"older continuous functions, the Central Limit Theorem, large deviations, and other limit laws. 
\end{maintheorem}

For an expanding on average diffeomorphism $f\in \PHb[2]$, the assumption $|\det Df|E^c|\leq 1$ implies the condition $\Gap>0$. This includes both the conservative and dissipative-along the center cases. In particular, as a direct consequence of this remarks we get the following.

\begin{corollary}\label{cor:detlessonecoro}
    If $f\in \PHb[2]$ is expanding on average and satisfies 
    \[
    \sup_{x\in M}|\det D_xf|_{E^c}| \leq 1,\]
    then $f$ verifies the conclusions of \Cref{mainthm:TheoremA}, \Cref{mainthm:TheoremB} and \Cref{mainthm:TheoremC}.
\end{corollary}

\begin{remark}
We mentioned that in \cite{BCSSPR}, the authors check the SPR property as a consequence of the continuity of Lyapunov exponents. In our setting,\Cref{mainthm:TheoremB} alone does not give the necessary continuity to check the SPR property. The reason for that is that in \Cref{mainthm:TheoremB} we only have continuity among $u$-Gibbs measures. In \Cref{mainthm:TheoremC}, we need to ``spread'' this continuity for measures with large pressure, but that are not necessarily $u$-Gibbs measures.
\end{remark}

In what follows, write $\PHbvol[r]\subset \PHb[r]$ be the set of partially hyperbolic $C^r$-diffeomorphisms with two dimensional center, the $u$-center bunching condition and that preserve a normalized volume measure $\mathrm{vol}$. Combining our results with Proposition $1.12$ from \cite{liaotaotongyang}, we obtain the following theorem.

\begin{maintheorem}
    \label{thm:Theorem E}
    Suppose that $f\in \PHbvol[2]$ be expanding on average with $\Gap>0$. If $\mathrm{vol}$ is ergodic for $f$, then for any $g\in \PHb[2]$ sufficiently $C^2$-close to $f$, there exists a unique SRB measure $\mu_g$ such that
    \begin{enumerate}
        \item $\mathrm{vol}(\mathcal{B}(\mu_g)) = 1$;
        \item $\mu_g = \frac{1}{k} \sum_{j=1}^k \mu_{g,j}$, as in \Cref{mainthm:TheoremC}, in particular, each $\mu_{g,j}$ is exponential mixing for $g^k$, among other statistical properties. 
    \end{enumerate}
    Moreover, if $\mu$ is Bernoulli, then for any $g\in \PHb[2]$ sufficiently $C^2$-close to $f$, the unique SRB measure above $\mu_g$ is Bernoulli and it has exponential decay of correlations for $g$, among other strong statistical properties. 
\end{maintheorem}

    \Cref{thm:Theorem E} follows directly from \Cref{mainthm:TheoremA}, \Cref{mainthm:TheoremC} and Proposition $1.12$ from \cite{liaotaotongyang}.

A direct consequence of a recent work by Cooper \cite{Coopererg}, shows that if $f\in \PHbvol[2]$ is expanding on average and verifies:
\begin{itemize}
    \item $\dim(E^u) = 1$;
    \item $|\det D_xf|_{E^c}| = 1$ for every $x\in M$; and 
    \item $f$ is dynamically coherent, that is, there exist a center stable and a center unstable foliation,
\end{itemize}
then $f$ is ergodic and, therefore, verifies the assumptions in \Cref{thm:Theorem E} (see \Cref{context} for more explanation about \cite{Coopererg}).

\subsection{Examples and applications}

In \Cref{sec:Examples} we give several examples of partially hyperbolic systems verifying the expanding on average condition with exponent gap. Let us summarize some of the examples we obtain:
\begin{itemize}
    \item In \Cref{subsec:henonlike}, we introduce examples derived from conservative H\'enon like maps. These examples are skew products, first coupling a H\'enon like map with an expanding endomorphism of the circle and then looking at its natural extension \eqref{eq.henonlike}.
    \item In \Cref{subsec:bcexample}, we recall the Berger-Carrasco's example which was introduced in \cite{Berger2014}. This is a skew product derived from the Taylor-Chirikov standard map, which couples the standard map with and Anosov base \eqref{eq.bcexample}. In \cite{Obata2018} it was showed that this example is $C^2$-stably ergodic. In particular, the example itself is ergodic, and actually Bernoulli. Combining this with \Cref{thm:Theorem E}, we obtain strong statistical properties for the SRB measure of any perturbation of this example, see \Cref{cor:bchasitall}. 
    \item The previous two examples are skew product constructions, in particular, there is a center foliation and it is formed by compact leaves. In \Cref{subsec:noncompactcenter}, we give examples of expanding on average systems with noncompact center foliation. These examples were obtained in \cite{extendedflex}.
 
\end{itemize}

\subsection{Context and relation with other works}\label{context}

In the random dynamics setting, the expanding on average condition appears in the works \cite{benardzhang,bxystandard,blor25,befh, cantatdujardin, Chung2020, deWittDolgo2024, Dolgopyat2007, potrieea}. It also appears with a few different names in those references, such as \emph{uniform expansion or weakly expaning}. This property is responsible for many rigidity results \cite{befh,Brown_2017, EskinLindenstrauss}, strong statistical properties \cite{deWittDolgo2024}.

In the context of differentiable systems, estimates similar to \eqref{eq.inequalityea} appears in other contexts in the literature, although its consequences do not seem to have been studied in isolation. The first instance is perhaps Yoccoz proof of Jakobson's theorem from around 1997, which first circulated in manuscript form but was not public until more recently \cite{yoccozJakobson}. For the diffeomorphism case an inequality similar to \eqref{eq.inequalityea} is the main thing obtained in the proof of the results in Berger--Carrasco \cite{Berger2014}, Carrasco \cite{Carrasco2020} and Carrasco--Saghin \cite{extendedflex}, see also \cite{Obata2023, COP} for related results. For surface endomorphisms it also appears in the recent work of Andersson--Carrasco--Saghin \cite{Andersson2022}, and in the work of Janeiro--Saghin \cite{Janeiro2025}. Although in those works this property is being proved and used implicitly, one of our main goals in this paper is to explicitly define this property and explore its consequences. 

In \cite{Coopererg}, Cooper studies the relationship of expanding on average and stable ergodicity in a volume preserving setting. This problem was suggested to him by the second author  during the writing of this paper. In \cite{Coopererg}, Cooper uses a slightly different notion of expanding on average, which is closer to what we call \emph{finitary expanding on average} (see \Cref{def:finitaryea}), and uses that to show that if $f\in \PHbvol[2]$ has center jacobian constant to $1$, has one dimensional strong unstable and it is dynamically coherent, that is, there exists a center stable and center unstable foliations, then $f$ is ergodic. Moreover, if dynamical coherence holds robustly, then $f$ is $C^2$-stably ergodic. Expanding on average in the random setting implies stable ergodicity, the idea of the problem given to Cooper was to extend that result in the partially hyperbolic setting. However the proof is a lot more difficult in that setting, as it is shown in \cite{Coopererg}. The arguments in \Cref{section.projectiveugibbs} shows that our definition of expanding on average implies the definition used in \cite{Coopererg}.

\subsection{Organization of the paper}

We start in the next section by giving some preliminary material that we will use in the rest of the article. In particular, the definition of $u$-Gibbs measure, SRB measure, and the notion of holonomies is recalled.

Then in \Cref{sec:Somebasicconsequences} we establish dynamical consequences of the expanding on average property. We precisely define the exponent gap quantity appearing in the main theorems, and establish hyperbolicity of any ergodic $u$-Gibbs measure for a expanding on average system with positive exponent gap. See \Cref{cor:Gibbshyperbolic}.

\Cref{section.projectiveugibbs} is devoted to study of projective $u$-Gibbs measures. A typical tool to understand the behavior of invariant measures for a given partially hyperbolic diffeomorphism is to consider the lift to the projectivization of its center bundle. If $E^c$ is differentiable, this extension is a manifold and, under the bunching hypothesis, the lifted map is partially hyperbolic as well. We are working under more general hypotheses, not assuming differentiability of $E^c$ (nor its integrability), and without the full bunching condition. By the general theory, the projectivization of the center bundle can only be shown to be a Hölder bundle, and therefore a large part of the geometry of $u$-Gibbs measures will be lost under the lift. We are forced to introduce a more general framework of $\emph{projective $u$-Gibbs}$ measures, and re-prove some central results of the theory of the (geometrical) $u$-Gibbs measures to them. After achieving this, we give an important equivalent formulation to the expanding on average property: it is equivalent (under the positive gap condition) to having the Lyapunov spectrum of every $u$-Gibbs being uniformly bounded away from zero, together with the essential non-invariance of weak stable directions under strong unstable holonomy for such measures. See \Cref{mainthm:TheoremD}.

In \Cref{THM:robustnessofea} we establish an important result of the expanding on average property: it is $C^2$-robust. Then in \Cref{Subsection:Continuityexp} we prove \Cref{mainthm:TheoremB}, establishing in particular the continuity of the Lyapunov exponents among $u$-Gibbs measures.

After that we move to establishing a technical result concerning Large Deviation bounds for projective $u$-Gibbs. Large deviations results are known to hold in unstable manifolds for partially hyperbolic systems (see \cite{Beyond,crovisierld,dolgopyatlimittheorem,lecturenotesdolgopyat}). For this Large Deviations estimate for the projective dynamics, we do an adaptation of a result of Dolgopyat \cite{dolgopyatlimittheorem,lecturenotesdolgopyat}, which we prove in \Cref{Subsection:Continuityexp}. Besides some interesting consequences given in \Cref{ssec:ConsequencesLD}, the Large Deviations estimates allow us to prove an important technical tool for our work (\Cref{thm:existsweakbundle}), which gives that on \emph{every} unstable leaf, on Lebesgue almost every point of that leaf there is a well defined contracting direction. 

Using the already established \Cref{mainthm:TheoremB}, in \Cref{Sec:uniformpesinblock} we prove the existence of a uniform Pesin block for the $u$-Gibbs measures, that is, a set having uniform control on expansion and contraction estimates whose measure with respect to every $u$-Gibbs is uniformly bounded from below. Using homoclinic intersections and a Hopf argument, this implies the existence of at most finitely many ergodic SRB measures. To address the existence of such measures, we prove in \Cref{sec:existenceSRB} that the expanding on average property together with positive gap implies a slightly different version of the effective hyperbolicity condition of Climenhaga-Dolgopyat-Pesin \cite{CDP}. Using their work, we obtain the existence of SRB measures.

\Cref{Sec:SPR} allow us to reap several consequences of our work, in particular strong limit laws and exponential mixing of Hölder observables, for the SRB measures of our expanding on average systems. For this we prove prove that our systems verify the SPR property recently introduced by Buzzi--Crovisier--Sarig \cite{BCSSPR}. This is achieved adapting some semi continuity results for the unstable entropy given in \cite{Yangunstableentropy}.

The last \Cref{sec:Examples} consists of several examples of systems having the expanding on average property. We suggest the reader to read the main theorems in the introduction and check the examples given in that section, to get a feeling of the extensive consequences of the results.

Let us emphasize that the understanding we obtained in \Cref{section.projectiveugibbs} is essential for obtaining all of the other results we have in the paper.

\section{Preliminaries}
\label{sec:Preliminaries}

In this section we establish some background material needed, and fix notation.

\subsection{Partial hyperbolicity and holonomies}
\label{ssec:Partialhyperbolicityandholonomies}

Throughout, $M$ denotes a closed manifold.

\begin{definition}\label{def:PH}
A $C^r$ diffeomorphism ($r\geq 1$) $f:M\to M$ is partially hyperbolic if there exists some Riemannian metric on $M$ and a $Df$-invariant splitting $TM=E^s\oplus E^c\oplus E^u$, where none of these bundles are trivial, and such that for every $p\in M$
\begin{align*}
\norm{D_pf|_{E^s}}< \min\left\{1, m(D_pf|_{E^c})\right\}\leq \max\left\{1, \norm{D_pf|_{E^c}}\right\}<m(D_pf|_{E^u}),
\end{align*}
where \(m(D_pf|_{E^*})=\norm{\left(D_{p}f\right)^{-1}|_{E^*}}^{-1}\), for $*=c,u$,  is the conorm.
We denote by $\PH[r]$ the set of partially hyperbolic diffeomorphisms of $M$ of class $C^r$.
\end{definition}

The distributions $E^s, E^c, E^u$ are called the \emph{stable}, \emph{center} and \emph{unstable} bundles of $f$, while 
\[
E^{cs}=E^s\oplus E^c\textrm{ and } E^{cu}=E^c\oplus E^{u}
\]
are the \emph{center stable} and \emph{center unstable} bundles. It is well known that $\PH[r]$ is $C^1$ open in $\diffM$. For $f\in \PH[2]$ its invariant distributions are at least $\alpha$-Hölder continuous, for some $\alpha\in (0,1)$. Also, due to the classical Stable Manifold Theorem, if $f\in \PH[r], r\geq 1$,  then $E^s, E^u$ are uniquely integrable to continuous foliations  $\Fs=\Fs_f=\{\Ws{p}\}_{p\in M}, \Fu=\Fu_f=\{\Wu{p}\}_{p\in M}$, whose leaves are of class $C^r$. See \cite{PesinLect,CrovisierPotrie2015} for these and other basic properties of partially hyperbolic systems. 

Given $p\in M$ and $R>0$ we write $\Wu{p,R}$ to be the piece of strong unstable manifold of size $R$. If we need to make the diffeomorphism explicit, we will write $\Wu{p,R,f}$.

In contrast, the center bundle is not always integrable, although it is a condition that typically holds in concrete applications. We will not assume that $E^c$ integrates.

\begin{notation}
Given  $f\in \PH, p\in M, *\in\{s,c,u\}$,  we write $D^*_pf=D_pf|_{E^*}$. 
\end{notation}

\begin{definition}[$u$-Center Bunching]\label{def:bunching}
We say that $f$ is $u-$center bunched if for every $p\in M$,
\[
  \frac{\norm{D^c_pf}}{\conorm{D^c_pf}}<\conorm{D^u_pf}.
\]
The set of all $u$-center bunched partially hyperbolic diffeomorphims is denoted by  $\PHb$.
\end{definition}

\begin{remark}  
Usually, in partially hyperbolic dynamics, people study a stronger condition called \emph{center bunching}, which on top of the inequality above, it also requires that 
\[
\norm{D^s_pf}<\frac{\conorm{D^c_pf}}{\norm{D^c_pf}}.
\]
As in this paper most of our analysis happens inside center unstable directions, we will only need $u$-center bunching.
\end{remark}

It is direct to check that $\PHb \subset \diffM$ is $C^1$ open. The main consequence of $u$-center bunching that we will use is the existence of linear holonomies for the center bundle $E^c$ (and its projectivization) along unstable manifolds. Let us explain what we mean by this.

\begin{notation}
   If $f\in \PH$ is fixed, then $p\sim_u q$ means that $p\in\Wu{q}$. Similarly, $p\sim_s q$ means that $p\in \Ws{q}$.
\end{notation}

\begin{definition}\label{def:linearholonomy}
 Fix $f\in \PH$, and let $\pi:\cE\to M$ be either a continuous vector bundle or a continuous fiber bundle whose fibers are smooth. Denote by $\operatorname{Bun}(\cE)=\{L_{q,p}: \cE(p)\to \cE(q)\}_{p,q\in M}$ the corresponding linear maps, or diffeomorphisms, between fiber. Let also $F:\cE\to\cE$ be a continuous bundle map over $f$.

 An invariant  unstable holonomy for $F$ is a continuous $\operatorname{Bun}(\cE)$-valued cocycle over the relation $\sim_u$, that is, a family of isomorphisms $\{\hu_{p,q}:\cE(q)\to\cE(p):p,q\in M, p\sim_u q\}$ satisfying:
 \begin{enumerate}
     \item $\hu_{p,p} = Id$, and if $q\sim_u p$ and $z\sim_uq$, the $\hu_{z,p} = \hu_{z,q}\circ\hu_{q,p}$;
     \item $\hu_{q,p}$ depends continuously on the choice of $p$ and $q$ such that $q\sim_u p$;
     \item if $p\sim_u q$ then  $\hu_{f(q),f(p)}=F\circ \hu_{q,p}\circ F^{-1}$.
 \end{enumerate}  
\end{definition}

We are interested in two cases: $\cE=E^c$ and $F=D^cf$, or $\cE=\Pp E^c$ (the projectivization of $E^c$) and $F=\Pp D^cf$, the induced action of $D^cf$ on $\Pp E^c$.

\begin{theorem}[Saghin, \cite{SAGHIN_2024}]\label{thm:saghin}
  If $f\in \PHb, r\geq 2$, then both  $D^cf:E^c\to E^c$ and $\Pp D^cf:\Pp E^c\to \Pp E^c$ admit invariant unstable holonomies.   Moreover, for any $R>0$, if $(f_n)_{n\in\Nat} \subset \PHb$ is a $C^2$-bounded sequence converging in the $C^1$-topology to $f$, and $(p_n)_n, (q_n)_n$ are sequences of points in $M$ with $q_n\in W^u_R(p_n,f_n)$, the unstable manifold for $f_n$ of size $R$ centered at $p_n$, $p_n \xrightarrow[n\mapsto\oo]{}p, q_n \xrightarrow[n\mapsto\oo]{}q$, then $\hu_{q_n,p_n, f_n}$ converges to $\hu_{q,p, f}$. 
\end{theorem}

\begin{notation}
   In both cases, the corresponding  unstable holonomy will be denoted as $\hu$.
\end{notation}

\begin{remark}\label{rem:meaningconvergenceC1}
In the theorem above, the convergence is understood in the following sense. Fix a foliation cube $U$ of $\Fu_f$, and consider a $C^{\oo}$ family of disks $\{\Sigma_x:x\in U\}$ that are transverse to $\Fu_f|_U$, each $\Sigma_x$ centered at $x$. Using the chart in $U$ we get a representation of the holonomy as locally defined maps in $\Real^{cs}$, $cs=\dim M-\dim E^u$. For $n$ large, the disks $\Sigma$ are also transverse to $\Fu_{f_n}$, and therefore we also get a representation of the correponding holonomies as maps in $\Real^h$. The $C^1$ convergence in the theorem refers as usual $C^1$ convergence of the corresponding holonomy maps in these charts. By compactness, this is independent of the foliation chart considered, provided it is sufficiently small.
\end{remark}

\subsection{\texorpdfstring{$u-$}{u}Gibbs measures and SRB measures} 
\label{sub:uGibbs}

 Understanding the existence and finiteness of physical measures in the partially hypebolic setting is closely related to understanding the so-called $u$-Gibbs measures. It is known that $u$-Gibbs measures always exist, they have good compactness properties, and every physical measure must be $u$-Gibbs. So, the set of such measures become a natural place for investigating physical measures.

\begin{notation}
  We denote by $\pr_f(M)$ the set of $f$-invariant Borel probability measures. It is equipped with the weak* topology.
\end{notation}

Let $f\in \PH[r]$, $r\geq 2$ and $\mu\in \pr_f(M)$.  Given a foliated box $U$ of $\Fu$, we have that the foliation $\Fu$ restricted to $U$, $\Fu|_{U}$, gives a $\mu$-measurable partition. In particular, we can disintegrate $\mu$ along the elements of this partition. So, there exists $U'\subset U$ with $\mu(U\setminus U')=0$ and a essentially unique family of probability measures $\{\mu^{U,\Fu}_p\}_{p\in U'}$, where $\mu^{U,\Fu}_p$ is supported on $\Wuloc{p}\cap U$, the plaque containing $p$, such that
\[
\mu|_U=\int_U \mu^{U,\Fu}_p\der\mu(p).
\]
We say that $\mu$ has \emph{absolutely continuous disintegration along $\Fu$} if there exists a covering $\{U_1,\cdots, U_k\}$ of $M$ so that for $\mu$ almost every $p\in M$, with $p\in U_i$, it holds that $\mu^{U_i,\Fu}_p\ll\Leb^u_p$, where $\Leb^u_p$ is the Lebesgue measure on $\Wu{p}$. This condition does not depend on the chosen covering.

\begin{definition}[$u$-Gibbs measures]\label{def:ugibbs}
An invariant probability measure $\mu$ for $f$ is a $u-$Gibbs measure if it has absolutely continuous disintegration along $\Fu_f$.
\end{definition}

We denote by $\Gibbs$ the set of $u-$Gibbs measures of $f$. The set of $u$-Gibbs measure is convex and compact. Moreover, the ergodic decomposition of any $u$-Gibbs measure is formed by $u$-Gibbs measures.  Such measures always exist. It is also known that for a $C^2$ (or $C^{1+\alpha}$) partially hyperbolic diffeomorphism $f$, there exists a set $X$ of full volume such that for any $x\in X$, any invariant measure $\mu$ which is a limit measure of the sequence 
\[
\frac{1}{n} \sum_{j=0}^{n-1} \delta_{f^j(x)},
\]
is a $u$-Gibbs measure (see Theorem 11.16 in \cite{Beyond}).
 So, in order to understand existence of physical  measures, it is natural to study such measures.

From now on, let us assume that all of the partially hyperbolic systems we are considering have two dimensional center.  Recall that, given an ergodic measure $\mu$, there are numbers $\lambda^{-c}(\mu) \leq \lambda^{+c}(\mu)$, which are the bottom and top center Lyapunov exponents, respectively. If we assume that $\lambda^{-c}(\mu) < \lambda^{+c}(\mu)$, then,  for $\mu$-almost every $p$, there is a decomposition $E^c_p = E^{-c}(p) \oplus E^{+c}_p$ such that $D_pf(E^*_p) = E^*_{f(p)}$ and $ \lim_{n\to \pm \infty} \frac{1}{n} \log \|D_pf^n|_{E^*}\| = \lambda^*(\mu),$ for $* = -c,+c$, by the Oseledets Theorem.

Observe that the behavior in the center does not play a role for a measure to be $u$-Gibbs. If we also consider possible positive exponents along the center, we obtain the definition of an SRB measure. Let us be more precise.  For any ergodic $f$-invariant measure $\mu$, let us  write 
\[
J^u(\mu) \coloneqq \int \log \left|\det D^u_pf \right| \der\mu(p).
\]
If $\mu$ is an ergodic measure, then $J^u(\mu)$ coincides with the sum of the Lyapunov exponents of $(f,\mu)$ along $E^u$. In the non-ergodic case, it coincides with the integral of the sum of the Lyapunov exponents along $E^u$. Among all of the invariant measures, the following type has a distinguished role.

\begin{definition}[Hyperbolic SRB measures]\label{def:SRB}
An ergodic $f$-invariant measure $\mu$ is a hyperbolic SRB measure if $\lambda^{-c}(\mu)$, $\lambda^{+c}(\mu)\neq 0$ and satisfies 
\[
   h(\mu)=J^u(\mu)+\lambda^{+c}(\mu)\vee 0+\lambda^{-c}(\mu)\vee 0,
\]
where $h_\mu(f)$ is the metric entropy of $\mu$, and $a\vee b$ is the maximum between the numbers $a$ and $b$. 
\end{definition}

The definition is more technical than the case of $u$-Gibbs measures, and takes into account the possible non-uniform expanding directions in the center. The relation between Gibbs and SRBs measures can be elucidated further in our context with the help of Ledrappier--Young's result \cite{LedYoungI}. It turns out that, for a given (say, ergodic) invariant measure $\mu$, there is a lamination consisting of smooth leaves where typical points have an asymptotic expansion bounded from below by the minimum positive Lyapunov exponent. These leaves are called Pesin unstable manifolds. Again it is possible to locally decompose $\mu$ into conditionals along these submanifolds, and thus compare with the corresponding induced Lebesgue measures. A consequence of a theorem of Ledrappier (or, more generally, due to Ledrappier--Young if zero exponents are allowed) implies that a hyperbolic ergodic measure $\mu$ is an SRB measure if and only if its conditional measures along Pesin unstable submanifolds are absolutely continuous. For a more precise discussion the reader is referred to \cite{LedYoungI}, and to the excellent survey article \cite{whatareSRB}.

We remark that for  $f\in \PH$, with $r>1$, the foliation $\Fu$ is absolutely continuous (see for example Chapter 7 in \cite{PesinLect}), thus any SRB is in particular a $u-$Gibbs measure. On the other hand, suppose that $\mu$ is an ergodic $u$-Gibbs measure with at $\lambda^{-c}(\mu) < 0 < \lambda^{+c}(\mu)$:  then, for $\mu$ to be an SRB measure, it must admit absolutely continuous conditional measures along Pesin unstable manifolds, instead of just being absolutely continuous along strong unstable manifolds. Constructing examples of ergodic $u$-Gibbs measures that do not have the SRB property is not difficult. For example, consider $f,g:\Tor^2 \to \Tor^2$ an Anosov diffeomorphisms such that 
\[
\sup_{x\in \Tor^2} \|D_xg\| < \inf_{y\in \Tor^2} m(D_yf|_{E^u_f}),
\]
and 
\[
\sup_{y\in \Tor^2} m(D_yf|_{E^s_f}\| < \inf_{x\in \Tor^2} m(D_xg).
\]
Then, $f\times g$ is partially hyperbolic. Let $\mu_f$ be the unique SRB measure for $f$ and suppose that $p$ is a fixed point for $g$. Then
\[
\mu = \mu_f \otimes\delta_p
\]
is a $u$-Gibbs measure but it is not SRB for $f\times g$. 

\begin{remark}
A consequence of the absolute continuity of the Pesin stable lamination  is that if $\mu$ is a hyperbolic SRB measure then it is also a physical measure. See Theorem 3 in \cite{whatareSRB}.
\end{remark}

\section{Some basic consequences of the expanding on average property}
\label{sec:Somebasicconsequences}

Let $f\in \PHb[2]$ be an expanding on average system in the sense of \Cref{def:AUEP}. Since the constant $C$ in \eqref{eq.inequalityea} need not be unique, we define the optimal expanding on average constant of $f$ by
\begin{equation}\label{eq.optimalconstant}
    \CEA \coloneqq \sup\left\{C>0 : \eqref{eq.inequalityea} \text{ holds with constant } C\right\}.
\end{equation}
In practice $\CEA$ will be given by the infimum of the top center Lyapunov exponents among $u$-Gibbs measures.

 It is not surprising that if $f$ is expanding on average, then each of its $u$-Gibbs measures has at least one positive center exponent (\Cref{cor:Gibbshyperbolic}). To control the other center exponent it is convenient to introduce the quantity
\begin{equation}\label{def:Gap}
   \Gap \coloneqq \CEA-\sup_{p\in M} \limsup_{n\to\oo} \frac{1}{n}\log|\det D_p^cf^n|,
\end{equation}
which we will call \emph{exponent gap}. If $\Gap>0$, then every $\mu\in\Gibbs$ is hyperbolic. We remark that the case where both exponents are positive for Lebesgue almost every point is already present in the literature (the so called mostly expanding case). See \cite{ANDERSSON_2017}, and references therein. Clearly, if $f$ is expanding on average and $\sup_{p\in M}|\det D_p^cf|\leq 1$ (the conservative and dissipative cases along center) then $\Gap>0$.

\begin{notation}
  Recall that $\Leb_p^u$ denotes the induced Lebesgue measure on $\Wu{p}$. Often, we omit the reference to $p$ since it is clear from the context.
\end{notation}

For $p\in M, N\geq 1$ and $\eta>0$, set
\begin{align*}
&B_{N,\eta}(p)=\{q\in\Wu{p}:\frac1n\log\norm{D^c_qf^n}<\CEA -\eta, \forall n\geq N\}\\
&B_\eta(p)=\bigcup_{N\geq 1} B_{N,\eta}(p).
\end{align*}

\begin{proposition}\label{pro:weakexpansion}
Let $f\in \PHb[2]$ be expanding on average. For any $p\in M$ and $\eta>0$ it holds $\Leb^u(B_\eta(p))=0$.   
\end{proposition}

Given $r>0$ and a point $p\in M$, write $\Wu{p,r}$ to be the disk inside $\Wu{p}$ with radius $r$. The following is a consequence of the $C^2$ regularity of the map, which in turn implies bounded distortion along unstable leaves.

\begin{lemma}\label{lem:densitypt}
   Let $f\in \PH[2]$ be a partially hyperbolic system. Consider $p\in M$ and $r>0$. If $K \subset \Wu{p}$ is such that $\Leb^u(K)>0$, then there exist $(p_n)_{n\geq 0}$ so that
   $p_n\in f^n(K)$ and 
   \[
   \lim_{n\to\oo} \frac{\Leb^u(f^n(K) \cap \Wu{p_n,r})}{\Leb^u(\Wu{p_n,r})}=1.
   \]
\end{lemma}

See for example Proposition $7.8$ \cite{Viana1999}.

\smallskip

\begin{proof}[Proof of \Cref{pro:weakexpansion}]

  Suppose the conclusion of \Cref{pro:weakexpansion} does not hold. So, there exist $p\in M$, $\eta>0$ and $N\in \mathbb{N}$ such that $\tilde B:= B_{N,\eta}(p)$ has positive $\Leb^u$-measure. We can suppose that $\eta$ is small enough so that 
  \[
  \CEA - \frac{\eta}{2} >0. 
  \]  
  For $0<\epsilon<\frac{\eta}{2\log\norm{D^cf}_{\oo}}$ and $r>0$, by \Cref{lem:densitypt}, we can take $k\in\Nat$ large and $q\in f^k(\tilde B)$ so that
\[
   \frac{\Leb^u(f^k(\tilde B) \cap \Wu{q,r})}{\Leb^u(\Wu{q,r})}\geq 1-\epsilon.
\]
Note that for $x\in \tilde B, n\geq N$,
\begin{align*}
\frac{1}{n}\log\norm{D^c_{f^k(x)}f^{n}}&=\frac{1}{n}\log\norm{D^c_xf^{n+k} \circ \left(D^c_xf^k\right)^{-1}}\leq \frac{1}{n}\log\norm{D^c_xf^{n+k}}-\frac{1}{n}\log\conorm{D^c_xf^k}\\
&\leq \CEA-\eta-\frac{1}{n}\log\conorm{D^c_xf^k}.
\end{align*}
Since $k$ is fixed, for some $N_1\geq N$, we can guarantee that $n\geq N_1$ implies
\[
   \frac{1}{n}\log\norm{D^c_{f^k(x)}f^{n}}<\CEA-\frac{\eta}{2}.
\]
This in turn implies that if $v\in E^c_q, D=\Wu{q,r}$ and $n\geq N_1$,

\begin{align*}
&\frac{1}{n\Leb^u(D)}
\int_D
\log \norm{D^c_{q'}f^n\cdot \chi_{q',q}(v)}
\der\Leb^u(q')
\\
&\quad =
\frac{1}{n\Leb^u(D)}
\left(
\int_{D\cap f^k(\tilde B)}
+
\int_{D\setminus f^k(\tilde B)}
\right)
\log \norm{D^c_{q'}f^n\cdot \chi_{q',q}(v)}
\der\Leb^u(q')
\\
&\quad \leq
\frac{\Leb^u(D\cap f^k(\tilde B))}{\Leb^u(D)}
\left(\CEA-\frac{\eta}{2}\right)
+
\frac{\Leb^u(D\setminus f^k(\tilde B))}{\Leb^u(D)}
\log\norm{D^cf}_{\oo}
\\
&\quad <
\CEA-\frac{\eta}{2}
+\epsilon\log\norm{D^cf}_{\oo}
<
\CEA.
\end{align*}

The above inequality contradicts the fact that 
\[
   \liminf_{n\to\oo} \frac{1}{n\Leb^u(D)} \int_{D} \log \norm{D^c_{q'}f^n\cdot \chi_{q',q}(v)} \der\Leb^u(q')\geq\CEA,
\]
which is the condition given by the expanding on average property. Thus,  $\Leb^u(B_\eta(p))=0$ as we wanted to show.
\end{proof}
Observe that if $0< \eta_1\leq \eta_2$, then $B_{\eta_2}(p) \subseteq B_{\eta_1}(p)$. In particular, \Cref{pro:weakexpansion} implies that for every $p \in M$, the set
\[
\displaystyle B(p) := \bigcup_{k\in \Nat} B_{\frac{1}{k}}(p)
\]
has zero $\Leb^u$-measure. This gives the following lemma.

\begin{lemma}\label{lem:expansioncontraction}
Let $f\in \PHb[2]$ be expanding on average with $\Gap>0$. For every $p\in M$, $\Leb^u$-almost every $q\in\Wu{p}$ it holds
\begin{align}
&\limsup_{n\to\oo} \frac{1}{n}\log\norm{D^c_qf^n}\geq \CEA\\
&\liminf_{n\to\oo} \frac{1}{n}\log\conorm{D^c_qf^n}\leq  -\Gap.
\end{align}
\end{lemma}

In \Cref{lem:expansioncontraction2}, the estimates above are strengthened. An immediate consequence of \Cref{lem:expansioncontraction} is the following.

\begin{corollary}\label{cor:Gibbshyperbolic}
   Let $f\in \PHb[2]$ be expanding on average with $\Gap>0$. If $\mu\in\Gibbs$, then $\lambda^{+c}(\mu)\geq \CEA$ and $\lambda^{-c}(\mu)\leq -\Gap<0$. In particular, every ergodic $u$-Gibbs measure of $f$ is a hyperbolic measure, with exponents uniformly away from zero.
\end{corollary}

\begin{proof}
    Let $\mu \in \Gibbs$ be an ergodic $u$-Gibbs measure.  Then, 
\begin{align}
\lambda^{+c}(\mu)=&\lim_{n\to\oo} \frac{1}{n}\log\norm{D^c_qf^n}\geq \CEA\\
\lambda^{-c}(\mu)=&\lim_{n\to\oo} \frac{1}{n}\log\conorm{D^c_qf^n}\leq  -\Gap.
\end{align}
Since conditional measures along unstables are equivalent to the $\Leb^u$ on the corresponding leaf. The result is then a direct consequence of \Cref{lem:expansioncontraction}. If $\mu$ is not ergodic, the result follows by working with ergodic components.    
\end{proof}

\section{Projective \texorpdfstring{$u-$}{u}Gibbs measures} \label{section.projectiveugibbs}

In this section we will study the dynamics on the projective center bundle. We will also define and study the notion of projective $u$-Gibbs measures.

\subsection{Linear holonomies and projective dynamics}
\label{subsection.projectivedynamics}

Let $f\in \PHb[2]$, consider $\hM=\mathbb{P}E^c$, the projectivization of $E^c$, and denote $\hpi: \hM \to M$ the natural projection; a point $\widehat{p} \in \hM$ is given by a pair $(p,[v])$ where $p\in M$ and $[v] \in \mathbb{P}E^c(p)$. There is a natural action induced by $f$ on $\hM$ given by $\hf(p,[v]) = (f(p), [ D_pfv])$.  Since the center bundle in principle is only Hölder continuous, $\hM$ is  a Hölder continuous bundle and $\hf$ is a Hölder continuous map. On the other hand, $\hf$ carries some extra structure and morally looks like a partially hyperbolic system.

By \Cref{thm:saghin}, there exist $u$-holonomy $\hu$. Given $\hat{p} \in \hM$, define 
\[
\hWu{\hat p} \coloneqq \{\hu_{q,p}(\hat p):q\sim_u p\},
\]
where $p = \hpi(\hp).$
We call this set the \emph{projective unstable manifold} of $\hat{p}$. Note that $\hWu{\hat p}$ is a Hölder manifold of topological dimension $\dim E^u$, and $\hpi|_{\hWu{\hat p}}:\hWu{\hat p}\to\Wu{\hpi(\hat p)}$ is a  bi-Hölder homeomorphism.  These projective unstable manifolds give rise to a partition $\widehat{\mathcal{W}}^u$ of $\widehat{M}$, which we call the \emph{projective unstable foliation}. Below we investigate the structure of  $\widehat{\mathcal{W}}^u$.

\paragraph{Foliation boxes.} 
Let us define foliation boxes for the projective unstable foliation. By a regular box we mean an open set $U \subset M$ which is the domain of a foliation chart $\phi:U\to\Real^{\dim(M)}$ of $\Fu$, with
$\phi(U)=(-1,1)^{\dim(E^u)}\times (-1,1)^{\dim(E^{cs})}$. Write $D^{cs}_{p_0} = \phi^{-1}\left(\{0\} \times (-1,1)^{\dim(E^{cs})} \right)$, where $p_0 = \phi^{-1}(0,0)$. $D^{cs}_{p_0}$ is a transversal to the foliation $\Fu$.  Observe that $U$ can be written as a disjoint union of unstable disks $\{D^u_p\}$, for $p\in D^{cs}_{p_0}$ and $D^u_p = \phi^{-1}\left((-1,1)^{\dim(E^u)} \times t\right)$, where $t$ is the $cs$-coordinate of $\phi(p)$.

This induces natural coordinates: if $q\in U$ then we can write $q = (q^u, q^{cs})$, where $q^{cs} \in D^{cs}_{p_0}$ and $q^u \in D^u_{q^{cs}}$. Along $D^{cs}_{p_0}$, fix a trivialization of $\mathbb{P}E^c$, that is, there exists a homeomorphism $\tau_{p_0}:\mathbb{P}E^c|_{D^{cs}_{p_0}} \to D^{cs}_{p_0} \times \mathbb{P}E^c(p_0)$. Given $(p,[v]) \in \mathbb{P}E^c|_{D^{cs}_{p_0}}$ write $\tau_{p_0}^2(p,[v])$ to be the coordinate in $\mathbb{P}E^c(p_0)$ of $\tau_{p_0}(p,[v])$.

Given a point $\hat{q} = (q,[v])= ((q^u, q^{cs}), [v]) \in \hpi^{-1}(U)$, we consider the point $q'\in D^{cs}_{p_0}$ such that $q \in D^u_{q'}$. Set 
\[
[v_q] = \left[\tau_{p_0}^2 (q', [h^u_{q',q}(v)])\right],
\]
which belongs to $\mathbb{P}E^c(p_0)$.In this way, we can identify $(q,[v])$ with  
\[
(\phi^{-1}(q^u, q^{cs}), [v_q]) \in (-1,1)^{\dim(E^u)} \times (-1,1)^{\dim(E^{cs})}\times \mathbb{P}E^c_{p_0}.\]
In other words, the set $\widehat{U} = \widehat{\pi}^{-1}(U)$ is homeomorphic to $(-1,1)^{\dim(E^u)} \times (-1,1)^{\dim(E^{cs})} \times \mathbb{P}E^c(p_0)$.  We say that such $\widehat{U}$ is  a regular box for the lift, and the disks $D^{u}_{\hp}$ obtained lifting corresponding $u-$disks in $U$ will be referred as regular $u$-disks. Using this contruction one gets directly that \((\hM, \widehat{\mathcal{W}}^u)\) is a foliated manifold in the sense of \cite{Moore1988}.

\begin{figure}[H]
\centering
\begin{tikzpicture}[
scale=0.8,
>=Latex,
point/.style={
circle,
fill=black,
inner sep=1.6pt
},
visible point/.style={
circle,
fill=black,
draw=white,
line width=0.45pt,
inner sep=1.8pt
},
holonomy/.style={
->,
very thick,
orange!80!black
},
unstable/.style={
blue!65!black,
thick
},
relevant unstable/.style={
blue!75!black,
very thick
},
transversal/.style={
purple!70!black,
thick,
dashed
},
soft label/.style={
fill=blue!3,
inner sep=1.2pt
},
every node/.style={
font=\small
}
]

\begin{scope}
\draw[
rounded corners=5pt,
thick,
fill=blue!3,
draw=black!70
]
(0,0) rectangle (6,4.5);

\node[font=\normalsize] at (3,4.85)
{$U\simeq D^u_{p_0}\times D^{cs}_{p_0}$};

\draw[unstable]
(0.15,0.65)
.. controls (1.55,0.47) and (4.45,0.83) ..
(5.85,0.65);

\draw[relevant unstable]
(0.15,1.45)
.. controls (1.55,1.27) and (4.45,1.63) ..
(5.85,1.45);

\draw[unstable]
(0.15,2.25)
.. controls (1.55,2.07) and (4.45,2.43) ..
(5.85,2.25);

\draw[unstable]
(0.15,3.05)
.. controls (1.55,2.87) and (4.45,3.23) ..
(5.85,3.05);

\draw[relevant unstable]
(0.15,3.65)
.. controls (1.55,3.47) and (4.45,3.83) ..
(5.85,3.65);

\draw[transversal]
(3,0.25)
.. controls (2.82,1.35) and (3.18,3.10) ..
(3,4.20);

\node[
purple!70!black,
rotate=90,
soft label
] at (3.55,2.35)
{$D^{cs}_{p_0}$};

\coordinate (p0)     at (3,1.45);
\coordinate (qprime) at (3,3.65);
\coordinate (q)      at (5.15,3.65);
\coordinate (qu)     at (5.15,1.45);

\node[point] at (q) {};
\node[above=3pt of q]
{$q$};

\node[point] at (qu) {};
\node[below=3pt of qu]
{$q^u$};

\draw[
densely dotted,
gray!75,
thick
]
(q) -- (qu);

\draw[
holonomy,
shorten >=4pt,
shorten <=4pt
]
([yshift=3pt]q)
to[bend right=18]
node[midway,below=6pt]
{$h^u_{q,q'}$}
([yshift=3pt]qprime);

\draw[
holonomy,
shorten >=4pt,
shorten <=4pt
]
([xshift=-4pt]qprime)
to[bend right=14]
([xshift=-4pt]p0);

\node[visible point] at (p0) {};
\node[below left=2pt of p0]
{$p_0$};

\node[visible point] at (qprime) {};
\node[above left=2pt and 4pt of qprime]
{$q'=q^{cs}$};

\node[
blue!75!black,
anchor=west,
soft label
] at (0.35,1.76)
{$D^u_{p_0}$};


\node[align=center] at (3,-0.45)
{(a) Coordinates and holonomy transport in \(U\)};
\end{scope}

\begin{scope}[xshift=0.55cm]
\coordinate (O)   at (8.0,0.55);
\coordinate (Uu)  at (12.0,0.55);
\coordinate (C)   at (9.20,1.90);
\coordinate (P)   at (8.0,3.55);

\coordinate (UC)  at (13.20,1.90);
\coordinate (UP)  at (12.0,3.55);
\coordinate (CP)  at (9.20,4.90);
\coordinate (UCP) at (13.20,4.90);

\fill[blue!5]
(O) -- (Uu) -- (UC) -- (C) -- cycle;

\fill[purple!5]
(O) -- (C) -- (CP) -- (P) -- cycle;

\fill[orange!5]
(P) -- (UP) -- (UCP) -- (CP) -- cycle;

\draw[gray!55,dashed]
(C) -- (UC);

\draw[gray!55,dashed]
(C) -- (CP);

\draw[gray!55,dashed]
(UC) -- (UCP);

\draw[thick,draw=black!75]
(O) -- (Uu) -- (UC);

\draw[thick,draw=black!75]
(O) -- (C);

\draw[thick,draw=black!75]
(O) -- (P);

\draw[thick,draw=black!75]
(Uu) -- (UP);

\draw[thick,draw=black!75]
(P) -- (UP) -- (UCP) -- (CP) -- (P);

\draw[thick,draw=black!75]
(CP) -- (C);

\draw[thick,draw=black!75]
(UCP) -- (UC);

\coordinate (A) at (8.70,2.95);
\coordinate (B) at (12.70,2.95);

\draw[
red!70!black,
line width=2.2pt
]
(A) -- (B);

\node[
circle,
fill=red!70!black,
inner sep=1.6pt
] (hq) at (11.20,2.95) {};

\node[
red!70!black,
below right=3pt of hq
]
{$\widehat{q}$};

\node[
red!70!black,
anchor=west
] at (13.45,2.82)
{$\widehat{q}=(q^u,q^{cs},[v_q])$};

\node[
red!70!black,
below=5pt
] at ($(A)!0.25!(B)$)
{$D^u_{\widehat{p}}$};

\node[below=4pt] at ($(O)!0.5!(Uu)$)
{$D^u_{p_0}$};

\node[
left=5pt,
align=center
] at ($(O)!0.68!(P)$)
{$\mathbb{P}E^c_{p_0}$};

\node[below right=2pt] at ($(O)!0.58!(C)$)
{$D^{cs}_{p_0}$};

\node[
align=center,
font=\normalsize
] at (10.55,5.35)
{$\widehat{U}
\simeq
D^u_{p_0}
\times D^{cs}_{p_0}
\times \mathbb{P}E^c_{p_0}$};

\node[align=center] at (10.55,-0.45)
{(b) Product structure of the lifted box};
\end{scope}

\end{tikzpicture}

\caption{
A regular box and its induced trivialization in the projectivized
central bundle. The direction at \(q\) is transported first along the
unstable plaque from \(q\) to \(q'\), and then along the
\(cs\)-transversal from \(q'\) to \(p_0\).
}
\label{fig:regular-lifted-box}
\end{figure}
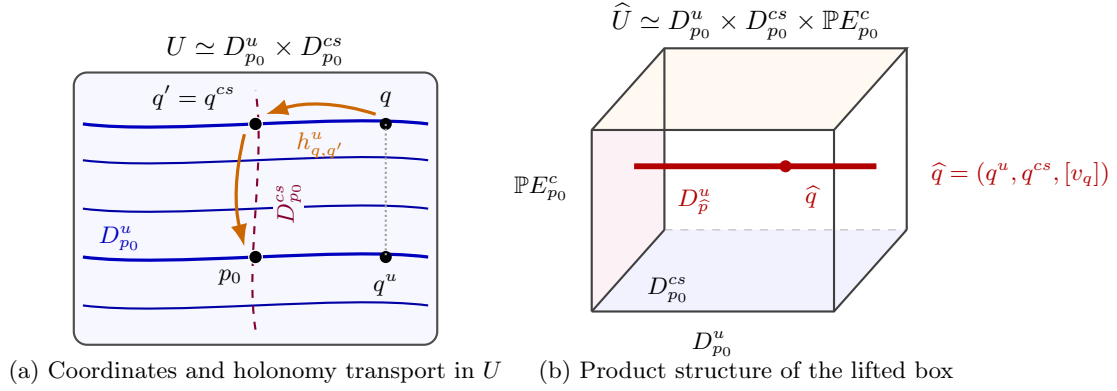


It also follows directly from \Cref{thm:saghin} that $\hpi\circ \hf|_{\hWu{\hp}}=f\circ \hpi|_{\hWu{\hp}}$. Observe that
\begin{itemize}
	\item $\hf$ preserves $\widehat{\mathcal{W}}^u$, this follows from the equivariance of linear holonomies \Cref{def:linearholonomy},and 
	\item $\hf^{-1}$ contracts the intrinsic distances induced by $\hpi$ on \emph{projective unstable manifolds}. 
\end{itemize}

\paragraph*{Projective $u-$Gibbs measure.}\label{page:canonicaldensities} Fix $\mu\in\Gibbs$, and consider a regular box $U$ as in the previous paragraph. Then, for $\mu$-almost every $p$, the conditional measure $\mu_D^{u}=\mu_{D,U}^u$ along $D=D_p$ is of the form 
\begin{align}\label{eq:canonicaldensity}
\der \mu_D^u(t)=\rho^U(p,t)\der\Leb_p^u(t)
\end{align}
where
\[
  \rho^{u}_D(q)=\frac{\Delta(p,q)}{\int_{D} \Delta(p,q)\der \Leb^u_p(q)}
\]
and
\begin{align}\label{eq:Delta}
\Delta(p,q)=\prod_{k=1}^{\oo}\frac{|\det D_{f^{-k}(p)}^uf|}{|\det D_{f^{-k}(q)}^uf|}=\prod_{k=0}^{\oo}\frac{|\det D_{f^{-k}(t)}^uf^{-1}|}{|\det D_{f^{-k}(p)}^uf^{-1}|}.
\end{align}
See, for example, \cite[Corollary $6.1.4$]{LedYoungI}. Although the disintegration is defined for $\mu$-typical leaves, the formula $\rho_D^u \der \Leb^u_p(t)$ is well defined for every $p$. If follows that for any $u$-disk $D$ containing $p$ (or more generally $D$ a set in $\Wu{p}$ with positive Lebesgue measure bounded from above), we can define a probability measure $\mu_D^u$ using the expected density of the $u$-Gibbs states, that is, by formula \eqref{eq:canonicaldensity}.  

One can show that $\log \rho^{u}_D$ is Lipschitz on $D$, bounded away from zero and infinity, with constants that only depend on $U$.

\begin{lemma}\label{lem:regularrho}
Given a regular box $U$, the function $\rho^u:\{(p,t):p\in U, t\in D^u_p\}\to\Real$ is continuous.
\end{lemma}

\begin{proof}
It suffices to show that $\Delta(p,t)$ is continuous. This follows from Proposition~2 in \cite{GibbsSinaiPesin}. \qedhere
\end{proof}

For $\hat{p} \in \widehat{M}$ we can use $\hpi|_{\hWu{\hat p}}$ to lift the Lebesgue measure on $\Wu{p}$ and define
\[
	\widehat{m}^u_{\hat{p}} \coloneqq \left(\left(\widehat{\pi}|_{\widehat{W}^u(\hat{p})}\right)^{-1}\right)_*\mathrm{Leb}^u_{\widehat{\pi}(\hat{p})}.
\]
This gives us a continuous family of measures 
\begin{align}\label{eq:Muprojective}
\mathcal{M}^u \coloneqq \{\widehat{m}^u_{\hat{p}} \}_{\hat{p} \in \widehat{M}}
\end{align}

\begin{definition}[Projective $u$-Gibbs]\label{def:projectiveuGibbs}
We say that an $\hat{f}$-invariant probability measure $\hmu$ in $\widehat{M}$ is a projective $u$-Gibbs measure for $\hat f$ if it has absolutely continuous disintegration with respect to the family $\mathcal{M}^u$, that is, along projective $u$-dics, the conditional measures are absolutely continuous with respect to $\widehat{m}^u$. 

The set of $u$-Gibbs measures for $\hf$ is denoted by $\Gibbs[\hf]$.
\end{definition}

Take any unstable disk $D^u_p$, centered at $p\in M$. For any point $\hat p=(p,[v]) \in \widehat{M}$, we may consider the lift of $D^u_p$ as $\widehat{D}^u_{\hat{p}} \coloneqq \{(q,[h^u_{q,p} v]): q\in D^u_p\}$.  If $\hE \subset \widehat{D}^u_{\hat{p}}$ has positive $\widehat{m}^u_{\hat{p}}$-measure, define 
\[
\widehat{\nu}_{\hat{p},\hE}^u \coloneqq \frac{\widehat{m}^u_{\hat{p}}\left(\cdot \cap \hE \right)}{\widehat{m}^u_{\hat{p}} \left(\hE\right)}.
\]
The following lemma states that Pesin--Sinai's construction also works to build projective $u$-Gibbs measures. 

\begin{lemma}\label{lem:projectivepesinsinai}
Using the previous notation, it holds that any accumulation point of 
\[
\displaystyle \widehat{\nu}_n \coloneqq \frac{1}{n} \sum_{j=0}^{n-1} \hat{f}^j_* \widehat{\nu}^u_{\hat{p},\hE}.
\]
is a projective $u$-Gibbs measure.
\end{lemma}
The proof of \Cref{lem:projectivepesinsinai} is based on Pesin--Sinai's constructions of $u$-Gibbs for partially hyperbolic diffeomorphisms. The key ingredient in the construction is bounded distortion. In this projective setting, this is given by the following lemma.

\begin{lemma}\label{lem:bdddistortion}
    For any $L>0$ there exists a constant $K\geq 1$ such that for any $\hp\in \hM$, any set $\hE\subset \hWu{\hp,L}$ of positive $\widehat{m}^u_{\hp}$-measure , where $\hWu{\hp,L}$ is the subset of $\hWu{\hp}$ that projects by $\hpi$ to $\Wu{p,L}$, for any $n\geq 1$,
    \[
   \frac{\widehat{m}^u_{\hp} (\hE)}{\widehat{m}^u_{\hp} (\hWu{\hp,L})} \frac{1}{K} \leq \frac{\widehat{m}^u_{\hf^{-n}(\hp)}(\hf^{-n} (\hE))}{\widehat{m}^u_{\hf^{-n}(\hp)}(\hf^{-n}(\hWu{\hp,L})} \leq K\frac{\widehat{m}^u_{\hp} (\hE)}{\widehat{m}^u_{\hp} (\hWu{\hp,L})}.
    \]
\end{lemma}
\begin{proof}
    It is well known that for any $L>0$, there exists $K\geq 1$ such that for any $p\in M$ and any $E\subset \Wu{p,L}$ having positive $\Leb^u_p$-measure, for any $n\geq 1$, we have that
    \[
  \frac{\Leb^u_{p}(E)}{\Leb^u_{p}(\Wu{p,L})}  \frac{1}{K} \leq \frac{\Leb^u_{f^{-n}(p)}(f^{-n}(E))}{\Leb^u_{f^{-n}(p)}(f^{-n}(\Wu{p,L})} \leq K  \frac{\Leb^u_{p}(E)}{\Leb^u_{p}(\Wu{p,L})}.
    \]
    This is the usual bounded distortion for unstable manifolds of partially hyperbolic systems (see for example Corollary 7.7 in \cite{Viana1999}).
    \Cref{lem:bdddistortion} then follows from the definition of $\widehat{m}^u_p$ and applying the bounded distortion above for $E = \hpi(\hE)$. 
\end{proof}
\begin{proof}[Proof of \Cref{lem:projectivepesinsinai}]

In the construction of $u$-Gibbs measures by Pesin--Sinai, the bounded distortion estimate is the key ingredient to show the desired absolute continuity of the conditional measures. The same proof, with \Cref{lem:bdddistortion} implies that any limit measure of the sequence $\widehat{\nu}_n$ is a projective $u$-Gibbs measure. See \cite{Viana1999} for the details of the referred proof. \qedhere

\end{proof}

\begin{notation}
  We denote by $\pr_{\hf}(\hM)$ the set of $\hf$-invariant probability measures. It is equipped with the weak* topology.
\end{notation}

\begin{lemma}\label{lem:compactGibbs}
The set $\Gibbs[\hf] \subset \pr_{\hf}(\hM)$ is closed (hence compact), and convex. Moreover, for a projective $u$-Gibbs measure, its ergodic components are projective $u$-Gibbs measures.
\end{lemma}

\begin{proof}
Assume that a sequence of projective $u$-Gibbs measures $(\hmu_n)_{n\in \mathbb{N}}$ converges to some $\hmu$ and fix a continuous function $\psi: \widehat{M} \to \mathbb{R}$, and a regular box $U$ such that $\mu(\partial U) =0$, where $\mu = \hpi_* \hmu$. Let $\hpi: \hM \to M$ be the natural projection and set $\hat U=\hpi^{-1}(U)$.  Due to \Cref{lem:regularrho}, there exists some continuous function $\hat\rho:\{(\hp,\hat t): \hp\in \hat U, \hat t\in \widehat{D}^u_{\hp}\}\to\Real_{\geq 0}$ so that for every $n\in \mathbb{N}$, 
\[
\int_{\hat U} \psi(\hp)\der \hmu_n(\hp)=\int_{\hat U} \left(\int_{\widehat{D}^u_{\hp}} \psi(\hat t)\hat\rho(\hat p,\hat t)\der \hm^u_{\hp}(\hat{t})\right)\der \hmu_n(\hp).
\]
By the continuity of the holonomies $h^u_{p,q}$, we have that the term inside the parentheses defines a continuous function of $\hat p$. Therefore, by the weak* convergence we get
\[
\int_{\hat U} \psi(\hp)\der \hmu(\hp)=\int_{\hat U} \left(\int_{\widehat{D}^u_{\hp}} \psi(\hat t)\hat\rho(\hat p,\hat t)\der \hm^u_{\hp}(\hat{t})\right)\der \hmu(\hp).
\]
Since this is true for every continuous function $\psi: \widehat{M} \to \mathbb{R}$,  uniqueness of the conditionals implies that the disintegration along $\hat{\mathcal{W}}^u|_{\hat U}$ is given by the family $\{\rho^{\hat U}(\hat p,\hat t)\der \hm^u_{\hp}(\hat t)\}$, hence $\hmu\in\Gibbs[\hf]$.

Convexity is immediate, and for the last part it suffices to show that $\hmu\in\Gibbs[\hf]$ is an extremal point if and only if it is ergodic. From this and the Krein--Millman Theorem one gets the existence of ergodic elements in $\hmu\in\Gibbs[\hf]$.

Since $\hf$ is continuous, it is well known that the extremal points of $\pr_{\hf}(\hM)$ are precisely the ergodic measures, which in particular shows that ergodicity implies extremality in $\Gibbs[\hf]$. Thus, it remains to show that ergodic components of $\hmu$ are projective $u-$Gibbs measures.

Suppose that $\hmu\in \Gibbs[\hf]$ and define  
 \[
 B(\hmu)=\left\{\hat p: \lim_{n\to +\infty}\frac{1}{n} \sum_{j=0}^{n-1}\delta_{\hf^j(\hp)}=\lim_{n\to + \infty}\frac{1}{n} \sum_{j=0}^{n-1}\delta_{\hf^{-j}(\hp)}\right\}.
 \]
By Birkhoff's ergodic theorem $\hmu(B(\hmu))=1$, therefore for $\hmu-$almost every $\hp\in \hM$, there exists a regular disk $D_{\hp}^u$ which verifies 
$\hm_{\hp}(B(\hmu)\cap \hat{D}^u_p)>0$. Note also that since $\hf^{-1}$ is locally contracting, the empirical measures of every $\hat q\in \hE=B(\hmu)\cap \hat{D}^u_p$ converge to the same measure, say $\hmu_{\hp}$; then
\begin{align*}
\lim_{n\to\oo}\frac{1}{n\widehat{m}^u_{\hat{p}}(\hE)}\int_{\hE} \sum_{j=0}^{n-1}\delta_{\hf^j\tilde{q}}\der \widehat{m}^u_{\hat{p}}(\hat q)=\hmu_{\hp},
\end{align*}
which by \Cref{lem:projectivepesinsinai} shows that $\hmu_{\hp}$ is a projective $u-$Gibbs measure. On the support of an ergodic component, empirical measures converge to the component, due to Birkhoff theorem, hence the result. 
\end{proof}

The proof of \Cref{lem:compactGibbs} gives something stronger.

\begin{lemma}\label{lem:closednessGibbs}
    Suppose $\left(f_n\right)_{n\in \Nat}$ is a sequence of diffeomorphisms in $\PHb[2]$ converging to $f\in \PHb[2]$ in the $C^2$-topology, and for each $n\in \Nat$  consider $\hmu_n \in \Gibbs[\hf_n]$. Suppose that $\hmu_n$ converges to some measure $\hmu$. Then $\hmu \in \Gibbs[\hf]$.
\end{lemma}
\begin{remark}\hfill
\begin{itemize}
    \item Observe that for each $n$, the bundles $\mathbb{P}E^c_{f_n}$ are different, so the measures $\hmu_n$ are measures in different spaces. However, all of these bundles are subsets of the same bundle $\mathbb{P}TM$, so for the convergence in $\Cref{lem:closednessGibbs}$ we can think about the weak* topology for measures in $\mathbb{P}TM$.
    \item The $C^2$ convergence in the statement of \Cref{lem:closednessGibbs} can be replaced by $C^1$-convergence plus bounded $C^2$-norm of the sequence, which is a weaker condition.
    \item \Cref{lem:closednessGibbs} is equivalent to upper semicontinuity of the set-valued map $g\mapsto \Gibbs[\hat g]$, with respect to the $C^2$-topology.
\end{itemize}
\end{remark}
\begin{proof}
First, a probability measure $\hnu$ in $\mathbb{P}E^c_f$ is also a probability measure on $\mathbb{P}TM$. Let $\hmu_n$ be a sequence of projective $u$-Gibbs measures for $\hat{f}_n$ converging to $\hmu$ as above. 
Observe that $\hf_n = \mathbb{P}Df|_{\mathbb{P}E^c}$, where $\mathbb{P} Df$ is the projective action of $f$. Since $\mathbb{P}Df_*\hmu_n = (\hf_n)_* \hmu_n = \hmu_n$, and $\mathbb{P} Df_n \to \mathbb{P}Df$ and $\hmu_n \to \hmu$, we have that $\hf_* \hmu = \hmu$, so $\hmu$ is $\hf$-invariant. 

We take $U$ a foliation box for the unstable foliation of $f$ (which with no loss of generality assume to be properly contained in a slightly larger foliation box), and consider $\widehat{U}=(\hat \pi)^{-1}(U)\simeq D^u_{p_0}\times D^{cs}_{p_0}\times\mathbb{P}E^c_{p_0}$. For every $q^{cs}\in D^{cs}$ we have a parametrization  
\[
  \Phi_{q^{cs}} \colon D^u_{p_0}\to M
\]
of the unstable plaque passing through $q^{cs}$. For all $n$ sufficiently large, the $f_n$-unstable plaque of $U$ through $q^{cs}$ can be parametrized as
\[
\Phi_{n,q^{cs}} \colon D^u_{p_0}\to M:  
\]
moreover, the parametrizations may be chosen so that
\[
\sup_{q^{cs}} \norm{\Phi_{n,q^{cs}}-\Phi_{q^{cs}}}_{C^1}\longrightarrow 0.
\]
Similarly, for every $q^T=(q^{cs}, [v_q])\in D^{cs}_{p_0}\times \mathbb{P}E^c_{p_0}$ we consider parametrizations 
\begin{align*}
\hat \Phi_{q^{T}} \colon D^u_{p_0}\to \mathbb{P}TM\\
\hat \Phi_{n,q^{T}} \colon D^u_{p_0}\to \mathbb{P}TM
\end{align*}
of the corresponding regular $u$-disks passing through $q^T$. If $\mathrm{pr}:\mathbb{P}TM\to M$ denotes the natural projection, we have
\begin{align}\label{eq:prdown}
&\mathrm{pr}\circ \hat \Phi_{q^{T}}=\Phi_{q^{cs}}\\
&\mathrm{pr}\circ \hat \Phi_{n,q^{T}}=\Phi_{n,q^{cs}}.
\end{align}
Observe also that due to \Cref{thm:saghin}, we have
\[
\sup_{q^T} \norm{\hat\Phi_{n,q^T}-\hat\Phi_{q^T}}_{C^0}\longrightarrow 0.
\]
Recall the definition given \eqref{eq:Delta}, and let $\Delta, \Delta_n$ correspond to $f,f_n$. Given $p\in D^{u}_{p_0}$ we write 
\begin{align*}
\rho_{q_T}(p)=\Delta(q^{cs},\Phi_{q^{cs}}(p))\\
\rho_{n,q_T}(p)=\Delta_{n}(q^{cs},\Phi_{n,q^{cs}}(p))
\end{align*}
If 
\[
\varphi_n(x)= \log \abs{\det\left(Df_n(x)|_{E^u_{f_n}(x)}\right)}.
\]
then the Jacobian formula gives
\[
\log \rho_{n,q^T}(p)=\sum_{k=1}^{\infty}\left(\varphi_n(f_n^{-k}(q^{cs}))-\varphi_n(f_n^{-k}(\Phi_{n,q^{cs}}(p))\right).
\]
Since $f_n\to f$ in the \(C^2\) topology, the partial-hyperbolicity constants can be chosen uniformly, the unstable bundles converge
uniformly, and the functions $\varphi_n$ have a common Hölder modulus. Thus, there exist $C>0$, $0<\theta< 1$, and $\alpha>0$, independent of $n$ and $q^{T}$, such that
\[
\abs{\varphi_n(f_n^{-k}q^{cs})-\varphi_n(f_n^{-k}\Phi_{q^{cs}}(p))}\leq C\theta^{\alpha k}.
\]
It follows that the tails of the series defining $\log \rho_{n,q^T}$ converge to zero uniformly in $n$, $q^T$ and $p$. For each fixed $N$, the first $N$ terms converge uniformly, since
\[
f_n^{-k}\bigl(\Phi_{n,q^{cs}}(p)\bigr)\longrightarrow f^{-k}\bigl(\Phi_{q^{cs}}(p)\bigr)
\]
uniformly in $p, q^{cs}$.  Splitting the series into its first $N$ terms and its tail therefore gives
\[
\sup_{q^T}\sup_{p\in D^u_{p_0}} \abs{\rho_{n,q^T}\bigl(\Phi_{n,q^T}(p)\bigr)-\rho_{q^{cs}}\bigl(\Phi_{q^{cs}}(p)\bigr)}\longrightarrow 0.
\]
The $C^1$-convergence of the plaque parametrizations implies uniform convergence of the corresponding leaf-volume densities. Hence
the normalizing constants also converge uniformly, and the normalized plaque probabilities converge uniformly in total variation. Applying the same to the parametrizations of the regular $u$-disks, and using \eqref{eq:prdown}, and arguing similarly as in the proof of \Cref{lem:compactGibbs}, we deduce that conditionals of $\hmu_n$ along regular $u$-disks of $\hat U$ converge to a family of measures which is absolutely continuous with respect to the family $\mathcal M^u$, and $\hmu$ is a projective $u$-Gibbs measure.
\end{proof}

Using this formalization of projective $u$-Gibbs measures, we obtain the following equivalent definition to the expanding on average property.

\begin{lemma}\label{rem:lengthdiscs}
    Let $f\in \PHb[2]$,  then the following are equivalent:
    \begin{enumerate}
        \item  $f$ is expanding on average;
        \item  there exists $C>0$ such that for any $p\in M$, any unit vector $v\in E^c(p)$,
        \[
    \liminf_{n\to\oo}\frac{1}{n\Leb^u(W^u(p,1))}\int_{W^u(p,1)}\log \norm{D_qf^n\chi_{q,p}(v)} \der\Leb^u(q)\geq C.
    \]
    \end{enumerate}

\end{lemma}

\begin{proof}
    \textit{1} implies \textit{2} is immediate. Let us show that \textit{2} implies \textit{1}. Let $\psi: \hM \to \R$ defined by $\psi(x,[v]) = \log \|D_xf|_{[v]}\|$. This is a H\"older continuous function. Set
    \[
    C_* \coloneqq \inf \left\{\int \psi \der\hmu: \hmu \in \Gibbs[\hf]\right\}.
    \]
    For any $u$-disk centered at a point $x$, $D_x$, and any vector $v\in E^c(x)$ unitary, write $\hat{D}_{x,v}:= \{(y,[\chi_{y,x}(v)]): y\in D_x\}$ and observe that this is a projective unstable disk over $D_x$. Write $\hat{m}^u_{x,v,D_x}$ the normalized $\hat m^u_{x,[v]}$ measure restricted to $\hat{D}_{x,v}$. By \Cref{lem:projectivepesinsinai}, any accumulation measure of 
    \[
    \hnu_{n}:= \frac{1}{n} \sum_{j=0}^{n-1} \hf^j_* \hat{m}^u_{x,v,D_x}
    \]
    is a projective $u$-Gibbs measure. Observe that 
    \[
    \frac{1}{n\Leb^u(D_x)} \int_{D_x} \log \|D_yf^n \chi_{y,x}(v) \| \der \Leb^u(y) = \frac{1}{n} \sum_{j=0}^{n-1} \int \psi \der \hf^j_* \hat m^u_{x,v,D_x}(\hat y). 
    \]
    Therefore,
    \[
    \liminf_{n\to+\infty}\frac{1}{n\Leb^u(D_x)} \int_{D_x} \log \|D_yf^n \chi_{y,x}(v) \| \der \Leb^u(y)\geq C_*.
    \]
Since the function $\hmu \mapsto \int \psi \der \hmu$ is continuous, and since the set of projective $u$-Gibbs measures is compact, then there exists $\hmu$ such that $C_* = \int \psi \der \hmu$. By convexity, the minimum has to be achieved in a extremal point of $\Gibbs[\hf]$, so we may suppose that $\hmu$ is ergodic.

Take $(p,[v])$ a $\hmu$-typical point. Then, $\hmu_{\hp}$-almost every point in $\hWu{\hp,1}$ is $\hmu$-generic, hence their Birkhoff's averages converge to $\hmu$. Write $p=\hpi(\hp)$. The dominated convergence theorem then implies 

\begin{multline}
  C\leq \liminf_{n\to+\infty}\frac{1}{n\Leb^u(\Wu{p,1})} \int_{\Wu{p,1}} \log \|D_qf^n \chi_{q,p}(v) \| \der \Leb^u(q)\\
  = \liminf_{n\to +\infty} \frac{1}{n} \sum_{j=0}^{n-1}\psi \circ \hf^j(\hp) = \int \psi \der \hmu = C_*.
\end{multline}
Hence, 
\[
C \leq C_* \leq \liminf_{n\to+\infty}\frac{1}{n\Leb^u(D_x)} \int_{D_x} \log \|D_yf^n \chi_{y,x}(v) \| \der \Leb^u(y).\qedhere
\]
\end{proof}

\subsection{Characterization of the expanding on average property} 
\label{sub:projetive_dynamics_continuity_of_exponents}

In this subsection we will give a useful characterization of the expanding on average property. This condition gives us a better understanding of what expanding on average is. 

Given an invariant measure, recall the definiton of the center Lyapunov exponents \eqref{eq:topexp} and \eqref{eq:botexp}. Due to Oseledets Theorem, given $\mu\in \pr_f(M)$, with $\lambda^{-c}(\mu) < \lambda^{+c}(\mu)$, there exist measurable $Df-$invariant subbundles  $E^{-c}, E^{+c} \subset E^c$, defined on a full $\mu-$measure set $\cR(\mu)$ verifying for $p\in \cR(\mu), v\in E^{c}\setminus\{0\}$,
\begin{align*}
v\in E^{*c}\Rightarrow \lim_{n\to\oo}\frac{1}{n}\log\norm{D_pf^n(v)}=\lam^{*c}(p)\quad \textrm{ for }*\in\{+,-\}.
\end{align*}

\begin{definition}
Let $f\in \PHb[2]$. Given an ergodic $\mu \in \Gibbs$, such that $\lambda^{-c}(\mu) < \lambda^{+c}(\mu)$, we say that $E^{-c}$ is $h^u$-invariant if there is a set $\widetilde{M}$ of full $\mu$-measure such that if $p,q \in \widetilde{M}$, with $q\in \Wu{p}$, then $h^u_{q,p}(E^{-c}(p)) = E^{-c}(q)$. 
\end{definition}

For convenience, we will denote $\mathcal{G} \subset \PHb[2]$ the set of diffeomorphisms $f$ so that
 \begin{enumerate}
   \item $\inf\{\lambda^{+c}(\mu):\mu\in \Gibbs\}>0>\sup\{\lambda^{-c}(\mu):\mu\in \Gibbs\}$, and
   \item for every ergodic $\mu \in \Gibbs$, the direction $E^{-c}$ is not $h^u$-invariant. 
\end{enumerate}

The main result in this subsection is the following.
\begin{maintheorem}\label{mainthm:TheoremD}
    Let $f\in \PHb[2]$. 
    \begin{enumerate}
        \item[\textbf{(A)}] Suppose that $f$ has the expanding on average property and $\Gap>0$, then $f\in \mathcal{G}$.
        \item[\textbf{(B)}] Suppose that $f\in \mathcal{G}$, then $f$ has the expanding on average property. Moreover, the constant $\CEA$ defined in \eqref{eq.optimalconstant} verifies
        \[
        \CEA = \inf \{\lambda^{+c}(\mu): \mu \in \Gibbs\}.
        \]
        If in addition $\sup_p |\det D^c_pf| \leq 1$ then $f$ has the expanding on average property and $\Gap>0$.
    \end{enumerate}
\end{maintheorem}

\begin{remark}
\Cref{mainthm:TheoremD} says that, essentially, the expanding on average property can be seen as the lack of $h^u$-invariance for the $E^{-c}$ direction. In the language of the projective dynamics, this can be understood as that projectively, the only $u$-Gibbs measures are supported on $E^{+c}$ directions. 
\end{remark}

In the rest of this subsection we prove \Cref{mainthm:TheoremD}. Observe that we can think of $\hf$ as a skew product over $f$. In particular, any $\hf$-invariant measure projects to an $f$ invariant measure and, conversely, any $f$-invariant measure can be lifted to at least one $\hf$-invariant measure.  The following is consequence of a well known result due to Ledrappier; see Proposition $5.1$ in \cite{LedraExp}.

\begin{notation}
	If $E \subset E^c(p)$ is one dimensional, we will abuse notation and write $E \in \mathbb PE^c(p)$. 
\end{notation}

\begin{lemma}\label{pro:Ledrappierproj}
Suppose that $\mu\in \pr_f(M)$. If for $\mu$-almost every point $p,$ we have $\lambda^{-c}(p) < \lambda^{+c}(p)$, then,
\begin{enumerate}
 	\item If $\mu$ is ergodic then it has exactly two ergodic lifts $\hmu^{+},\hmu^{-}\in \pr_{\hf}(\hM)$, which can be characterized as:
 	\begin{enumerate}
 		\item \[ \hmu^{-} = \displaystyle \int_{M_0} \delta_{E^{-c}(p)} \der \mu(p);\]
 		
 		the disintegration of $\hmu^-$ along the fibers of $\hM$ are given by $\{\delta_{E^{-c}(p)}\}_{p\in M_0}$, where $\mu(M_0)=1$;
 		\item \[ \hmu^{+} = \displaystyle \int_{M_0} \delta_{E^{+c}(p)} \der\mu(p).\]
 		
 		the disintegration of $\hmu^+$ along the fibers of $\hM$ are given by $\{\delta_{E^{+c}(p)}\}_{p\in M_0}$, where $\mu(M_0)=1$;
 	\end{enumerate}
 	\item If $\mu$ is not ergodic, then for any lift $\hmu\in \pr_{\hf}(\hM)$ there exists $M_0$ of full measure and a measurable $f$-invariant function $\rho:M_0\to [0,1]$ so that
    \[
    \displaystyle \hmu = \int_{M_0} \rho(p)\delta_{E^{-c}(p)}+(1-\rho(p))\delta_{E^{+c}(p)}\der\mu(p), 
    \]
    the disintegration of $\hmu$ along the fibers of $\hM$ is given by $\{\rho(p)\delta_{E^{-c}(p)}+(1-\rho(p))\delta_{E^{+c}(p)}\}_{p\in M_0}$.
 \end{enumerate} 
\end{lemma}

The following is a general result for partially hyperbolic systems with linear holonomies. 

\begin{lemma}\label{lem:invariantEc}
Let $f\in \PHb[2]$  and let $\mu\in \pr_f(M)$ be an invariant measure and suppose that for $\mu$-almost every point, the center Lyapunov exponents are different. Let $p,q\in \cR(\mu)$ with $q\in \Wu{p}$, then   
\[
	 \hu_{q,p}(E^{+c}(p))=E^{+c}(q).
\]
Moreover, if $\mu \in \Gibbs$ is a $u$-Gibbs measure such that for $\mu$-almost every $p$, $\lambda^{-c}(p) < \lambda^{+c}(p)$, then  $\hmu^+$ is a projective $u$-Gibbs measure.
\end{lemma}

\begin{proof}
Without loss of generality, we can suppose that $\mu$ is ergodic. For the general case we just work in ergodic components. For $p,q\in \cR(\mu)$, write $p_n=f^n(p), q_n=f^n(q)$. From the definition of linear holonomies,
\begin{equation}\label{eq:invarianceDhol}
	D_{q}f^{-n}\circ \hu_{q,p}=\hu_{q_{-n},p_{-n}}\circ D_pf^{-n}.
\end{equation}
Observe that the Oseledets directions are well defined in $\cR(\mu)$. Therefore,
\[
\left\|D_{q}f^{-n}\circ \hu_{q,p}|_{E^{+c}(p)}\right\|=\left\|\hu_{q_{-n},p_{-n}}\circ D_pf^{-n}|_{E^{+c}(p)}\right\|.
\]
Since the distance between $p_{-n}$ and $q_{-n}$ goes to zero, we have that $\lim_{n\to +\infty} \|h^u_{q_{-n}, p_{-n}}\|=1 $. Hence,
\begin{align*}
\limsup_{n\to +\infty} \frac{1}{n} \log \left\|\hu_{q_{-n},p_{-n}}\circ D_pf^{-n}|_{E^{+c}(p)}\right\| &\leq \limsup_{n\to +\infty} \frac{1}{n} \left(\log \|h^u_{q_{-n},p_{-n}}\| + \log \|D_pf^{-n}|_{E^{+c}(p)}\|\right) \\
&= -\lambda^{+c}(\mu).
\end{align*}

If $h^u_{q,p}(E^{+c}(p)) \neq E^{+c}(q)$, and since $E^c(q) = E^{-c}(q) \oplus E^{+c}(q)$, we would have that 
\begin{align*}
    \displaystyle  \limsup_{n\to +\infty} \frac{1}{n} \log \left\|\hu_{q_{-n},p_{-n}}\circ D_pf^{-n}|_{E^{+c}(p)}\right\|= \limsup_{n\to + \infty} \frac{1}{n} \log \left \|D_{q}f^{-n}\circ \hu_{q,p}|_{E^{+c}(p)}\right\|\\
     \geq \liminf_{n\to +\infty}\frac{1}{n} \log \left \|D_{q}f^{-n}|_{h^u_{q,p}(E^{+c}(p))}\right\| = -\lambda^{-c}(\mu).
\end{align*}
Since we assumed that $\lambda^{-c}(\mu) < \lambda^{+c}(\mu)$, we obtain a contradiction. 

Take a regular box $U$ in $M$, let $\hat{U}$ its lift in $\hM$: then, using the notation of the previous part,
\[
\hmu^+|_{\hat{U}} = \displaystyle \int_U  \delta_{E^{+c}(p)}\der\mu(p)=\int_T \left( \int_{D^u_p} \delta_{E^{+c}(t)} \rho^U(p,t) \der \Leb^u_p(t) \right) \der\mu_T(p).
\]
where $T=D^{cs}_p$, and $\mu_T$ is the quotient measure. For $\mu_T$-almost every $p$, we can choose a point $q_p \in D^u_p$, making the choice $p \mapsto q_p$ measurable, such that $q_p$ is a regular point. Therefore, $q_p$ has a well defined Oseledets direction $E^{+c}(q_p)$. Using this choice of point as a reference point in each plaque, we can rewrite this, using the $\hu$-invariance of $E^{+c}$, as
\[
\hmu^+|_{\hat{U}} = \displaystyle \int_T \left( \int_{D^u_p} (\hu_{t,q_p})_*\delta_{E^{+c}(q_p)} \rho^U(p,t) \der \Leb^u_p(t) \right) \der\mu_T(p).
\]
If $\hat p=(p,h^u_{p,q_p}(E^{+c}(q_p)))$, then   
\[
\displaystyle \bigcup_{t\in D^u_p} \hu_{t,q_p}(E^{+c}(q_p)) =\hat W^u(\hp, \hat U) \eqqcolon \hat D^u_{\hp}.
\]
Writing $\hat{t} = \hu_{t,q_p}(E^{+c}(q_p))$, we get
\[
\hmu^+|_{\hat{U}} = \displaystyle \int_T \left(\int_{\hat{D}^u_{\hat{p}}} \rho^U(\hat{p}, \hat{t}) d\hat{m}^u_p(t) \right) \der\mu_T(p),
\]
which means that $\hat{\mu}^+$ is a projective $u$-Gibbs measure.  
\end{proof}
We now prove part $(A)$ of \Cref{mainthm:TheoremD}.

\begin{lemma}\label{pro:thmB1}
   Let $f\in \PHb[2]$ be a partially hyperbolic diffeomorphism with the expanding on average property and verifying $\Gap>0$. Then $f\in\mathcal{G}$.
\end{lemma}

\begin{proof}
   We observed in \Cref{cor:Gibbshyperbolic} that the top center exponent of every $u$-Gibbs measure is bounded from below by $\CEA$ and that the bottom center exponent is bounded from above by $-\Gap$. 
   
   Suppose there is an ergodic $u$-Gibbs measure $\mu$ such that $E^{-c}$ is $h^u$-invariant. Therefore, the lift of the measure $\mu$ to its stable Oseledets directions,  $\hmu^-$, is a projective $u$-Gibbs measure.  However, for $\hat{\mu}^-$-almost every $(p,[v])$, if we consider the projective unstable disk
\[
\hat{D}_{(p,[v])} := \{(q,[w]): q\in \Wu{p,1} \textrm{ and } [w] = [h^u_{q,p}(v)]\}
\]
with the measure 
\[
\hat{m}_{\hat{D}_{(p,[v])}} := \frac{1}{\Leb^u(W^u(p,1))}\int_{W^u(p,1)} \delta_{(q,[h^u_{q,p}v])} \der\Leb^u(q),
\]
then its averages
\[
\displaystyle \hat{\mu}_n^-:= \frac{1}{n} \sum_{j=0}^{n-1} \hat{f}^j_* \hat{m}_{\hat{D}_{(p,[v])}}
\]
converge to $\hat{\mu}^-$.  Integrating the function $\psi(p,[v]) = \log \|D_pf (\frac{v}{\norm{v}})\|$, we obtain that
\[
\lambda^{-c}(\mu) = \int \psi d\hat{\mu}^-.
\]
Since this is a continuous function in the projective bundle, we obtain that 
\[
\displaystyle \lim_{n \to +\infty} \int \psi d\hat{\mu}_n^- =  \int \psi \der\hat{\mu}^- = \lambda^{-c}(\mu).
\]
However, by the chain rule, 
\[
\displaystyle \lim_{n\to +\infty} \int \psi d\hat{\mu}_n^- = \lim_{n\to +\infty} \frac{1}{n \Leb^u(W^u(p,1))} \int \log \left\|D^c_q f^n \cdot \chi_{q,p}(v)\right\| \der\Leb^u(q) \geq \CEA
\]
which gives a contradiction. 
\end{proof}

\begin{lemma}\label{cor:onlygibbs}
    If $f\in \mathcal G$ then for every $\mu\in\Gibbs$ its only lift to $\hM$ in $\Gibbs[\hf]$ is the measure $\hmu^{+}$.
 \end{lemma} 

 \begin{proof}
     Suppose that $f\in\mathcal{G}$ and $\mu\in\Gibbs$ is ergodic. 
     The lift $\hat{\mu}^-$ is not a projective $u-$Gibbs measure. Indeed, if we assume the contrary, then $\hmu^-$ has absolutely continuous conditionals along lifted unstable disks. However, its conditionals on $\mathbb{P}E^c$ fibers are given by $\delta_{E^{-c}}$. Therefore, for $\mu$-almost every $p$, and for $\Leb^u_p$ almost every $q$, $h^u_{q,p}(E^{-c}(p)) = E^{-c}(q)$ which implies that $f\notin \mathcal{G}$.  

     If $\mu$ is not ergodic, then \Cref{pro:Ledrappierproj} implies that there exists an ergodic component $\nu$ of $\mu$ for which $\hnu^-$ is a projective $u$-Gibbs measure, which gives a contradiction. \qedhere

 \end{proof}

After these preparations, we are ready to conclude the proof of \Cref{mainthm:TheoremD}.

\begin{proof}[Proof of \Cref{mainthm:TheoremD}]
  It remains to prove part $(B)$. Suppose that $f\in\mathcal G$ does not have the expanding on average property. Then, by \Cref{rem:lengthdiscs}, for every $\eta>0$ we can find some point $p$, some unit vector $v\in E^c_p$, and some infinite set $S \subset\Nat$ so that
  \[
  \lim_{n\in S} \frac{1}{n\Leb^u(\Wu{p,1})}\int_{\Wu{p,1}}\log\norm{D_qf^n\cdot \chi_{q,p}(v)}\der\Leb^u(q)\leq \eta.
  \]
As in the proof of \Cref{pro:thmB1}, write
\[
\hat{m}_{\hat{D}_{(p,[v])}} := \frac{1}{\Leb^u(W^u(p,1))}\int_{W^u(p,1)} \delta_{(q,[h^u_{q,p}(v)])} \der\Leb^u(q),
\]
and consider the sequence of projective measures
\[
\displaystyle \hat{\mu}_n^\eta:= \frac{1}{n} \sum_{j=0}^{n-1} \hat{f}^j_* \hat{m}_{\hat{D}_{(p,[v])}}.
\]
Let $\psi$ be as in the proof of \Cref{pro:thmB1}, and by the chain rule,
\[
\displaystyle \int \psi \der\hat{\mu}^\eta_n =   \frac{1}{n\Leb^u(\Wu{p,1})}\int_{\Wu{p,1}}\log\norm{D_qf^n\cdot \chi_{q,p}(v)}\der\Leb^u(q).
\]
Hence, $\lim_{n\in S}\int \psi \der\hat\mu^{\eta}_n\leq \eta$. The only accumulation measures of the $(\hat\mu^{\eta}_n)_n$ are projective $u-$Gibbs measures (\Cref{lem:projectivepesinsinai}). Up to considering a subsequence,  there is some $\hat\mu\in\Gibbs[\hf]$ so that 
\[
\displaystyle \lim_{n\in S}\int \psi \hat\mu^{\eta}_n =\int \psi \der \hmu.
\]
Note that $\mu=\pi_*\hat\mu$ is a $u$-Gibbs measure. Therefore, there exists a function $\rho$ taking values in $[0,1]$, such that 
\[
\hat{\mu} = \displaystyle \int\left( \rho(x) \hat{\mu}_x^+ + (1-\rho(x)) \hat{\mu}_x^- \right)\der\mu(x),
\]
where $\mu_x$ is the ergodic component of $\mu$ containing $x$, for $\mu$-almost every $x$, and $\hat{\mu}_x^*$ are the lifts to the stable or unstable Oseledets direction, for $*=-,+$.  \Cref{cor:onlygibbs} tell us that the only projective $u$-Gibbs measures that project on $\mu$ are supported on $E^{+c}_{\mu}$, hence $\hat{\mu} = \hat{\mu}^+$, and 
\[
\displaystyle \int \psi \der\hat{\mu}^+ = \lambda^{+c}(\mu)\leq \eta.
\]
For $0 < \eta< \inf \{\lambda^{+c}(\nu): \nu \in \Gibbse\}$, this contradicts the definition of $\mathcal{G}$. It is immediate that if $f$ is expanding on average and $\sup |\det D^c_x f| \leq 1$ then $\Gap>0$.
\end{proof}

\subsection{Robustness of the expanding on average property}
\label{ssec:Robustness}

 In this subsection we show the following theorem.
\begin{theorem}\label{THM:robustnessofea}
    Suppose that $f\in \PHb[2] $ has the expanding on average property and $\Gap>0$, then for any $g\in \PHb[2]$ which is sufficiently $C^2$-close to $f$ has the expanding on average property with $\mathrm{Gap}(g)>0$. 
\end{theorem}
\begin{proof}
    Suppose the conclusion does not hold, then there exists a sequence $ \left(f_n\right)_{n\in \Nat}$ converging to $f$ in the $C^2$-topology such that for $n$ sufficiently large, the diffeomorphism $f_n$ either 
    \begin{enumerate}
        \item is not expanding on average, or 
        \item $f_n$ has the expanding on average property and $\mathrm{Gap}(f_n) \leq 0$.
    \end{enumerate}
    Let us first show that the first item does not happen for sufficiently large $n$. Let $\CEA(f)$ be the expanding on average constant for $f$ and fix $\eta \in (0, \CEA(f))$.
    Suppose not, then, up to considering a subsequence, we can assume that  for each $n\in \Nat$ sufficiently large, there exist $p_n \in M$, a unit vector $v_n \in E^c_{p_n,f_n}$ and some infinite set $S_n^\eta \subset \Nat$ such that
    \[
    \displaystyle \lim_{k\in S_n^\eta} \frac{1}{k \Leb^u(\Wu{p_n,f_n})} \int_{\Wu{p_n,f_n}} \log \|D_q f_n^k \cdot \chi_{q,p_n,f_n}(v_n)\| \der\Leb^u(q) \leq \eta.
    \]
    
    As in the proof of Theorem \hyperlink{theoremD}{D}, for each $n$, we find a measure $\hmu_n \in \Gibbs[\hf_n]$ such that 
    \[
    \displaystyle \int \psi_n \der\hmu_n \leq \eta,
    \]
    where we consider the function $\psi_n(x,[v]) = \log \|D_xf_n|_{[v]}\|$. We can consider the function $\psi_n$ as a function defined in $\mathbb{P}TM$. If $f_n$ converges to $f$ in the $C^1$-topology, then $\psi_n$ converges to $\psi(x,[v])= \log \|D_xf|_{[v]}\|$ in the $C^0$-topology. 

    Up to considering a subsequence, we may assume that $\hmu_n$ converges to some $\hf$-invariant measure $\hmu$ and by \Cref{lem:closednessGibbs} the measure $\hmu\in \Gibbs[\hf]$. So,
    \[
    \displaystyle \CEA(f) > \eta \geq \lim_{n\to +\infty} \int \psi_n \der\hmu_n = \lim_{n\to +\infty} \int \psi \der\hmu_n + \int (\psi_n - \psi )\der\hmu_n = \int \psi \der\hmu.
    \]
    By \Cref{mainthm:TheoremD} and \Cref{cor:onlygibbs},
    \[
    \int \psi \der\hmu \geq \CEA(f)
    \]
    which is a contradiction.

    Since in the computation above we could have taken any $\eta \in (0, \CEA(f))$, then, for any $\varepsilon>0$, for $n$ large enough $\CEA(f_n) \geq \CEA(f) - \varepsilon$. Suppose now that, up to a subsequence, each $f_n$ has the expanding on average property but $\mathrm{Gap}(f_n) \leq 0$, for $n$ sufficiently large. Recall that 
    \[
    \displaystyle \mathrm{Gap}(f_n) := \CEA(f_n) - \sup_x \limsup_{k \to +\infty} \frac{1}{k} \log |\det D_x^c f_n^k|.
    \]
    \begin{claim}
    The function 
    \[
    \displaystyle \PH[1] \ni g \mapsto \sup_x \limsup_{k \to +\infty} \frac{1}{k} \log |\det D^c_x g^k|\] is upper semicontinuous.
    \end{claim}
    \begin{proof}
    Given $g\in \PH[1]$, write
    \[
    c(g) := \sup_x \limsup_{k \to +\infty} \frac{1}{k} \log |\det D^c_x g^k|.
    \]
    We want to show that if $g_n $ converges to $g$ in the $C^1$-topology, then $\limsup_n c(g_n) \leq c(g)$. For each $g\in \PH[1]$, define $\psi(x,g) = \log |\det D_x^cg|,$ defined on $M \times \PH[1]$. Observe that this function is continuous, where we are using the $C^1$-topology in the second coordinate. It is easy to see that for each $g\in \PH[1]$, we have
    \[
\displaystyle     c(g) = \sup \left\{ \int \psi(x,g) \der \nu, \textrm{ where }\nu \textrm{ is $g$-invariant}\right\}.
    \]
    Suppose that $g_n$ converges to $g$. For each $\varepsilon>0$ and $n\in \Nat$, there exists a $g_n$-invariant measure $\nu_n$ such that
    \[
    \displaystyle \int \psi(x,g_n) \der \nu_n(x) \geq c(g_n) - \varepsilon. 
    \]
    Let $n_k$ be a sequence such that $\limsup_n c(g_n) = \lim_k c(g_{n_k})$. 
    Up to considering a subsequence, suppose that $\nu_{n_k}$ converges to a measure $\nu$, which is $g$-invariant. Then, since the function $\psi$ is continuous,
    \[
\displaystyle     \limsup_n c(g_n) - \varepsilon = \lim_k c(g_{n_k}) - \varepsilon < \int \psi(x,g_{n_k}) \der \nu_{n_k}  = \int \psi(x,g) \der \nu(x) \leq c(g).
    \]
    Therefore,
    \[
    \limsup_n c(g_n) \leq c(g) + \varepsilon.
    \]
    Since this is true for any $\varepsilon>0$, we obtain the upper semicontinuity. \qedhere
    \end{proof}
     Take $\varepsilon = \Gap$. For $n$ large enough,
    \[
\displaystyle    \CEA(f_n) > \CEA(f) - \frac{\varepsilon}{3} 
    \]
    and
    \[
    \displaystyle  \sup_x \limsup_{k \to +\infty} \frac{1}{k} \log |\det D^c_x f_n^k| < \sup_x \limsup_{k \to +\infty} \frac{1}{k} \log |\det D^c_x f^k| + \frac{\varepsilon}{3}.
    \]
    Therefore, 
    \[
    \displaystyle \mathrm{Gap}(f_n) > \Gap - \frac{2\varepsilon}{3} = \frac{\Gap}{3}>0. \qedhere
    \]
    
\end{proof}

\begin{remark}
  It thus follow that the Examples in \Cref{sec:Examples} are robust under $C^2$-perturbations. One gets, in particlar, $C^2$-open sets of non-uniformly hyperbolic diffeomorphism without global dominated splitting. Examples as such are not-addressed in the literature. 
\end{remark}

\subsection{Continuity of Lyapunov exponents and the finitary expanding on average property}\label{Subsection:Continuityexp}

We are ready to prove \Cref{mainthm:TheoremB}.

\begin{proof}[Proof of \Cref{mainthm:TheoremB}]

    Let us first show the continuity of the top center Lyapunov exponent. Let $f_n$ be as in the statement of \Cref{mainthm:TheoremB}. Observe that, by \Cref{THM:robustnessofea}, we may assume that $f_n$ has the expanding on average property with $\Gap>0$.  As in the proof of \Cref{THM:robustnessofea}, consider $\psi_n(x,[v]) := \log \|D_xf_n|_{[v]}\|$. For each $\mu_n$, consider $\hmu_n^+$ to be the lift of $\mu_n$ to the Oseledets direction $E^{+c}$.

By \Cref{THM:robustnessofea}, \Cref{mainthm:TheoremD}, we can apply  \Cref{cor:onlygibbs}, which implies that  this is the only projective $u$-Gibbs measure that projects to $\mu_n$. Moreover, we know that
  \[
   \lambda^{+c}(\mu_n,f_n) = \int \psi_n d\hat{\mu}_n^+.
   \]
  Take $\hat \mu$ a weak*-accumulation measure of the sequence $(\hat{\mu}_n^+)_n$, which by \Cref{lem:closednessGibbs} is also a projective $u$-Gibbs measure. We know that the measure $\hat{\mu}$ projects to $\mu$ and, since $f$ has the expanding on average property and $\Gap>0$, we have that $\hat{\mu} = \hat{\mu}^+$. Write $\psi(x,[v]) = \log \|D_xf|_{[v]}\|$ and as in the proof of \Cref{THM:robustnessofea},
  \[
  \lim_{n\to +\infty} \lambda^{+c}(\mu_n,f_n) = \lim_{n\to +\infty} \int \psi_n \der \hmu_n^+ = \lim_{n\to +\infty} \int \psi \der \hmu_n^+ + \int (\psi_n - \psi )\der \hmu_n^+ = \int \psi \der\hmu^+ = \lambda^{+c}(\mu,f),
  \]
  which gives us the continuity of the top center exponent. 
 
 The continuity of the bottom center exponent then is an easy consequence of the fact that $(\nu,g) \mapsto \int \log|\det D_x^cg | d\nu(x)$ is continuous among partially hyperbolic systems for the $C^1$-topology, and for any $f$-invariant measure $\nu$ we have $\lambda^{+c}(\nu) + \lambda^{-c}(\nu)  = \int \log |\det D^c_xf| d\nu(x)$.
\end{proof}

This continuity property of center Lyapunov exponents will  crucial in this work.  It is convenient to come up with finite time versions of the expanding on average property. We will introduce two versions of these finite time conditions.

\begin{definition}
    \label{def:utillingR}
    For $R>r>0$, a \emph{$u$-tiling} bounded by $(r,R)$ is a family of measurable sets of unstable leaves $\mathcal{Q}$, which we will call $u$-cells, with the following properties:
    \begin{enumerate}
        \item for every $p\in M$, the $u$-cells contained in $\Wu{p}$ cover this unstable manifold, and, moreover, two distinct $u$-cells intersect in a set of $\Leb^u$-measure zero;
        \item every $Q\in \mathcal{Q} $ satisfies $\Leb^u(Q)>0$, $\Leb^u(\partial Q)=0$ and $\Wu{x_Q,r} \subset P \in \Wu{x_Q,R}$, for some $x_Q \in Q.$
    \end{enumerate}
If we only require the control from above by $R$, we will say that that the $u$-tiling is bounded from above by $R$.
\end{definition}

Recall the given $D$ a $u$-set of positive Lebesgue measure, we can construct the measure $\dd\mu_D^u(q)=\rho_D^u(q)\dd\Leb^u_D(q)$ using the canonical density, as we explained in \cpageref{page:canonicaldensities}.

\begin{definition}[Finitary expanding on average on $u$-tilings property]\label{def:finitaryutillings}
    Let $\mathcal{Q}$ be a $u$-tilling. We say that $f$ verifies the finitary expanding on average property with respect to $\mathcal{Q}$ if there are constants $N\in \N$ and $C>0$ such that for any $Q\in \mathcal{Q}$, every $p\in Q$ and unit vector $v\in E^c(p)$, we have
    \[
    I_N(Q,p,v):= \frac{1}{N} \int_Q \log \|D_qf^N \cdot \chi_{q,p}(v)\|d\mu_Q^u(q)\geq C,
    \]
    where $\mu^u_Q$ is the canonical measure on $Q$. In this case, we say that $f$ is finitary expanding on average with respect to the tiling $\mathcal{Q}$ and constants $(N,C)$.
\end{definition}

 Another similar definition that we will use is the following.

\begin{definition}[Finitary expanding on average property] \label{def:finitaryea}
   We say that  $f\in\PHb$  satisfies the \emph{finitary expanding on average property }
   if  there exist $C>0$ and $N\in\Nat$ so that: for every unit vector $v\in E^c_p$, it holds
    \[
    \displaystyle \frac{1}{N\Leb^u(\Wu{p,1})}\int_{\Wu{p,1}}\log \norm{D_qf^N\cdot\chi_{q,p}(v)} \der\Leb^u(q)>C.
    \]
    In this case we say that $f$ is finitary expanding on average with constants $(N,C)$. 
 \end{definition}

 In what follows, we will explore the relation between these definitions and the expanding on average definition we have. 

 \begin{lemma}\label{prop:finitaryUEAP}
     Suppose that $f\in \PHb[2]$ is expanding on average and recall the definition of $\CEA$ \eqref{eq.optimalconstant}. For any $C_1\in (0,\CEA)$, we have
     \begin{enumerate}
         \item $f$ verifies the finitary expanding on average property for any $N$ sufficiently large, that is, there exists $N_0 \in \N$ such that for any $N\geq N_0$, $f$ is finitary expanding on average with constants $(N,C_1)$. 
         \item For any $0<r<R$, for every $u$-tilling $\mathcal{Q}$ bounded by $(r,R)$, there exists $N_0$ such that for any $N\geq N_0$, $f$ is finitary expanding on average with respect to $\mathcal{Q}$ with constants $(N,C_1)$.
     \end{enumerate}
 \end{lemma}

 \begin{proof}
     The proof is essentially the same for both items and follows similar ideas as in the proof of \Cref{THM:robustnessofea} and \Cref{mainthm:TheoremB}. Let us start with the first. Suppose the conclusion does not hold. Fix $C_1 \in (0,\CEA)$ and suppose the conclusion does not hold, then for  every $N\in \Nat$ there are $p_N\in M$ and a unit vector $v_N\in E^c_{p_N}$ such that
    \begin{equation}\label{eq.eqneeded1}
    \frac{1}{N\Leb^u(\Wu{p_N,1})}\int_{\Wu{p_N,1}}\log \norm{D_qf^N\cdot\chi_{q,p_N}(v_N)} \der\Leb^u(q)\leq C_1.
    \end{equation}
    Write $D_N := \Wu{p_N,1}$ and consider the sequence of projective measures 
    \[
    \hmu_N:= \displaystyle \frac{1}{N} \sum_{j=0}^{N-1} \hf^j_*\hat{m}_N,
    \]
    where
\[
\hat{m}_N:= \frac{1}{\Leb^u(D_N)}\int_{D_N} \delta_{(q,[h^u_{q,p_N}(v_N)])} \dd\Leb^u(q).
\]
    Observe that \eqref{eq.eqneeded1} gives that
    \[
    \displaystyle \int \psi \dd\hmu_N \leq C_1,
    \]
    where $\psi(x,[v]) = \log \|D_xf|_{[v]}\|$. Up to considering a subsequence, we may assume the existence of a measure $\hmu$ such that $\hmu_N \to \hmu$. Observe that the bounded distortion estimates from \Cref{lem:bdddistortion} holds uniformly for every such disk. Therefore, the same proof of \Cref{lem:projectivepesinsinai}, following Pesin-Sinai's proof, implies that $\hmu$ is a projective $u$-Gibbs measure and therefore
    \[
    \displaystyle   C_1 \geq \int \psi \dd\hmu = \lambda^{+c}(\hpi_*\hmu) \geq \CEA,
    \]
    which is a contradiction.

    The same argument works to prove by contradiction the second item. The boundedness of the $u$-tiling gives the bounded distortion estimates needed to follow the proof above. Since the $u$-cells are bounded from below, they contain a disk of size $r$, this is what gives the uniformity on time $N_0$ by following the same proof above. 
 \end{proof}

\Cref{prop:finitaryUEAP} gives that the expanding on average property implies both finitary expanding on average that we defined. It turns out that \Cref{def:finitaryutillings} for every $u$-tilling is equivalent to the expanding on average property. We do not know whether \Cref{def:finitaryea} is equivalent to expanding on average. We are introducing both finitary expanding on average properties because the one with $u$-tillings is what we use to check that the examples we introduce in \Cref{sec:Examples} are expanding on average. \Cref{def:finitaryea} is a useful property in some of the analysis we will do in the next few sections. Let use remark that the definition of expanding on average that Cooper uses in \cite{Coopererg} is closer to our \Cref{def:finitaryea}.

Next, we want to show that \Cref{def:finitaryutillings} implies the expanding on average property. We use two important properties of canonical measures. They have the property that for any set $D$ inside an unstable leaf with positive $\Leb^u$-measure has the property that $f^k_* \mu^u_D = \mu^u_{f^k(D)}$ and $\mu^u_D(.|_E) = \mu^u_E$.

\begin{proposition}\label{pro:finitarydim1}
   Assume that $f\in\PHb[2]$ has the finitary expanding on average property with respect to a $u$-tiling $\mathcal{Q}$ which is bounded by $R$. Then $f$ has the expanding on average property.
\end{proposition}

\begin{proof}
    Suppose that $f$ is finitary expanding on average with respect to $\mathcal{Q}$ and constants $(N,C)$. Write $c_f:= \max_x\max \{\log\|D_xf\|, - \log m(D_xf)\}$ and let $g = f^N$. 
    \begin{claim}
        Suppose that $g$ has the finitary expanding on average property with constant $K>0$. Then, $f$ has the finitary  expanding on average property.

        Likewise, if $g$ has the finitary expanding on average property with respect to a $u$-tiling $\mathcal{Q}$, then the same follows for $f$.
    \end{claim}
    \begin{proof}
        For $n\in \N$, write $n = mN + k$, where $k\in \{0, \cdots, N-1\}$. For any center vector $v\in E^c(x)$, by the chain rule,
        \[
        \log \|D_xf^n v\| \geq \log \|D_xg^m v\| - N c_f.
        \]
        Therefore, for any $p\in M$, any $u$-disk $D_p$ and any unit vector $v\in E^c(p),$ 
        \begin{multline*}
        \liminf_{n\to +\infty} \frac{1}{n\Leb^u(D_p)}\int_{D_p} \log \|D_qf^n \cdot \chi_{q,p}(v)\| \der \Leb^u(q)\\
        \geq \frac{1}{n \Leb^u(D_p)} \int_{D_p} \log \|D_q g^m \cdot \chi_{q,p}(v)\| \der \Leb^u(q)- Nc_f  
        \end{multline*}
        which gives 
        \begin{multline*}
         \liminf_{n\to +\infty} \frac{1}{n\Leb^u(D_p)}\int_{D_p} \log \|D_qf^n \cdot \chi_{q,p}(v)\| \der \Leb^u(q) \geq\\ \frac{1}{N} \liminf_{m\to +\infty} \frac{1}{m\Leb^u(D_p)} \int_{D_p} \log \|D_qg^m \cdot \chi_{q,p}(v)\| \der \Leb^u(q) \geq K.
        \end{multline*}
    So $f$ is expanding on average.

    The second case is proven analogously.
    \end{proof}

    Up to replacing $f$ by $f^N$, we may suppose that $f$ has the finitary expanding on average property for $\mathcal{Q}$ with constants $(1,C)$.

    Given $D$ inside an unstable leaf with positive $\Leb^u$-measure, the integral with respect to $\mu^u_D$ has better additive properties than the corresponding ones for $\Leb^u$:
\begin{equation*}
    \frac 1n\int_D\log\|D_qf^n\chi_{q,p}(v)\|\dd\mu^u_D(q)=\frac 1n\sum_{k=0}^{n-1}\int_{f^k(D)}\log\|D_qf\chi_{q,f^k(p)}(v^{k}))\|\dd\mu^u_{f^k(D)}(q)
\end{equation*}
where $v^{k}= \frac{D_pf^k(v)}{\|D_pf^k(v)\|}$ are unit vectors.

Fix $D_p$ a disk centered at $p$ inside $\Wu{p}$ and a unit vector $v\in E^c(p)$. 
For every $k\geq 0$, $f^k(D_p)$ consider $\mathcal{Q}_k$ to be the collection of $u$-cells in $\mathcal{Q}$ that are completely contained in $f^k(D_p)$ and set $R_k = f^k(D_p) - \bigcup_{Q\in \mathcal{Q}_k} Q.$ Observe that $R_k$ is contained in a neighborhood of radius $R$ of the boundary of $f^k(D_p)$. Thus, 
\[
\mu_{f^k(D_p)}^u(R_k) = \mu_{D_p}^u(f^{-k}(R_k)) \leq \mu^u_{D_p}(f^{-k}(\partial_Rf^k(D_p)))\leq\mu^u_{D_p}(\partial_{R\lambda^{-k}}D_p)=a_k\rightarrow 0,
\]
where $\partial_R f^k(D_p)$ is the $R$ neighborhood of the boundary of $f^k(D_p)$ and $\lambda\in (0,1)$ is related to the contraction of the unstable Jacobian in the past. In this part, we are only assuming that the $u$-cells have size bounded from above by $R$. Observe that there are at most countably many $u$-cells in $\mathcal{Q}^k$. Write
\[
I_n(p,v):= \frac{1}{n}\int_{D_p} \log \|D_qf^n(q) \cdot \chi_{q,p}(v)\| \der \mu^u_{D_p}(q).
\]
 Thus we can write 
    \begin{align*}
    I_n(p,v) &=\frac 1n\sum_{k=0}^{n-1}\int_{f^k(D_p)}\log\|D_qf\chi_{q,f^k(p)}(v^k))\|\dd\mu^u_{f^k(D_p)}(q)\\
    &=\frac 1n\sum_{k=0}^{n-1}\sum_{Q^k_j\in \mathcal{Q}^k}\int_{Q^k_j} \log \|D_qf \cdot \chi_{q,f^k(p)}(v^k) \der \mu^u_{f^k(D_p)}(q)\\
    &\hspace{2cm}+\int_{R_k}\log\|D_qf\chi_{q,f^k(p)}(v^k)\|\dd\mu^u_{f^k(D_p)}(q)\\
        &=\frac 1n\sum_{k=0}^{n-1}\sum_{Q^k_j\in \mathcal{Q}^k}\mu^u_{f^k(D_p)}(Q^k_j)\int_{Q^k_j} \log \|D_qf \cdot \chi_{q,f^k(p)}(v^k) \der \mu^u_{Q^k_j}(q)\\
        &\hspace{2cm}+\int_{R_k}\log\|D_qf\chi_{q,f^k(p)}(v^k)\|\dd\mu^u_{f^k(D_p)}(q)\\
    &\geq \frac 1n\sum_{k=0}^{n-1}\left(C(1-a_k)-c_fa_k\right)=C-\frac{(C + c_f)\sum_{k=0}^{n-1}a_k}{n}.
    \end{align*}
    Therefore,
    \[
    \liminf_{n\to\infty}I_n(p,v)\geq C,
    \] and $f$ has the expanding on average property.
    
\end{proof}

Since in the previous proposition we have only used that the atoms are bounded from above, the same idea can be applied to increasing partitions subordinated to the unstable instead of $u$-tilings. In \Cref{sec:appendix} we indicate how to construct adapted measurable partitions, independently of the underlying $u$-Gibbs measure, and show the existence of good tilings in the following sense.

\begin{lemma}\label{lem:goodtiling}
  Given $R>0$ sufficiently small, there is a Borel family $\mc Q$ of relatively compact $u$-cells which verify the following.
  \begin{enumerate}
    \item On every unstable leaf, the cells form a locally finite covering, with disjoint interiors, and $Q\cap Q'\neq\emptyset$ implies $\Leb^u(Q\cap Q')=0$. 
    \item For every $Q\in\mc Q$ there exists $x_Q\in Q$ so that
        \[
        \Wu{x_Q,R} \subset Q \subset \Wu{x_Q,6R}.
        \]
        Moreover $\mu(\partial_u Q)=0$, and for some uniform constants $v, V>0$ we have
        \[
        v\leq \Leb^u(Q)\leq V.
        \]
        \end{enumerate}
\end{lemma}

\subsection{Large deviations estimates}
\label{ssec:LargeDeviation}

Our goal is to explain the following result:

\begin{proposition}
    \label{prop.largedeviationsprojective}
    Let $f\in \PHb[2]$ and let $\varphi:\hM \to \mathbb{R}$ be a H\"older continuous function and set 
    \[
    I_\varphi:= \left\{\int \varphi \der \hmu: \hmu \textrm{ is a projective $u$-Gibbs measure}\right\}.
    \]
    For every $\varepsilon>0$ and $R>0$,  there are constants $L>0$ and $\gamma>0$ such that for any $\hp \in \hM$, for every $n \geq 1$,
    \[
\displaystyle \hm_{\hWu{\hp,R}}^u \left(x: d \left( \frac{1}{n}\sum_{j=0}^{n-1} \varphi (\hf^j(x)), I_\varphi \right) \geq \varepsilon \right)\leq L e^{-\gamma n}. 
    \]
\end{proposition}

\begin{remark}
    Since the set of projective $u$-Gibbs measures is a convex, compact set, the set $I_\varphi$ above is actually a closed interval. 
\end{remark}
\begin{proof}
    Large deviations estimates are known to hold along unstable manifolds for partially hyperbolic systems. In what follows, we give a sketch of a proof following the approach by Dolgopyat that can be found in \cite{lecturenotesdolgopyat}.  Let us give this sketch for partially hyperbolic systems. We will then explain that the same ingredients needed in the partially hyperbolic setting hold for the projective unstable manifolds we are considering.

In what follows, we will use some notations used \cite{lecturenotesdolgopyat}. 
    Let $f$ be a partially hyperbolic $C^2$-diffeomorphism. Let $W$ be a local  unstable manifold, and $r>0$ small. Choose a set $Q = \{q_l\}$, which is a maximal $r$-separated set, that is, $d^u(q_l,q_m) \geq r$ for $l\neq m$ and for any $z\in W/Q$, there exists $q_l$ such that $d^u(q_l, z) < r$. Define
    \[
    K_l = \{z\in W: d^u(z,q_l) = \min_m d^u(z,q_m)\}.
    \]
    The sets $K_l$ are called \emph{cells}, and $W = \bigcup_l K_l$. This type of decomposition is also known as Voronoi Tesselation.

   It is easy to see that $B(q_l,r/2) \subset K_l \subset B(q_l,r)$, and the boundary of $K_l$ consists of finitely many smooth hypersurfaces, so $K_l$ is given by an open set and a piecewise smooth boundary.

    Fix $R>0$ and let $D = \Wu{p,R}$. Given a bounded subset of an unstable leaf with positive measure, $K$, we define the canonical density on $K$ as $\rho_{\mathrm{can}}(.) = \rho^K_{\mathrm{can}}(.)$ as the unique function determined by 
    \[
    \frac{\rho_{\mathrm{can}}(x)}{\rho_{\mathrm{can}}(y)} = \prod_{j=0}^{\infty}\frac{|\det D^u_{f^{-j}(x)}f^{-1}|}{|\det D^u_{f^{-j}(y)}f^{-1}|} \textrm{ and } \int_K \rho_{\mathrm{can}}(x) d\Leb^u_p(x) = 1.    \]
    The existence of the canonical density is a consequence of the bounded distortion. 
    Let $m$ be the measure in $D$ given by the corresponding canonical density. By bounded distortion, the normalized Lebesgue measure on $\Wu{p,R}$ is uniformly comparable with the measure $m$, this uniformity is independent of the choice of $p$.
    
    Fix $\varphi$ a H\"older continuous function on $M$ and $\varepsilon>0$. Write $S_n \varphi(x) = \varphi(x) +\cdots + \varphi(f^{n-1}(x))$. We can split the estimate in the proposition by estimating the measure of the sets 
    \[
    \left\{x\in D: \frac{1}{n} S_n \varphi(x) \geq \max I_\varphi + \varepsilon\right\} \textrm{ and } \left\{x\in D: \frac{1}{n} S_n \varphi(x) \leq \min I_\varphi - \varepsilon\right\}.
    \]
    We will estimate the first set, as the second set reduces to the first by replacing the arguments from $\varphi$ to $-\varphi$. Up to replacing $\varphi$ by $\varphi - \max I_\varphi - \varepsilon,$ we can assume that
    \[
    \int \varphi d\mu \leq -\varepsilon,
    \]
    for every $u$-Gibbs measure $\mu$. So it suffices to show that there exist constants $L$ and $\delta>0$ such that for every $n\in \Nat$
    \[
    m_K\left\{ x\in D: S_n\varphi(x) \geq 0 \right\} \leq L e^{-\delta n},
    \]
    where the estimates don't depend on $p$. The first ingredient is given by the following lemma.

    \begin{lemma}\label{lem:lemmaabove}
        There exists $n_0$ such that for any $n\geq n_0$, any $p\in M$, any set $K$ containing $p$ such that $\Wu{p,r/2} \subset K \subset \Wu{p,r}$, for $\nu$ either the normalized Lebesgue measure on $K$ or the measure obtained by the canonical density, we have
        \[
        \displaystyle \int \frac{1}{n} S_n \varphi d\nu \leq -\frac{\varepsilon}{2}.
        \]
    \end{lemma}
    \begin{proof}
    Consider the weak* neighborhood $\mathcal{V}$ of $\Gibbs$ obtained by $\mu \in \mathcal{V}$ if $\int \varphi d\mu <- \frac{\varepsilon}{2}$. Since the sets $K$ considered contain an unstable disk of size $r/2$ and are contained in an unstable disk of size $r$, there exists $n_0$ such that for any such set $K$, let $\nu$ be as in the statement, then $\frac{1}{n} \sum_{j=0}^{n-1} f^j_* \nu \in \mathcal{V}$, otherwise, one would obtain a $u$-Gibbs measure not contained in $\mathcal{V}$. The result then follows. 
    \end{proof}

    \Cref{lem:lemmaabove} states that any unstable set $K$ which is not too small, must be far from $0$ in a uniform time $n_0$.

    Given a $\gamma$-H\"older continuous function $\psi$, its H\"older norm is given by
    \[
    \|\psi\|_\gamma := \sup_{x\neq y}\frac{|\psi(x)-\psi(y)|}{d(x,y)^\gamma}.
    \]
    
    There exists a constant $L'$ with the property that for any $K$ inside an unstable ball of radius $r$ and having positive measure, the corresponding canonical density is $\log$-H\"older with $\|\log \rho_{\mathrm{can}}\|_\gamma \leq L'$, where $\gamma$ is the corresponding H\"older exponent. 
    
    For a set $K$, write $m_K$ the probability measure obtained by the canonical density on $K$. Using the quotient formula for the canonical density, one can show that $f_*m_K = m_{f(K)}$. This is the key property it is used for canonical densities.

   Since $\varphi$ is H\"older, there exists $L''>0$ with the following property: for any $K$ and $n\in \Nat$ such that $f^n(K)$ is contained in an unstable ball of radius $R$, verifies
    \[
    \displaystyle \sup_{x\in K} S_n\varphi(x) - \inf_{y\in K} S_n\varphi(y) \leq L''. 
    \]

    The key lemma for the proof of Large deviations in \cite{lecturenotesdolgopyat} is the following:
    \begin{lemma}[\cite{lecturenotesdolgopyat}, Lemma 4]
        \label{lemma:decompositionrenewal}
        There exist constants $L'''$ and $\theta \in (0,1)$ such that for any $K$ inside an unstable leaf containing an unstable ball of radius $1/2$ and contained in an unstable ball of radius $1$, whose boundary consists of finitely many smooth hypersurfaces,  there is a countable partition $K = \bigcup_j K_j$ and integers $n_j \geq 1$ with 
        \begin{enumerate}
            \item $f^{n_j}(K_j)$ is a cell in $f^{n_j}(K)$;
            \item $m_K\left( \bigcup_{j: n_j>N} K_j \right)< L''' \theta ^N$;
            \item $\sum_j m_K(K_j) \sup_{x\in K_j}\left( S_{n_j} \varphi (x)+ \frac{\varepsilon}{4} n_j \right) \leq 0.$
        \end{enumerate}
    \end{lemma}

    Item $(1)$ represents that $n_j$ is a type of ``renewal'' time where $K_j$ becomes an object of scale  $\approx r$. If one defines a ``stopping time'' $\tau(x)$ by   $\tau(x) = n_j$ if $x\in K_j$, then the second item says that $\tau$ has exponential tails, in other words $m_K(\tau>N) \leq L''' \theta^N$.   The third item is related to the following
    \[
    \int S_{\tau} \varphi \der m_K \leq \sum_j m(K_j) \sup_{x\in K_j} S_{n_j}\varphi(x) \leq -\frac{\varepsilon}{4} \int \tau \der m_K \leq -\frac{\varepsilon}{4},
    \]
    since $\tau \geq 1$.

Put 
\[
H_j:=\sup_{x\in K_j}S_{n_j}\varphi(x).
\]
So $|H_j|\le\|\varphi\|_\infty n_j$, and let
$\psi_K(\delta):=\sum_j m_K(K_j)e^{\delta H_j}$ for $\delta\ge0$. Then:
\begin{itemize}
    \item $\psi_K(0)=\sum_j m_K(K_j) = 1$;
    \item conclusion (3) of \Cref{lemma:decompositionrenewal} implies that $\sum_j m_K(K_j)H_j\le-\frac{\varepsilon}{4}$;
    \item conclusion (2) of \Cref{lemma:decompositionrenewal}  implies that 
    \[
    m_K(\tau = N) \leq L''' \theta^{N-1}.
    \]
    Set $a_0=\frac{1}{2}\log(1/\theta)$. We have that
    \[
\displaystyle  L'''':= \sup_{K} \sum_j m_K(K_j)n_j^2e^{a_0n_j} \leq \sum_{N\geq 1}N^2 e^{a_0 N} L''' \theta^{N-1} = \sum_{N \geq 1} L''' N^2 \theta^{-\frac{N}{2}} \theta^{N-1}< +\infty,
    \] 
    where the supremum is taken over all of the sets $K$ where \Cref{lemma:decompositionrenewal} applies. 
\end{itemize}
 Applying $e^t\le1+t+\frac{t^2}2e^{|t|}$, we have for $0<\delta\le a_0/\|\varphi\|_\infty$,
\[
  \psi_K(\delta)\leq 1-\frac{\delta\varepsilon}{4}+\frac{\delta^2}{2}\|\varphi\|_\infty^2L'''' ,
\]
so that for $\delta$ small enough, $\psi_K(\delta)\leq\overline{\theta}<1$. This is independent on the choice of $p$, which determines $D$.

Given $l\in \Nat$, we define a countable partition $K = \bigcup_j K_{j,l}$ with times $n_{j,l}$ defined inductively, by applying \Cref{lemma:decompositionrenewal} as follows. Suppose we had defined this partition for $l$, let us explain how to define it for $l+1$. We have that $f^{n_{j,l}}(K_{j,l})$ is a cell. By \Cref{lemma:decompositionrenewal}, there is a maximal $1$-separated set $Q = \{q_i\}$ and we can write $f^{n_{j,l}}(K_{j,l}) = \bigcup_i B_i$, where we denote by $B_i$ the corresponding cell for $q_i$. Set $K_{j,l,i}:= f^{-n_{j,l}}(B_i)$ and $n_{j,l,i} = n_{j,l} + n(B_i)$. Then, $f^{n_{j,l,i}}(K_{j,l,i})$ is again a cell and we may reindex $\{K_{j,l,i}\}$ as $K_{j,l+1}$. Write $n_l(x) = n_{j,l}(x)$ and observe that $n_0(.)< n_1(.) < n_2(.) < \cdots $. The numbers $n_l$ represent the total amount of time for the $l^{th}$-renewal. 

Since $S_{n_{j,l,i}}\varphi(x) = S_{n_{j,l}}\varphi(x) + S_{n(B_i)}\varphi(f^{n_{j,l}}(x))$, we have that
\[
\sup_{K_{j,l,i}} S_{n_{j,l,i}} \varphi \leq \sup_{K_{j,l}} S_{n_{j,l}} \varphi + \sup_{B_i} S_{n(B_i)} \varphi.
\]
 Using this and the estimates above, we obtain that
\begin{equation}\label{eq.integralsn}
  \int e^{\delta S_{n_l(x)}\varphi(x)} \der m_K(x)\leq \overline{\theta}^{l},
\end{equation}
see page $7$ in \cite{lecturenotesdolgopyat}. Fix $\varepsilon>0$ small enough so that $e^\varepsilon \overline \theta <1$. By \eqref{eq.integralsn} and Markov inequality, we have that
\begin{equation}
    \label{eq.decaymeasure1}
    m_K\left( S_{n_l(x)} \varphi(x) \geq -\frac{\varepsilon l}{\delta} \right) \leq \left(e^\varepsilon \overline \theta \right)^l. 
\end{equation}
Set $\eta:= e^\varepsilon \overline \theta.$

A similar argument gives that there exist constants $C$ large and $\tilde{\theta} \in (0,1)$ such that
\begin{equation}\label{eq.decaymeasure2}
m_K(n_l(x) \geq Cl) \leq \tilde \theta^l. 
\end{equation}

Observe that 
\begin{equation}
    \label{eq.decaymeasure3}
    m_K\left( n_{l+1}(x)- n_{l}(x) >N \right) \leq L''' \theta^{N}.
\end{equation}

Set $\overline \varepsilon:= \frac{\varepsilon}{\delta}$. Let $l(n,x)$ be the number of renewals up to time $n$. If
$S_n\varphi(x)\ge0$, then  one of the following holds (see \cite{lecturenotesdolgopyat}, page 8)
\begin{itemize}
    \item $l(n,x) > \frac{n}{C}$ (too few renewals);
    \item $S_{n_l} \varphi(x) \geq -\overline{\varepsilon}l$;
    \item $n_{l+1}(x) - n_l(x) > \frac{\overline{\varepsilon}l}{\|\varphi\|_{\infty}}$.
\end{itemize}
By \eqref{eq.decaymeasure1}, \eqref{eq.decaymeasure2} and \eqref{eq.decaymeasure3}, all of these three events have an exponential decay in $l$. This gives the desired large deviations estimate. 

\medskip
\textbf{Conclusion}
\medskip

In the argument above, the two key properties of unstable manifolds used  are bounded distortion for $m_K$, and uniform expansion of unstable leaves. Recall that in our definition of projective $u$-Gibbs measures, and in all of the projective constructions we work with, we lift the Lebesgue measures on unstable leaves to projective unstable leaves using holonomies. In particular, we also have all of the distortion estimates (\Cref{lem:bdddistortion}) and expansion estimates needed for the construction above. Indeed, for a fixed projective unstable leaf, we can follow the entire construction above in the corresponding unstable leaf on the base, and lift everything to the projective unstable leaf.  So the proof of large deviations holds verbatim in this case. 
\end{proof}

\subsection{Consequences of Large deviations}
\label{ssec:ConsequencesLD}

A useful consequence of \Cref{prop.largedeviationsprojective} is the following. 

\begin{corollary}
    \label{cor.accugibbspt}
    Let $f\in \PHb[2]$. For every $p\in M$ and any unit vector $v\in E^c(p)$, there exists a set $G(p,v) \subset \Wuloc{p}$ of full $\Leb^u_p$-measure, with the property that for any $q\in G(p,v)$, any accumulation point of
    \[
    \displaystyle \frac{1}{n} \sum_{j=0}^{n-1} \delta_{\hf^j(q,[\chi_{q,p}(v)])}
    \]
    is a projective $u$-Gibbs measure. 
\end{corollary}
\begin{proof}
    Take $\{\varphi_k\}_{k\in \mathbb{N}}$ a sequence of H\"older functions which is dense in $C^0(\hM)$. Fix $\hp=(p,[v]) \in \hM$ and write $\hat D:= \hWuloc{\hp}$. For each $k$, by \Cref{prop.largedeviationsprojective} and the Borel-Cantelli Lemma, there exists a set $G_k(\hp)$ of full $\hm_{\hp^u}$-measure such that for any $\hq \in G_k(\hp)$
    \[
\displaystyle  \lim_{n\to +\infty} d\left(\frac{1}{n} \sum_{j=0}^{n-1} \varphi_k (\hf^j(\hq)), I_{\varphi_k}\right) = 0.
    \]
    Set $G(\hp) := \bigcap_k G_k(\hp)$ and let us check that this set verifies the conclusion of \Cref{cor.accugibbspt}. 

    First, observe that the map
    \[
    C^0(\hM) \ni \varphi \mapsto I_\varphi
    \]
    is $1$-Lipschitz, where we are using the $C^0$-topology in the domain and the Hausdorff distance in the range. This follows easily from the following observation, for any $a\in I_\varphi$, there exists $\hmu \in \mathrm{Gibbs}^u(\hf)$ such that $a= \int \varphi \der \hmu$. In particular, for any $\psi \in C^0(\hM)$, 
    \[
    \displaystyle \left| \int \varphi \der \hmu - \int \psi \der \hmu\right| \leq \|\varphi - \psi\|_0.
    \]
    Therefore, $\sup_{a\in I_\varphi} d(a, I_\psi) \leq \|\varphi - \psi\|_0$ and, similarly $\sup_{b\in I_\psi} d(I_\varphi, b) \leq \|\varphi - \psi\|_0.$ This implies that $d_H(I_\varphi, I_\psi) \leq \|\varphi - \psi\|_0$. 

    Take $\hq \in G(\hp)$ and consider $\hnu$ an accumulation point of the sequence  
    \[
    \displaystyle \frac{1}{n} \sum_{j=0}^{n-1} \delta_{\hf^j(\hq)}.
    \]
    Let $\varphi \in C^0(\hM)$ be  a continuous projective function and let $k_j \to +\infty$ be a sequence  such that $\varphi_{k_j}$ converges to $\varphi$ in the $C^0$-topology. Observe that for each $j$, $\int  \varphi_{k_j} \der \hnu \in I_{\varphi_{k_j}}$. Since $I_{\varphi_{k_j}} \to I_\varphi$, we obtain that
    \[
    \int \varphi \der \hnu \in I_\varphi.
    \]
    Suppose $\hnu$ is not a projective $u$-Gibbs measure. Recall that the set of projective $u$-Gibbs measures is a convex and compact subset of the space of measures in $\hM$. The Hahn-Banach Theorem gives a continuous function $\psi$ such that 
    \[
\displaystyle    \int \psi \der \hnu > \sup \left\{\int \psi \der \hmu: \hmu \in \mathrm{Gibbs}^u(\hf)\right\},
    \]
    which is a contradiction with $\int \psi \der \hnu \in I_\psi.$ \qedhere
    
\end{proof}

Using \Cref{cor.accugibbspt}, we can strengthen the estimate obtained in \Cref{lem:expansioncontraction}. This is given in the following lemma. 

\begin{lemma}\label{lem:expansioncontraction2}
Let $f\in \PHb[2]$ be expanding on average with $\Gap>0$. For every $p\in M$, $\Leb^u$-almost every $q\in\Wuloc{p}$ it holds
\begin{align}
&\liminf_{n\to\oo} \frac{1}{n}\log\norm{D^c_qf^n}\geq \CEA\\
&\limsup_{n\to\oo} \frac{1}{n}\log\conorm{D^c_qf^n}\leq  -\Gap.
\end{align}
\end{lemma}
\begin{proof}
Let $f\in \PHb[2]$ be expanding on average with $\Gap>0$.  Fix $p\in M$ and $v\in E^c(p)$ a unit vector. By \Cref{cor.accugibbspt}, there exists a set $G(p,v) \subset \Wuloc{p}$ of full $\Leb^u_p$-measure such that every accumulation point of     \[
    \displaystyle \hmu_{q}^n:= \frac{1}{n} \sum_{j=0}^{n-1} \delta_{\hf^j(q,[\chi_{q,p}(v)]}
    \]
    is a projective $u$-Gibbs measure. Consider the continuous projective function $\psi(x,[w]) = \log \|D_xf|_{[w]}\|$. Observe that 
    \[
    \displaystyle \int \psi \der \hmu_q^n = \frac{1}{n} \sum_{j=0}^{n-1} \log \|D_{q_j}f|_{[D_qf^j(\chi_{q,p}(v)]}\| = \frac{1}{n} \log \|D_qf^n|_{[\chi_{q,p}(v)]}\| \leq \frac{1}{n} \log \|D^c_q f^n\|.
    \]
    By \Cref{mainthm:TheoremD} and \Cref{cor:onlygibbs}, the only projective $u$-Gibbs for $f$ are supported on the unstable Oseledets directions. Therefore,
    \[
    \displaystyle \liminf_{n\to +\infty}\int \psi \der \mu_q^n \geq \inf\{\lambda^{+c}(\mu): \mu \in \Gibbs\} = \CEA. 
    \]
    Hence,
    \begin{equation} \label{eq.CEAnorm}
    \CEA \leq \liminf_{n\to +\infty}\frac{1}{n} \log \|D^c_q f^n\|.
    \end{equation}
    From  the definition of $\Gap$ given in \Cref{def:Gap}, we have that
    \[
    \displaystyle \limsup_{n\to +\infty} \frac{1}{n} \log |\det D^c_qf^n| = \limsup_{n\to +\infty} \frac{1}{n} \log \|D^c_qf^n\|\cdot m(D^c_qf^n) \leq -\Gap +\CEA.
    \]
    By \eqref{eq.CEAnorm}, we obtain that
    \[
    \displaystyle \limsup_{n\to +\infty} \frac{1}{n} \log m(D_q^cf^n) \leq -\Gap.\qedhere
    \]
\end{proof}

\section{Existence of a weak stable direction}\label{sec:existenceweakbundle}

The main goal of this subsection is to prove  \Cref{thm:existsweakbundle} below, which states that inside any unstable leaf, Lebesgue-almost every point has a well defined stable direction in the center. 

Fix $f\in \PHb[2]$ an expanding on average diffeomorphism and suppose that $\Gap>0$. The arguments in the proof of \Cref{thm:existsweakbundle} are based in the proof of the two dimensional Osedelets Theorem. We point out  that, in the setting of Oseledets theorem where some invariant measure $\mu$ is used as a reference, one can use the Kingman's sub-addtive theorem in order to deduce that for $\mu-$almost every point $p$ the sequences $(\frac1n\log\norm{D^cf^n_p})_{n\geq 0}, (\frac1n\conorm{D^cf^n_p})_{n\geq 0}$ asymptotically behave like the largest and smallest center Lyapunov exponent, respectively. This information on the expected assymptotic expansion and contraction simplifies the construction of the corresponding invariant sub-bundles. In our case the main difficulty is that there is no a priori fixed invariant measure, hence we need to obtain some uniformity in the behavior of the above sequences just from the expanding on average property. Our goal is to show the following result.

\begin{theorem}\label{thm:existsweakbundle}
    Suppose that $f\in \PHb[2]$ has the expanding on average property and $\Gap>0$. Then, for every $p\in M$, for $\Leb^u$-almost every point $x$ in $\Wu{p}$, there is a well defined stable direction $E^{-c}(x)$ and 
    \[
    \displaystyle \limsup_{n\to +\infty} \frac{1}{n} \log \|D_xf^n|_{E^{-c}(x)}\| \leq -\Gap.
    \]
    Moreover, for $\Leb^u-$almost every $x\in \Wu{p}$, if $v\notin E^{-c}(x),$ then
    \[
    \liminf_{n\to +\infty} \frac{1}{n} \log \|D_xf^n v\| >0.
    \]
\end{theorem}

\begin{notation}
   For this section only, given $p\in M, n\in\Nat$ we write
   \begin{align*}
   &A^n(p)=D_p^cf^n.
   \end{align*}
\end{notation}

From now on, let us write $\Wuloc{p}$ to be the strong unstable manifold of size $1$ (this choice is arbitrary and we could have fixed any other constant). We now establish some finer control on the center derivative.  It will be convenient to denote, for $p\in M$, $m^u_p$ the normalized Lebesgue measure along $\Wuloc{p}$.

In what follows, we will first focus on the proof of \Cref{thm:existsweakbundle} by first assuming some lemmas which will be proved in the next subsection. 

\begin{lemma}\label{cor:minimaexp}
For any $C_1\in (0,\CEA)$,  for every $p\in M$ and unit vector $v\in E^c_p$, the following holds: for $m^u_p-$almost every $q\in \Wuloc{p}$ there exists $n(x,v)\in\Nat$ such that
\[
   \forall n\geq n(x,v),\ \norm{A^n(x)\cdot \chi_{x,p}(v)}\geq e^{C_1n}.
\] 
\end{lemma}

 In order to simplify the presentation, we assume that $E^c$ is a trivial bundle. In general, one considers trivializing charts of $E^{c}$ with null boundary, and obtains a measurable trivialization. Let us start the construction of stable directions.

Since we are assuming that $E^{c}$ is trivial, we can identify $\{A^n(x)\}_{n\geq 0}$ as a cocycle of matrices, for $x\in M$. Without loss of generality, we can assume that the norm used in $\Real^2$ is the Euclidean norm. Let $u_n(x), s_n(x)\in E^c_x\approx \Real^2$ be unit vectors so that 
\begin{align*}
&d^u_n(x):=\norm{A^n(x)}=\norm{A^n(x) \cdot u_n(x)}\\
&d^s_n(x):=\conorm{A^n(x)}=\norm{A^n(x) \cdot s_n(x)}.
\end{align*}
Write $A^n(x)=U_n(x)\cdot \diag\left(d_n^s(x), d_n^u(x)\right)\cdot V_n(x)$ with $U_n(x), V_n(x)\in\operatorname{O}(2)$: in particular, $u_n(x)\perp v_n(x)$ and $A^n(x)\cdot u_n(x)\perp A^n(x)\cdot v_n(x)$. If $s_n(x)\in E^c_x=\mathtt{span}\{u_{n+1}(x), s_{n+1}(x)\}$, we can write
\[
   s_n(x)=\cos\theta_n(x)s_{n+1}(x)+\sin\theta_n(x)u_{n+1}(x),\textrm{ where } \theta_n(x)=\angle(s_n(x), s_{n+1}(x)).
\]
Observe that
\begin{align*}
&\norm{A^{n+1}(x)\cdot s_n(x)}\geq |\sin\theta_n(x)| d_{n+1}^u(x)\\
&\norm{A^{n+1}(x)\cdot s_n(x)}=\norm{A(f^n(x))A^n(x)\cdot s_n(x)}\leq \norm{A(f^n(x))}\norm{A^n(x)\cdot s_n(x)}=\norm{A(f^n(x))} d^s_n(x).
\end{align*}
Observe that $d^u_{n+1}(x) \geq m(A(f^n(x))) d^u_n(x) \geq m_{\infty}(A) d^u_n(x),$ where $m_{\infty}(A) = \inf \{m(A(y))\}$. Similarly, define $\|A\|_\infty = \sup \{\|A(y)\|\}.$ Therefore,
\[
   |\sin\theta_n(x)|\leq \frac{d^s_n(x)}{d^u_{n+1}(x)}\norm{A(f^n(x))}=\frac{\norm{A(f^n(x))}}{\conorm{A(f^n(x))}}\frac{\conorm{A^n(x)}}{\norm{A^n(x)}}\leq \frac{\norm{A}_{\oo}}{m_\infty(A)}\frac{\conorm{A^n(x)}}{\norm{A^n(x)}}.
\]

By \Cref{lem:expansioncontraction2}, if $x$ is $m^u_p$-typical (that is, typical on its unstable leaf), for any $\beta>0$ then there exists $n_0 = n_0(x,\beta)$ such that for every $n\geq n_0$ we get
\begin{align*}
|\sin\theta_n(x)|\leq \exp\left(-n(\CEA+\Gap -\beta)\right)=\exp(-nK_\beta),
\end{align*}
where  $K_\beta= \CEA + \Gap - \beta>0$ is positive, and this implies that $\{s_n(x)\}$ converges to some vector $s(x)$. Since on $[0,\frac{\pi}{2}]$ we have that $\frac{2}{\pi} \theta \leq |\sin \theta|$, then
\[
|\theta_n(x)| \leq \frac{2\|A\|_\infty}{\pi m_\infty (A)} \cdot \frac{m(A^n(x))}{\|A^n(x)\|}.
\]
If $|\alpha_n(x)|=|\angle(s(x), s_n(x))|$, then
\[
\displaystyle |\alpha_n(x) | \leq  \sum_{k\geq 0} |\theta_{n+k}(x)| \leq \sum_{k\geq 0} \frac{2\|A\|_\infty}{\pi m_\infty (A)} \cdot \frac{m(A^{n+k}(x))}{\|A^{n+k}(x)\|}. 
\]
Since $|\theta| \geq |\sin \theta|$,
\[
   |\sin\alpha_n(x)|\leq \frac{2\|A\|_\infty}{\pi m_\infty(A)} \cdot  \sum_{k\geq 0}\frac{\conorm{A^{n+k}(x)}}{\norm{A^{n+k}(x)}}\leq \frac{2\|A\|_\infty}{\pi m_\infty(A)} e^{-n K_\beta},
\]
with $D(x)=\frac{C(x)}{1-e^{-K_\beta}}$ and $n\geq n_0(x,\beta)$. 

\begin{definition}\label{def:weakstablebundle}
   Define $E^{-c}_x=\Real\cdot s(x)$.
\end{definition}

Our goal is to show that for every $p$, for $ m^u_p$-almost every $x$, the direction $E^{-c}_x$ contracts exponentially fast under the action of $Df^n$. Define
\begin{equation}\label{eq.J}
J:= \sup_q \limsup_{n\to +\infty} \frac{1}{n} \log |\det A^n(q)|,
\end{equation}
and, since $\Gap>0$, we have that $\Gap = \CEA - J >0$. We will need the following lemma which will be proved in the next subsection. 

\begin{lemma}\label{lem:markovwithdet}
Assume that $f\in \PHb[2]$ has the expanding on average property and $\Gap>0$. For any $C_1\in (\max\{J,0\},\CEA)$ and any $C_2\in (J,C_1)$ the following holds: for any $p\in M$, and for any unit vector $v\in E^c(p)$, for $m^u_p$-almost every $x\in \Wuloc{p}$, there exists $n(x,v)\in \Nat$ such that for any $n\geq n(x,v)$, 
\[
\displaystyle \frac{|\det A^n(x)|}{\|A^n(x) \cdot \chi_{x,p}(v)\|} \leq e^{-(C_1 - C_2)n}.
\]
\end{lemma}

For the rest of this section, let us fix $C_1$ and $C_2$ as in \Cref{lem:markovwithdet}. Also, fix $v\in E^c(p)$ unit vector. 

For $x\in \Wuloc{p}$ and  $n\in \Nat$, write $s(x)=v^n_s+v^n_u$, where we write the vector $s(x)$ in the basis $\{u_{n}(x), s_{n}(x)\}$, whenever those directions are defined, and let $\alpha_n(x)=\sphericalangle(s(x), s_{n}(x))$. We have that
\begin{equation}
    \label{eq.decomposing}
    \|A^n(x) \cdot s(x)\| \leq |\cos(\alpha_n(x))|\|A^n(x) \cdot s_n(x)\| + |\sin(\alpha_n(x))| \|A^n(x) \cdot u_n(x)\|.
\end{equation}
Observe that for  $ m^u_p$-almost every $x$ there exists $n(x,v)$ such that if $n\geq n(x,v)$, then, 
\[
\norm{A^n(x)\cdot v^n_s}=|\cos(\alpha_n(x))|d_n^s(x)\leq e^{-(C_1-C_2)n}.
\]

Recall that,
\begin{align*}
|\sin\alpha_n(x) d^u_n(x)|\leq \frac{2\|A\|_\infty}{\pi m_\infty(A)}  \sum_{k\geq 0}\conorm{A^{n+k}(x)}\frac{\norm{A^{n}(x)}}{\norm{A^{n+k}(x)}}.
\end{align*}
To estimate this term, we need some type of lower bound on $\|A^{n+k}(x)\|$ for a $m^u_p$-typical point. In order to accomplish this, we need the following lemma.
\begin{lemma}\label{lem.subtemperedness}
For any $\beta>0$ small, for $m^u_p$-almost every $x$, there exists $Q_1(x)$ such that
\[
\displaystyle \frac{\|A^n(x)\|}{\|A^{n+k}(x)\|} \leq Q_1(x) e^{-(1-\beta)C_1 k + \beta n}.
\]
\end{lemma}
We will prove this lemma in the next subsection. \Cref{lem.subtemperedness} implies that

\begin{equation}
\label{eq.dunsin}
\displaystyle |\sin\alpha_n(x) d^u_n(x)|\leq \frac{2\|A\|_\infty}{\pi m_\infty(A)} \sum_{k\geq 0}\conorm{A^{n+k}(x)}\frac{\norm{A^{n}(x)}}{\norm{A^{n+k}(x)}} \leq \frac{2\|A\|_\infty}{\pi m_\infty(A)} \sum_{k \geq 0} m(A^{n+k}(x))Q_1(x)e^{-(1-\beta)C_1 k+ \beta n}
\end{equation}

Observe that
    \[
    \frac{|\det A^{n+k}(x)|}{m(A^{n+k}(x))} = \frac{|\det A^{n+k}(x)|}{\|A^{n+k}(x)\cdot s_{n+k}(x)\|} = \|A^{n+k}(x)\| \geq \|A^{n+k}(x)\cdot \chi_{x,p}(v)\|.
    \]
    Therefore,
    \begin{equation}\label{eq.ansn}
    \|A^{n+k}(x)\cdot s_{n+k}(x)\| \leq  \frac{|\det A^{n+k}(x)|}{\|A^{n+k}(x)\cdot \chi_{x,p}(v)\|}.
    \end{equation}
By \Cref{lem:markovwithdet}, for $n$ sufficiently large
\[
m(A^{n+k}(x)) \leq  e^{-(C_1-C_2)(n+k)}.
\]
Hence,
\[
|\sin\alpha_n(x) d^u_n(x)| \leq Q_1(x) \left(\sum_{k\geq 0} e^{-((C_1-C_2)+ (1-\beta)C_1 )k}\right) e^{-(C_1-C_2 -\beta)n}.
 \]
 Set $Q_2 = \sum_{k\geq 0} e^{-((C_1-C_2)+ (1-\beta)C_1 )k} $. For $n$ sufficiently large, we conclude that
 \[
\|A^n(x) \cdot s(p) \| \leq   e^{-(C_1-C_2) n} + Q_1(x)Q_2  e^{-(C_1-C_2-\beta)n}, 
\]
which decays exponentially fast to $0$ at a rate at least $C_1 - C_2-\beta$, where we take $\beta < C_2 - C_1$. So $E^{-c}_x$ is a genuine stable direction. We have showed that for any $C_1$ and $C_2$  as in the statement of \Cref{lem:markovwithdet}, for $ m^u_p$-almost every $x$ the stable direction $E^{-c}_x$ contracts at a rate bounded from above by $C_1-C_2-\beta$. As this happens in a full measure set, one can take a sequence $C_1^{m}$ converging to $\CEA$ from below, $C_2^m$ converging to $J$ from above, and $\beta_m>0$ converging to $0$. Considering the countable intersection of the full measure sets corresponding to $(C_1^m,C_2^m, \beta_m)$, we obtain a full measure set of points whose contraction rate is at least $-\Gap$. 

Let us check that for $m^u_p$-almost every $x$, $D_xf E^{-c}(x) = E^{-c}(f(x))$. Indeed, for $m^u_{f(p)}$-almost every $f(x)$, the direction generated by \[\frac{D_xf \chi_{x,p}(v)}{\|D_xf \chi_{x,p}(v)\|}\] expands exponentially fast. If $D_xf E^{-c}(x)$ does not coincide with $E^{-c}(f(x))$, then it would have a component in $\frac{D_xf \chi_{x,p}(v)}{\|D_xf \chi_{x,p}(v)\|}$ and it would expand in the future, which is a contradiction with the fact that it contracts. This argument also shows that for $m^u_p$-almost every $x$ if $v$ is a vector not contained in $E^{-c}(x)$, then it expanding in the future. This concludes the proof of \Cref{thm:existsweakbundle} assuming Lemmas \ref{cor:minimaexp}, \ref{lem:markovwithdet} and \ref{lem.subtemperedness}.

    \subsection{Proof of Lemmas \ref{cor:minimaexp}, \ref{lem:markovwithdet} and \ref{lem.subtemperedness}}

\begin{proof}[Proof of \Cref{cor:minimaexp}]
    Fix $C_1 \in (0, \CEA)$ and write $\varphi(x,[v]) = \log \|D_xf|_{[v]}\|$. Observe that $\varphi$ is a H\"older continuous function on $\hM$. Observe that \Cref{mainthm:TheoremD} implies that 
    \[
    \CEA = \inf \left\{\int \varphi \der \hmu : \hmu \textrm{ is a projective $u$-Gibbs measure}\right\},
    \]
    since the quantity above, under our assumptions of expanding on average and $\Gap>0$, coincides with the smallest positive center exponent among all $u$-Gibbs measures. 
    Fix $v\in E^c(p)$ a unit vector, by \Cref{prop.largedeviationsprojective}, the event
    \[
    \displaystyle E_n:= \left\{x\in \Wuloc{p}: \frac{1}{n} \sum_{i=0}^{n-1}\varphi(\hf^i(x,[\chi_{x,p}(v)]))\leq  C_1\right\} 
    \]
    verifies
    \[
    m^u_p\left(E_n\right) \leq L e^{-\gamma n}.
    \]
    By the Borel-Cantelli Lemma,
    \[
    m^u_p\left( \bigcap_{j \geq 1} \bigcup_{n \geq j} E_n \right) = 0.
    \]
    Therefore, for $m^u_p$-almost every $x$, there exists $n(x,v)$ such that for any $n \geq n(x,v)$,
    \[
    \displaystyle \frac{1}{n} \sum_{i=0}^{n-1} \varphi(\hf^i(x,[\chi_{x,p}(v)])) > C_1,
    \]
    which implies that
    \[
    \|D_xf^n \cdot \chi_{x,p}(v)\| \geq e^{C_1 n}. \qedhere
    \]

\end{proof}

\begin{proof}[Proof of \Cref{lem:markovwithdet}]
Let $C_1$, $C_2$, $p\in M$ and $v\in E^c(p)$ be as in the statement of \Cref{lem:markovwithdet}. By \Cref{cor:minimaexp}, for $m^u_p$-almost every $x$, there exists $n_1 = n_1(x,v)$ such that 
\[
\|A^n(x) \cdot \chi_{x,p}(v)\| \geq e^{C_1 n},
\]
for $n\geq n_1$. 

\begin{claim}
 For $C_2> J,$   there exists a uniform $n_2 = n_2(C_2)$ such that for any $q\in M$, if $n\geq n_2$ then
\[
|\det A^n(q)|\leq e^{C_2 n}.
\]
\end{claim}
\begin{proof}
    Observe that \eqref{eq.J} implies that for any invariant measure $\mu$,
    \[
    \displaystyle \int \log |\det A(q)| \der \mu(q) \leq J. 
    \]
    Since the quantity $\frac{1}{n} \log |\det A^n(q)| = \frac{1}{n} \sum_{j=0}^{n-1} \log |\det A(f^j(q))|$ is a Birkhoff average of a continuous function, for any $C_2>J$, we conclude that there exists $n_2 = n_2(C_2)$ such that for any $n\geq n_2$ and any $q\in M$, 
    \[
    \frac{1}{n} \log |\det A^n(q)| < C_2,
    \]
    which proves the claim. \qedhere
\end{proof}
For $n\geq \max\{n_1,n_2\}$, we have that
\[
\displaystyle \frac{|\det A^n(x)|}{\|A^n(x) \cdot \chi_{x,p}(v)\|} \leq e^{-(C_1 - C_2) n}.\qedhere
\]
    
\end{proof}

Next we start the preparation to prove \Cref{lem.subtemperedness}. We will need the following covering lemma.

\begin{lemma}\label{lemma.covering}
    Let $f\in \PH[2]$ be a partially hyperbolic diffeomorphism. There exists a constant $B$ with the following property: for any $D$ unstable disk and any $n\in \mathbb{N}$, there is a cover of $f^n(D)$, $\{D_i\}$ such that
    \begin{enumerate}
        \item each $D_i$ is an unstable disk of size $1$, and
        \item each $D_i$ intersects at most $B$  disks of this cover. 
    \end{enumerate}
\end{lemma}
\begin{proof}
    Let $D$ be an unstable disk and $n\in \Nat$. For $f^n(D)$, let $\{x_1, \cdots , x_m\}$ be a maximal $1$-separated set in $f^n(D)$,where we are using the topology generated by the Riemannian metric restricted to  $f^n(D)$. Observe that, by the maximality of $\{x_1, \cdots, x_m\}$, we have
    \[
    f^n(D) \subset \displaystyle \bigcup_{i=1}^m \Wu{x_i,1}.
    \]
    Moreover, if $i\neq j$ then $\Wu{x_i,\frac{1}{2}} \cap \Wu{x_j,\frac{1}{2}} = \emptyset$ and if $\Wu{x_i,1} \cap \Wu{x_j,1} \neq \emptyset$, then, by the triangular inequality, 
    \[
    \Wu{x_j,\frac{1}{2}} \subset \Wu{x_i,\frac{5}{2}}, 
    \]
    since if $z\in \Wu{x_j,\frac{1}{2}}$, then $d^u(z,x_i) \leq d^u(z,x_j) + d^u(x_j,x_i) < \frac{1}{2} + 2.$
    By the uniform control on the geometry of unstable disks of size $\frac{1}{2}$ and of size $\frac{5}{2}$, we have that there are constants $0<a_1<a_2$ such that
    \[
    a_1 \leq \Leb^u(\Wu{x,\frac{1}{2}}) \textrm{ and } \Leb^u(\Wu{y,\frac{5}{2}} )\leq a_2,
    \]
    for any $x,y \in M$. For each $i$, let $B(i)$ be the number of $j$ such that $\Wu{x_i,1} \cap \Wu{x_j,1} \neq \emptyset$. Then,
    \[
\displaystyle     B(i) a_1 \leq \sum_{j: \Wu{x_j,1} \cap \Wu{x_i,1} \neq \emptyset} \Leb^u(\Wu{x_j,\frac{1}{2}}) \leq \Leb^u(\Wu{x_i,\frac{5}{2}} \leq a_2. 
    \]
    Therefore, by setting $B = \frac{a_2}{a_1} + 1$, we get that $B(i) \leq B$. Set $D_i = \Wu{x_i,1}$. Observe that this bound is independent of $D$ and $n$. 
    \end{proof}

    \begin{proof}[Proof of \Cref{lem.subtemperedness}]
       For the function $\varphi(x,[v]) = \log \|D_x f|_{[v]}\|$ and $\varepsilon = \CEA - C_1$, let $\gamma$ be given by \Cref{prop.largedeviationsprojective}. Fix $\beta>0$ small  and fix $\delta>0$ small enough such that
        \begin{equation}\label{eq.choiceofdelta}
        \delta \left((1-\beta) C_1 - \log m_\infty(A) \right) < \gamma.
        \end{equation}
        
        Let $k,n\in \mathbb{N}$, take $p\in M$, write $D= \Wuloc{p}$ and let $m^u_p$ be the normalized volume measure on $D$. Our goal is to estimate
        \[
        \displaystyle \int_D \frac{\|A^n(x)\|^\delta}{\|A^{n+k}(x)\|^\delta}\der m^u_p(x).
        \]
        First, let us fix $\{D_i\}_{i\in \mathcal{I}_n}$ a cover of $f^n(D)$ obtained from \Cref{lemma.covering}. For $i\in \mathcal{I}_n$, write $m^u_i(\cdot) = m^u_{p}(f^{-n}(D_i) \cap \cdot)$. Observe that $m^u_i$ is not a probability measure. Therefore,
        \[
        m^u_{p} \leq B \sum_{i\in \mathcal{I}_n} m^u_i.
        \]
        Since $x\mapsto \frac{\|A^n(x)\|^\delta}{\|A^{n+k}(x)\|^\delta}$ is a nonnegative function, we have that
        \[
        \displaystyle \int_D \frac{\|A^n(x)\|^\delta}{\|A^{n+k}(x)\|^\delta}\der m^u_p(x)\leq B \sum_{i\in \mathcal{I}_n} \int_{f^{-n}(D_i)} \frac{\|A^n(x)\|^\delta}{\|A^{n+k}(x)\|^\delta}\der m^u_i(x) 
        \]
        For each $i\in \mathcal{I}_n$, let us estimate
        \[
        \int_{f^{-n}(D_i)} \frac{\|A^n(x)\|^\delta}{\|A^{n+k}(x)\|^\delta}\der m^u_i(x).
        \]

        \begin{claim}
            \label{claim.approxholvec}
            There exists a constant $L_1 \geq 1$ such that for any two points $x,y$ in the same unstable leaf and $n\in \Nat$ such that $d^u(f^j(x),f^j(y)) < 2$ for $0\leq j \leq n$, let $v_x\in E^c(x)$ be a unit vector that $\|A^n(x)v_x\| = \|A^n(x)\|$, then
            \[
            1\leq \frac{\|A^n(y)\|}{\|A^n(y)\cdot \chi_{y,x}(v_x)\|} \leq L_1.
            \]
        \end{claim}
        \begin{proof}
            Since the unstable distance is bounded by $2$, then, there exists a constant $L \geq 1$ such that $L^{-1} \leq \|h^u_{p,q}\| \leq L$, for any $p,q$ in the same unstable and with unstable distance at most $2$. Hence,
            \[
            \begin{array}{lll}
                L^{-1}\|A^n(x)\| \leq \|h^u_{f^n(y),f^n(x)}(A^n(x)  \cdot v_x)\| &= & \|h^u_{y,x}(v_x)\| \cdot \|A^n(y) \cdot \chi_{y,x}(v_x)\| \leq L \|A^n(y)\| \\
                 & \leq & L \|A^n(y) \cdot v_y \| \leq L^2 \|h^u_{f^n(x),f^n(y)}(A^n(y) \cdot v_y)\|\\
                 & = & L^2\|A^n(x) \cdot h^u_{x,y} (v_y)\| \leq L^3 \|A^n(x)\|,
            \end{array}
            \]
            where $v_y$ is a unit vector in $E^c(y)$ that realizes norm of $A^n(y)$. The above implies that 
            \[
            1\leq \frac{\|A^n(y)\|}{\|A^n(y) \cdot \chi_{y,x}(v_x)\|}.
            \]
            Moreover, it implies that 
            \[
            L^{-2} \leq \frac{\|A^n(y)\|}{\|A^n(x)\|} \leq L^2,
            \]
            and that
            \[
            L^{-2} \|A^n(x)\| \leq \|A^n(y) \cdot \chi_{y,x}(v_x)\|.
            \]
            Therefore,
            \[
            L^{-4}\|A^n(y)\| \leq \|A^n(y) \cdot \chi_{y,x}(v_x)\|.
            \]
            Take $L_1 = L^4$ and this concludes the proof of the claim. 
        \end{proof}

        For each $i\in \mathcal{I}_n$, let $x^i$ be the point in $f^{-n}(D_i)$ such that $f^n(x^i)$ is the center of the disk $D_i$. Observe that, for any $x\in f^{-n}(D_i)$, by \Cref{claim.approxholvec},
        \[
        \|A^{n+k}(x)\| \geq L_1^{-1} \|A^k(f^n(x)) \cdot \chi^u_{f^n(x),f^n(x^i)}(v_i^n)\| \|A^n(x)\|,
        \]
        where $v_i^n = \frac{A^n(x^i) \cdot v_i}{\|A^n(x^i) \cdot v_i\|},$ and $v_i$ is a unit vector such that $\|A^n(x^i)\| = \|A^n(x^i) \cdot v_i\|.$  Write $J^u_{-n}(y) = \det D_yf^{-n}|_{E^u}$. By the change of variables formula,
        \[
        \begin{array}{l}
        \displaystyle \int_{f^{-n}(D_i)} \frac{\|A^n(x)\|^\delta}{\|A^{n+k}(x)\|^\delta}\der m^u_i(x)\leq  L_1^2\int_{f^{-n}(D_i)}\frac{\|A^n(x)\|^\delta}{\|A^k(f^n(x))\cdot \chi^u_{f^n(x),f^n(x^i)}(v_i^n)\|^\delta \|A^n(x)\|^\delta} \der m^u_i(x)\\
        \displaystyle = L_1^2 \int_{D_i}\|A^k(y) \cdot \chi^u_{y,f^n(x^i)}(v_i^n)\|^{-\delta} |J^u_{-n}(y)| \der \Leb^u(y).
        \end{array}
        \]
        By bounded distortion, there exists a constant $L_2 \geq 1$ such that for any $x,y \in D_i$,
        \[
        L_2^{-1} \leq \frac{|J^u_{-n}(y)|}{|J^u_{-n}(x)|} \leq L_2.
        \]
        
        Hence,
        \[
        \begin{array}{l}
      \displaystyle  L_1^2 \int_{D_i}\|A^k(y) \cdot \chi^u_{y,f^n(x^i)}(v_i^n)\|^{-\delta}|J^u_{-n}(y)| \der \Leb^u(y) \leq\\ 
      \displaystyle L_1^2 L_2 \left(\min_{x\in D_i} |J^u_{-n}(x)|\right) \int_{D_i}\|A^k(y) \cdot \chi^u_{y,f^n(x^i)}(v_i^n)\|^{-\delta} \der \Leb^u(y).
        \end{array}
        \]
        Since $D_i$ is a disk of size $1$ inside an unstable foliation, we can find a universal constant $L_3 \geq 1$ such that for any $q\in M$, $L_3^{-1} \leq \Leb^u(\Wuloc{q}) \leq L_3$. In particular,
        \[
        L_3^{-1} \Leb^u(D_i \cap \cdot) \leq m^u_{f^n(x^i)}(\cdot) \leq L_3 \Leb^u(D_i \cap \cdot),
        \]
        where $m^u_{f^n(x^i)}$ is the normalized Lebesgue measure on $D_i$.
        
        Thus,
        \[
        \begin{array}{l}
        \displaystyle L_1^2 L_2 \left(\min_{x\in D_i} |J^u_{-n}(x)|\right) \int_{D_i}\|A^k(y) \cdot \chi^u_{y,f^n(x^i)}(v_i^n)\|^{-\delta}  \der \Leb^u(y) \leq \\
        \displaystyle 
        L_1^2 L_2 L_3 \left(\min_{x\in D_i} |J^u_{-n}(x)|\right) \int_{D_i}\|A^k(y) \cdot \chi^u_{y,f^n(x^i)}(v_i^n)\|^{-\delta}  \der m^u_{f^n(x_i)}(y).
        \end{array}
        \]

        \begin{claim}
            For any $x\in M$, 
            \[
            \displaystyle \int_{\Wuloc{(y)}} \|A^k(y)\|^{-\delta} \der m^u_x(y) \leq e^{-C_1 \delta k} + L e^{-( \gamma + \delta \log m_\infty(A)) k},
            \]
            where $\gamma$ and $L$ are given by \Cref{prop.largedeviationsprojective}.
        \end{claim}
        \begin{proof}
            Fix $v\in E^c(x)$ a unit vector and write 
            \[
            H_k:= \{y\in \Wuloc{x}: \|A^k(y)\cdot \chi_{y,x}(v)\|>e^{C_1 k}\}.
            \]
            By \Cref{prop.largedeviationsprojective},
            \[
            m^u_x \left( H_k^c\right) \leq L e^{-\gamma k}.
            \]
            Since 
            \[
            \|A^k(y)\| \geq \|A^k(y) \cdot \chi_{y,x}(v)\|,
            \]
            where $v$ is any unit vector in $E^c(x)$, we obtain
          \[
          \begin{aligned}
          \int_{\Wuloc{(y)}} \|A^k(y)\|^{-\delta} \der m^u_x(y)
          &=
          \int_{H_k} \|A^k(y)\|^{-\delta} \der m^u_x(y)
          +
          \int_{H_k^c} \|A^k(y)\|^{-\delta} \der m^u_x(y)
          \\
          &\leq
          e^{-C_1\delta k}
          +
          L e^{-(\gamma + \delta\log m_\infty(A))k}.
          \end{aligned}
        \qedhere
        \]
        \end{proof}
   
    We conclude that
    \[
    \displaystyle \int_{f^{-n}(D_i)} \frac{\|A^n(x)\|^\delta}{\|A^{n+k}(x)\|^\delta}\der m^u_i(x) \leq L_1^2 L_2 L_3 \left(\min_{x\in D_i} |J^u_{-n}(x)|\right)  \left(e^{-C_1\delta k }+ L e^{-(\gamma + \delta \log m_\infty(A))k}\right).
    \]

    We have that
    \[
    \begin{array}{l}
\displaystyle          B \sum_{i\in \mathcal{I}_n} \int_{f^{-n}(D_i)} \frac{\|A^n(x)\|^\delta}{\|A^{n+k}(x)\|^\delta}\der m^u_i(x) \leq BL_1^2 L_2 L_3\left(e^{-C_1\delta k }+ L e^{-(\gamma + \delta \log m_\infty(A))k}\right)\sum_{i\in \mathcal{I}_n}\left(\min_{x\in D_i} |J^u_{-n}(x)|\right)\\
\displaystyle \leq  BL_1^2 L_2^2 L_3^2 \left(e^{-C_1\delta k }+ L e^{-(\gamma + \delta \log m_\infty(A))k}\right)\sum_{i\in \mathcal{I}_n} m^u_i(D_i).
    \end{array}
    \]
    We devide the disks $D_i$ into $\mathcal{G}_n:=\{i\in \mathcal{I}_n: D_i \subset f^n(D)\}$ and $\mathcal{B}_n:=\mathcal{I}_n / \mathcal{G}_n$. For $i\in \mathcal{B}_n$, $f^{-n}(D_i)$ is contained in a neighborhood $V_n$ of the boundary of $D$ which is exponentially small. In particular, $\Leb^u(V_n)$ decays exponentially fast with $n$. Hence,
\[
\displaystyle \sum_{i\in \mathcal{I}_n} m^u_i(f^{-n}(D_i)) = \sum_{i\in \mathcal{G}_n} m^u_i(f^{-n}(D_i)) + \sum_{i\in \mathcal{B}_n} m^u_i(f^{-n}(D_i)) \leq B \left(m^u_p(D) + \Leb^u(V_n)\right) \leq 2B,
\]
for $n$ sufficiently large. Write $K = 2B^2(L_1 L_2 L_3)^2$, then
\[
\displaystyle \int_D \left(\frac{\|A^n(x)\|}{\|A^{n+k}(x)\|}\right)^\delta \der m^u_p(x) \leq K\left(e^{-C_1\delta k }+ L e^{-(\gamma + \delta \log m_\infty(A))k}\right). 
\]
Applying Markov inequality,
\[
m^u_p\left(x:\left(\frac{\|A^n(x)\|}{\|A^{n+k}(x)\|}\right)^\delta >e^{-(1-\beta)C_1 \delta k+ \beta\delta n} \right) \leq Ke^{-\beta \delta n}\left(e^{-\beta C_1 \delta k} + L e^{(-\gamma + \delta((1-\beta)C_1 - \log m_\infty(A)))k}\right). 
\]
By \eqref{eq.choiceofdelta}, the right hand side decays exponentially fast with $n$ and $k$. The conclusion of the result follows from a simple application of the Borel-Cantelli Lemma. 
    \end{proof}

    \subsection{Accumulations of the weak stable direction}

Let us end this section by studying possible accumulation measures of the weak stable direction. Our main goal is to prove the following subsection.

\begin{proposition}
\label{prop.accumulationprojectiveus}
Let \(f\in \PHb[2]\) satisfy the expanding on average property and suppose that
\(\Gap>0\). Then, for every \(p\in M\), there exists $\Lambda$ a subset of $\Wuloc{p}$ of full leaf volume such that the following property hold for every
\(q\in\Lambda\):  Define the  sequence of projective  measures
    \[
    \hmu^-_n(q)
    :=
    \frac{1}{n}
    \sum_{j=0}^{n-1}
    \delta_{\left(f^j(q),
    \left[A^j(q)\bigl(E^{-c}(q)\bigr)\right]\right)}.
    \]
    Let $\hmu$ be an accumulation point of
    $\{\hmu^-_n(q)\}_{n\geq 1}$ and let $\mu = \hpi_*\hmu$, then
    \[
    \hmu
    =
    \hmu^-.
    \]
    In other words, $\hmu$ is concentrated on the 
    \(E^{-c}\)-directions.

\end{proposition}
\begin{proof}
Fix a unit vector $v\in E^c(p)$ and, observe that, \Cref{cor.accugibbspt} implies the existence of a set $\Lambda_1 = \Lambda_1(p,v)$ such that for any $q\in \Lambda_1$, any accumulation point of the sequence of projective measures
\[
\hmu_n(q,v) = \displaystyle \frac{1}{n} \sum_{j=0}^{n-1} \delta_{(f^j(q), [A^j(q)\chi_{q,p}(v)]}
\]
is a projective $u$-Gibbs measure.  Let $\Lambda_2 \subset \Wuloc{p}$ be the set of points we obtained in \Cref{thm:existsweakbundle}. Set $\Lambda = \Lambda_1 \cap \Lambda_2$. Observe that $\Lambda$ has full leaf volume. Fix $q\in \Lambda$ and, up to considering a subsequence, let us suppose that $\hmu_n^-(q)$ converges to a projective measure $\hmu$. We have that $\hpi_* \hmu$ is a $u$-Gibbs measure $\mu$.  Up to considering another subsequence, we may suppose that $\hmu_n(q,v)$ converges to the projective $u$-Gibbs measure $\hmu^+$. 

Suppose that $\hmu$ is not supported on stable directions, then, there exists a function $x \mapsto \alpha(x) \in [0,1]$ such that
\[
\hmu = \displaystyle \int \alpha(x) \delta_{E^{+c}(x)} + (1-\alpha(x)) \delta_{E^{-c}(x)} \der \mu(x),
\]
where $\alpha(x)$ is strictly positive in a set of positive $\mu$-measure. By \eqref{eq.decomposing}, \eqref{eq.dunsin} and \eqref{eq.ansn}, for any $\beta>0$, we have that 
\[
\|A^n(q)|_{E^{-c}(q)}\| \leq \mathrm{Constant}(q) \cdot \frac{|\det A^n(q)|}{\|A^n(q) \cdot \chi_{q,p}(v)\|} e^{2\beta n}. 
\]
This inequality implies that
\[
\displaystyle \frac{1}{n} \sum_{j=0}^{n-1} \varphi\left(\hf^j(q,[E^{-c}(q)])\right)\leq \frac{1}{n}\left( \sum_{j=0}^{n-1}\log|\det A(f^j(q))| - \varphi\left( \hf^j(q), [\chi_{q,p}(v)]\right) \right) + 2\beta +\frac{1}{n} \log \mathrm{Constant(q)},
\]
where $\varphi(q,[v]) = \log |A(q)|_{[v]}\|$. Since the functions $\varphi$ and $x\mapsto \log|\det A(x)|$ are continuous, indeed H\"older continuous, by taking the limit in $n$, we obtain
\[
\displaystyle \int \varphi \der \hmu \leq \int \log |\det A|\der \mu - \int \varphi \der \hmu^+ +2\beta.
\]
Therefore,
\[
\displaystyle 2\beta + \int \lambda^{+c}(x) + \lambda^{-c}(x) \der \mu(x) \geq \int \left(\alpha(x) \lambda^{+c}(x) + (1-\alpha(x)) \lambda^{-c}(x) \right) \der \mu(x) + \int \lambda^{+c}(x) \der \mu(x).
\]
Hence,
\[
\displaystyle 2\beta + \int \alpha(x) \lambda^{-c}(x) \der \mu(x) \geq \int \alpha(x) \lambda^{+c}(x) \der \mu(x).
\]
By taking $\beta \to 0$, we obtain that 
\[
\displaystyle  \int \alpha(x) \lambda^{-c}(x) \der \mu(x) \geq \int \alpha(x) \lambda^{+c}(x) \der \mu(x).
\]
Since $\lambda^{+c}(x) > 0 > \lambda^{-c}(x)$ for $\mu$-almost every $x$, the inequality above is only possible if $\alpha(x) = 0$ for $\mu$-almost every $x$. Therefore, $\hmu = \hmu^-$. \qedhere

\end{proof}

\begin{lemma}    \label{lem.projconvgotherdirections}

Let $f\in \PHb[2]$ be expanding on average with $\Gap>0$.  There exists a set $G$ of full volume such that for any  $x\in G$ has a well define weak stable direction $E^{-c}(x)$, and the following properties hold:
    \begin{enumerate}
    \item for any $v\notin E^{-c}(x)$, any accumulation point of the sequence of projective measures
    \[
    \displaystyle \frac{1}{n} \sum_{j=0}^{n-1} \delta_{\hf^j(x,[v])}
    \]
    is a projective $u$-Gibbs measure;
    \item any accumulation point $\hmu$ of the sequence
    \[
    \displaystyle \frac{1}{n}\sum_{j=0}^{n-1} \delta_{\hf(x,[E^{-c}(x)]}
    \]
    verifies that $\hmu = \hmu^-,$ where $\hmu^-$ is the lift of the measure $\mu = \hpi_* \hmu$ to the stable Oseledets direction. 
    \end{enumerate}
\end{lemma}
    \begin{proof}
    Fix $\{U_i\}$ a finite cover of $M$ by foliated charts for $\Fu$. On each $U_i$, fix $D^{cs}_i$ a disk transverse to $\Fu$ and fix a continuous unit vector field $V_i:E^{cs}_i \to E^c$, where for each $q\in D^{cs}_i$, the vector $V_i(q)$ belongs to $E^c(q)$. For each $i$, applying \Cref{cor.accugibbspt} and \Cref{prop.accumulationprojectiveus} we obtain a set $G_i\subset U_i$ of full $\Leb$-measure such that for any $x\in G_i$, any accumulation point of the sequence
    \begin{equation}\label{eq.accvi}
    \displaystyle \frac{1}{n} \sum_{j=0}^{n-1} \delta_{\hf^j(x, [\chi_{x,q}(V_i(q))])},
    \end{equation}
    where $q$ is the unique point in $\Wuloc{x}$  that intersects $D^{cs}_i$. Moreover, since $E^{-c}(x)$ exists and contracts exponentially fast, for any vector $v\notin E^{-c}(x)$, the accumulation points of the sequence
    \[
    \displaystyle \frac{1}{n} \sum_{j=0}^{n-1} \delta_{\hf(x,[v])}
    \]
    coincides with the accumulation points of the sequence \eqref{eq.accvi}. This completes the proof of the first part.  The second part is a direct application of \Cref{prop.accumulationprojectiveus}. \qedhere

\end{proof}

\section{Uniform Pesin blocks and finiteness of SRB measures}\label{Sec:uniformpesinblock}

The main goal of this section is to prove the following result.
\begin{theorem}\label{THM:finiteness}
Suppose that $f\in \PHb[2]$ has the expanding on average property and $\Gap>0$. Then $f$ has at most finitely many SRB measures. 
\end{theorem}
This result is a consequence of the continuity of the center Lyapunov exponents proved in \Cref{mainthm:TheoremB}, as we will explain. Let us recall a few things from Pesin theory. 

\begin{definition}\label{Pesin.block} Let $f$ be a $C^r$ diffeomorphism. Given $C>0$, $\chi>0$ and $\varepsilon>0$, we denote $\Delta_{C,\chi,\varepsilon}$ the Pesin block defined by the following property: for every $y\in \Delta_{C,\chi,\varepsilon}$ and $k\in \mathbb{Z}$, there is a decomposition $T_{f^k(y)}M=E^-(f^k(y))\oplus E^+(f^k(y))$ such that for any $n\geq 0$, $$\max \left( \|Df^n\vert_{E^-(f^k(y)}\|, \|Df^{-n}\vert_{E^+(f^{k}(y))} \| \right)\leq C\, e^{-\chi n + \varepsilon |k|}.$$
\end{definition}

Since, in our setting, we are studying partially hyperbolic diffeomorphisms with two dimensional center, we only need to consider the definition above restricted to the center direction.  The key proposition here is the following:

\begin{proposition}\label{prop.uniform.pesin}
   Let $f\in \PHb$ be expanding on average property with $\Gap >0$.  Then, for every $\delta>0$, there exist constants $C, \chi, \varepsilon>0$ such that 
   \[
   \mu(\Delta_{C,\chi, \varepsilon})>1-\delta,
   \]
   for every $u$-Gibbs measure $\mu$. 
\end{proposition}

Using \Cref{mainthm:TheoremB}, \Cref{prop.uniform.pesin} is a direct consequence of the following proposition.

\begin{proposition}\label{prop.generaluniformpesin}
Let $f: M \to M$ be a $C^1$-diffeomorphism. Suppose that $\mathcal{K}$ is a set of invariant measures with the following properties:
\begin{enumerate}
\item $\mathcal{K}$ is compact for the weak*-topology;
\item there exists an integer $i\in \{0, \cdots,  d)\}$,  where $d$ is the dimension of $M$, such that for every  $\nu \in \mathcal{K}$ and for $\nu$-almost every $x$, the ergodic component of $\nu$ containing $x$, $\nu_x$,  is a hyperbolic measure with  $i$ positive Lyapunov exponents;
\item the function  $\mathcal{K} \ni \nu \mapsto \lambda_1( \nu) + \cdots + \lambda_i(\nu)$ is continuous, where $\lambda_1(\nu) \geq \lambda_2(\nu) \cdots \geq \lambda_d(\nu)$ are the Lyapunov exponents of $\nu$;
\item there exists $\chi>0$ such that for any $\nu \in \mathcal{K}$, we have that for $\nu$-almost every $p$, $\lambda_1(p), \cdots, \lambda_d(p) \notin [-\chi, \chi]$.
\end{enumerate}
Then, for every $\delta \in (0,1)$ there exists constants $C, \chi, \varepsilon$ such that 
\[
\nu(\Delta_{C,\chi, \varepsilon}) > 1-\delta, \textrm{ for every $\nu \in \mathcal{K}$. }
\]

\end{proposition}

\begin{proof}[Proof of \Cref{prop.uniform.pesin}]
    Set $\mathcal{K} = \Gibbs$, this set is compact for the weak*-topology. The assumptions that $f$ is expanding on average and $\Gap>0$ implies that any $\mu \in \Gibbs$ has $\dim(E^u) + 1$ positive Lyapunov exponents, \Cref{cor:Gibbshyperbolic}. \Cref{mainthm:TheoremB} implies the continuity required in the third item. The last item is also a consequence of the continuity of the sum of exponents. The last item is also a consequence of \Cref{cor:Gibbshyperbolic}, since all of the Lyapunov exponents are not contained in $[-\Gap, \CEA].$
\end{proof}
Observe that by the second assumption in $\mathcal{K}$ for any measure $\nu \in \mathcal{K}$ its positive Lyapunov exponents are given by $\lambda_1(\nu), \cdots, \lambda_{d-i}(\nu)$. 

The proof of \Cref{prop.generaluniformpesin} is essentially contained in the computations done in Sections $2$ and $3$ of \cite{BCSSPR}.  For completeness, we will explain these computations. We will need the following results from \cite{BCSSPR}. 

\begin{lemma}[\cite{BCSSPR} Lemma $2.18$]\label{lem.pliss}
Let $T$ be an invertible measure preserving map on a probability space $(\Omega, \mathcal{B},\upsilon)$. Suppose that $\varphi \in L^{\infty}(\upsilon)$, $L\in \mathbb{R}$, $\kappa \in [0,1]$, and $\upsilon\{x\in \Omega: \varphi(x) >L\} \leq \kappa$. Then, for any $\beta <L$,
\[
\upsilon\{x\in \Omega:  \varphi_j(x) \geq \beta j, \textrm{ for all }j\geq 0\} \geq \frac{\int \varphi d\upsilon - \beta - \kappa \|\varphi - L\|_{\infty}}{L-\beta},
\]
where $\varphi_j(x) = \varphi(x) + \cdots+ \varphi(T^{j-1}(x))$. 
\end{lemma}

We will also use the proposition below.

\begin{proposition}[\cite{BCSSPR} Proposition $2.21$]\label{prop.hyptimestopesinblock}
Let $f$ be a $C^r$ diffeomorphism. Given $N_0 \geq 1$ and $\chi >0$,  $P_{N_0} = P_{N_0}(\chi)$ be defined as 
\[
P_{N_0} := \left\{x: \begin{array}{l} \textrm{ there is a splitting $T_xM = E_x \oplus F_x$ s.t. $\forall j\geq 0$}\\ \|Df^{jN_0}(x)|_{E}\| \leq e^{-\chi j N_0} \textrm{ and } \|Df^{-jN_0}(x)|_F\| \leq e^{-\chi j N_0} \end{array} \right\}.
\]
Write $C = C(f, N_0):= \max_{k\in [1, N_0]} \max\{\|Df^k\|, \|Df^{-k}\|\}.$Then,  for every $\varepsilon >0$, the Pesin block $\Delta_{C,\chi, \varepsilon}$ verifies
\[
\upsilon(M - \Delta_{C,\chi, \varepsilon}) \leq 4 \max \left\{ \frac{c_0}{\varepsilon}, 1 \right\} \upsilon(M - P_{N_0}), \textrm{ for all invariant measures $\upsilon$},
\]
where $c_0 = \chi + \max\{\log\|Df\|, \log \|Df^{-1}\|\}.$
\end{proposition}

\begin{proof}[Proof of \Cref{prop.generaluniformpesin}]
Fix $\nu\in \mathcal{K}$.  Let us do some estimates first for $(f,\nu)$.  We will then use the continuity of the Lyapunov exponents to ``spread'' these estimates in a neighborhood of $\nu$. We can then use the compactness of $\mathcal{K}$ to conclude the proof.

For a $\nu$-typical point, let  $E_x^-$ be the smallest subspace containing all of the Oseledets directions corresponding to negative Lyapunov exponents, and let $E_x^+$ be the smallest subspace containing all of the Oseledets  directions corresponding to the positive Lyapunov exponents.  Since every measure $\nu$ in $\mathcal{K}$ does not have any exponent in the interval $[-\chi, \chi]$,  for each $N\in \mathbb{N}$, define \[
K_{N} := \{ x\in M : \|Df^{N}_x|_{E_x^-}\| < e^{-\chi N} \textrm{ and } \|Df^{-N}_x|_{E_x^+}\| < e^{-\chi N}\}.
\]
The following claim is a direct consequence of the hyperbolicity of $\nu$.
\begin{claim}\label{claim.largemeasureK}
For every $\delta>0$, if $N$ is sufficiently large, then
\[
\nu(K_N) > 1-\delta.
\]
\end{claim}

\begin{claim}\label{claim.largemeasureP}
Let $\chi' = \frac{\chi}{4}. $ For every $\rho>0$, there exists $\delta>0$,  such that if $\nu(K_N) >1-\delta$, for some $N\in \mathbb{N}$,  then $\nu(P_N(\chi')) > 1-\rho$. 
\end{claim}
\begin{proof}
Suppose that $K_N$ has measure larger than $1-\delta$. We will see later what is the dependence of the measure of $P_N$ on $\delta$.   Consider the function $\psi(x) = \mathds{1}_{K_N}(x)$, where $\mathds{1}_{K_N}$ is the characteristic function on the set $K_N$.  Let $\psi^+$ be the forward Birkhoff average of $\psi$.

Consider 
\begin{equation}\label{eq.H}
H = \max \{\|Df\|, \|Df^{-1}\|\}
\end{equation}
and take $\beta = \max\{ 1- \frac{\chi}{2 \log H}, \frac{3}{4}\}$.  Take $L=1$, $\kappa =0$ and $T= f^N$.  By Lemma \ref{lem.pliss},  we obtain that 
\[
\nu\left\{x: \psi_j(x) \geq \beta j, \textrm{ for all }j\geq 0\right\} \geq \displaystyle \frac{\int \psi d\nu - \beta}{1-\beta} \geq \frac{1-\delta - \beta}{1-\beta} = 1- \frac{\delta}{1-\beta}.
\]
Write $K^+_N:=\left\{x:\psi_j(x) \geq \beta j, \textrm{ for all }j\geq 0\right\} $.  For $x\in M$,  $j\in \mathbb{N}$, set $G(x,j,N):= \{i\in [0, j-1]: f^{iN}(x) \in K_N\}$. If $x\in K^+_N$, we have that for every $j\in \mathbb{N}$, 
\[
\displaystyle \frac{\# G(x,j,N)}{j} \geq \beta.
\]
For such $x$, we obtain
\[
\begin{array}{ll}
&\|Df^{jN}_x|_{E^-}\|\leq \\
&\displaystyle \left(\prod_{i\in G(x,j,N)} \|Df^N_{f^{iN}(x)}|_{E^-}\|\right)\left( \prod_{i\in (G(x,j,N))^c \cap [0,j-1]} \|Df^N_{f^{iN}(x)}|_{E^-}\|\right)\\
& = I . II
\end{array}
\]
From the definition of $K_N$, we have
\[
I \leq e^{-\chi N \# G(x,j,N)} \leq e^{-\chi  N \beta j},
\]
and 
\[
II \leq e^{N \log H (j - \# G(x,j,N))}  \leq e^{N \log H j(1-\beta)}.
\]
From the definition of $\beta$, we obtain 
\[
-\chi  N \beta j + N \log H j (1-\beta)  \leq -\frac{3 \chi }{4} Nj +  \frac{\chi}{2}Nj = -\frac{\chi}{4} Nj.
\]
Therefore, 
\[
\|Df^{jN}_x|_{E^-}\| \leq e^{-\frac{\chi }{4} jN},
\]
for every $j\in \mathbb{N}$, whenever $x\in K^+_N$.  Define similarly $K^-_N$ for past iterates and $E^+$ instead of $E^-$. By the same computation,  we conclude that 
\[
\nu(K^-_N) >1 - \frac{\delta}{1-\beta}.
\]
Thus,
\[
\nu(K^-_N \cap K^+_N )  = 1- \nu((K^-_N)^c \cup (K^+_N)^c) \geq 1 - \frac{2\delta}{1-\beta}.
\]
Given $\rho>0$, take $\delta = \frac{\rho(1-\beta)}{2}$ and take $N$ large enough such that $\nu(K_N) > 1- \delta$. Thus, $\nu(P_N(\chi')) > 1-\rho$. 
\end{proof}

Just as above, set $\chi' = \frac{\chi}{4}$. Fix $\varepsilon>0$ and $\theta>0$.  Set $L' = 4\max\left\{\frac{c_0}{\varepsilon}, 1\right\}$, where $c_0 = \chi' + \max\{\log \|Df\|, \log \|Df^{-1}\|\}$. Take $\rho < \frac{\theta}{L'}$. By Claims \ref{claim.largemeasureK} and \ref{claim.largemeasureP}, we fix $\delta>0$ small, and $N$ large such that $\nu(K_N) > 1-\delta$, which implies that $\nu(P_N(\chi')) > 1-\rho$.  By Proposition \ref{prop.hyptimestopesinblock}, 
\[
\nu(\Delta_{C,\chi', \varepsilon}) > 1-\theta,
\]
where $C = C(f,N,\nu)$ is given in the statement of that proposition.

Consider $\mathbb{G}_i$ and $\mathbb{G}_{d-i}$ to be the Grassmaniann bundles of $i$-planes and $d-i$-planes.  Consider $\widetilde{K}_N^i$ to be the subset of $\mathbb{G}_i$ defined by points of the form $(x,E_x^+)$, where $x\in K_N$. Consider $\widetilde{K}_N^{d-i}$ to be the subset of $\mathbb{G}_{d-i}$ of points of the form $(x, E_x^-)$, where $x\in K_N$.   Observe that the condition defining $K_N$ is open in $\mathbb{G}_i$ and $\mathbb{G}_{d-i}$ in the sense that there are two open sets $\widetilde{U}_i \subset \mathbb{G}_i$ and $\widetilde{U}_{d-i}\subset \mathbb{G}_{d-i}$, such that 
\[
\widetilde{K}_N^i \subset \widetilde{U}_i \textrm{ and } \widetilde{K}_N^{d-i} \subset \widetilde{U}_{d-i},
\] and these open sets have the property that if $(x,E) \in \widetilde{U}_i$, then $\|Df^{-N}_x|_{E}\| < e^{-\chi' N}$. Similar conclusion holds for $\widetilde{U}_{d-i}$ and future iterates. 

On $\mathbb{G}_i$, consider the lift of $\nu$ given by $\hat{\nu}^+ := \int \delta_{(x,E^+_x)} d\nu(x)$. This gives a probability measure defined in $\mathbb{G}_i$ which is invariant under the action induced by $f$.  We can define similarly the stable lift $\hat{\nu}^-$ which is a probability measure in $\mathbb{G}_{d-i}$. 
Lemma $3.10$ in \cite{BCSSPR} implies that the function $\nu \mapsto \lambda_1(\nu) + \cdots + \lambda_i(\nu)$ is continuous, if and only if the function $\nu \mapsto \hat{\nu}^+$ is continuous in the weak*-topology. Similarly, the function $\nu \mapsto \lambda_{i+1}(\nu) + \cdots + \lambda_d(\nu)$
is continuous if and only if the function $\nu \mapsto \hat{\nu}^-$ is continuous. 

By assumption, the function $\nu \mapsto \lambda_1(\nu) + \cdots + \lambda_i(\nu)$ is continuous, when restricted to $\mathcal{K}$. Therefore, $\nu \mapsto \hat{\nu}^+$ is also continuous. Since 
\[
\displaystyle \int \log |\det D_xf| \der \nu = \lambda_1(\nu) +\cdots + \lambda_d(\nu),
\]
and the function $\nu \mapsto \int \log |\det D_xf| \der \nu(x)$ is continuous, continuity of the sum of the top $i$ exponents implies continuity  $\nu \mapsto \lambda_{i+1}(\nu) + \cdots + \lambda_d(\nu)$.
 Therefore, for every $\nu'\in \mathcal{K}$ sufficiently close to $\nu$, we have that 
\[
\hat{\nu'}^+(U_i)>1-\frac{\delta}{2} \textrm{ and } \hat{\nu'}^-(U_{d-i}) > 1-\frac{\delta}{2},
\] which implies that $\nu'(K_N) > 1-\delta$.  The rest of the estimates holds in a uniform way, since $N$, $\varepsilon$ and $\theta$ are fixed. Therefore, $\nu'(\Delta_{C,\chi', \varepsilon})>1-\theta$.  This implies that every measure $\nu \in \mathcal{K}$ has an open neighborhood such that for every $\nu'$ in this neighborhood will see the same Pesin block in a uniform way. To conclude the proof use the compactness of $\mathcal{K}$ and cover $\mathcal{K}$ by finitely many open sets where we can apply the estimates above. Then, we take $C$ to be the maximum constant obtained in each open set. Thus, we found a Pesin block that works for every measure in $\mathcal{K}$. 
\end{proof}

\begin{proof}[Proof of \Cref{THM:finiteness}]
Fix $\delta>0$. By \Cref{prop.uniform.pesin}, there exists a Pesin block $\Delta_{C, \chi, \varepsilon}$ such that 
\[
\mu(\Delta_{C,\chi, \varepsilon})> 1-\delta,
\]
for every $u$-Gibbs measure $\mu$. Observe that for $\rho>0$ small, any two points $x,y \in \Delta_{C,\chi,\varepsilon}$ with $d(x,y)< \rho$ are homoclinically related (see \Cref{def:homoclinicrelation}). This follows easily from the uniform estimates one has for the invariant manifolds on Pesin blocks.

Suppose there are infinitely many different ergodic SRB measures $(\mu_n)_{n\in \Nat}$. Up to considering a subsequence, we may suppose that they converge to a $u$-Gibbs measure $\mu$. Consider $\Delta =  \mathrm{supp}(\mu|_{\Delta_{C,\chi,\varepsilon}})$. Cover $\Delta$ by finitely many $U_i$, which are open balls with radius $\frac{\rho}{2}$. Set $U= \bigcup_i U_i$ and observe that 
\[
\mu(U) > 1-\delta.
\]
For every $n$ sufficiently large $\mu_n(U\cap \Delta_{C,\chi,\varepsilon})>1-2\delta$. Since there are finitely many balls $U_i$, by the pigeonhole principle, there exists $i$ such that for $\mu_n(U_i \cap \Delta_{C,\chi,\varepsilon})>0$ for infinitely many $n$s. Let $n$ and $j$ be such that $n\neq j$ and $\mu_*(U_i\cap \Delta_{C,\chi,\varepsilon})>0$ for $*=n,j$. Therefore, the measures $\mu_n$ and $\mu_j$ are homoclinically related. 

Observe that two SRB measures that are homoclinically related are the same by a Hopf argument (see for instance \cite{HHTUSRB}). This is a contradiction with our assumption that $\mu_n$ is different than $\mu_j$ if $n\neq j$. \qedhere
\end{proof}

\section{Existence of SRB measures} 
\label{sec:existenceSRB}

The goal of this section is to show the following result.
\begin{theorem}\label{thm.existenceSRB}
Suppose that $f\in \PHb[2]$ has the expanding on average property and satisfies $\Gap>0$. Then $f$ has an SRB measure.
\end{theorem}

To show this result, we will use the work of Climenhaga-Pesin \cite{Climenhaga_Pesin} and Climenhaga-Dolgopyat--Pesin \cite{CDP} on effective hyperbolicity.  In those works, they introduce the notion of effective hyperbolicity and show the existence of SRB measures under the assumption that the set of effective hyperbolic points has positive volume. It is convenient for us to work with a slightly different definition of effective hyperbolicity as we will explain below.

\subsection{Effective hyperbolicity and the existence of SRB measures}\label{subsetction.effectivehyperbolicity}

Let $f$ be a $C^2$-diffeomorphism.   Given a point $x\in M$ and $E, F$, subspaces of $T_xM$, define
\begin{eqnarray}
\lambda^-(x,E):= \sup \{\log \|D_xfv\|: v\in E, \|v\|=1\}\\
\lambda^+(x,F):= \inf \{\log \|D_xfv\|: v\in F, \|v\|=1\}.
\end{eqnarray}
Moreover, define
\[
\tilde{\Delta}(x,E,F) := \max\{0, \lambda^-(x,E) - \lambda^+(x,F)\}.
\]
Define
\[
\theta(x,E,F):= \inf \{\sphericalangle(v,w): v\in E \textrm{ and } w\in F\}.
\]
In what follows, for each $n$, we write $E_n = D_xf^n E$ and $F_n = D_xf^n F$.  
For $\lambda>0$, define
\[
\displaystyle \Gamma^e_\lambda(x, E,F):=\left\{ n\in \mathbb{N}: \sum_{j=k}^{n-1} \lambda^+(f^j(x),F_j) - \tilde{\Delta}(f^j(x),E_j,F_j) \geq \lambda(n-k), \textrm{ for all } 0\leq k < n\right\}. 
\]
For constants $C\geq 1$ and $ \lambda>0$ and $m\in \mathbb{N}$, we also define
\[
\displaystyle \Gamma^s_{(C,\lambda),m}(x,E):= \left\{ n\in \mathbb{N}: \|D_{f^n(x)}f^{-k} v\| \geq C e^{\lambda k}\|v\|, \textrm{ for all } 0 \leq k \leq m  \textrm{ and } v\in E_n\right\}. 
\]
We say that $(x,E,F)$ verifies the (modified) \emph{effective hyperbolic} conditions if 
\begin{itemize}
\item[\textbf{(EH1)}] $\displaystyle \lim_{m\to +\infty} \underline{\delta}\left( \Gamma^e_\lambda(x,E,F) \cap \Gamma^s_{(C,\lambda),m}(x,E)\right)>0$, and 
\item[\textbf{(EH2)}] $\displaystyle \lim_{\overline{\theta} \to 0} \overline{\delta}\{n\in \mathbb{N}: \theta(f^n(x),E_n,F_n) < \overline{\theta}\} = 0,$
\end{itemize}
where given $\Lambda \subset \mathbb{N}$, a subset of the natural numbers,  the quantities $\overline{\delta}(\Lambda)$ and $\underline{\delta}(\Lambda)$ are the upper and lower densities of $\Lambda$ inside the natural numbers, respectively.  The following theorem is a variation of the main theorem from \cite{CDP}.

\begin{theorem}\label{thm.CDPmain}
Let $f:M\to M$ be a $C^2$-diffeomorphism.  Suppose there is a disk $D$ such that there is a set $D' \subset D$ with the following properties:
\begin{enumerate}
\item for $x\in D'$, there is a subspace $E(x)$ such that $T_xM = E(x) \oplus T_xD$ such that  \textbf{(EH1)} and \textbf{(EH2)}  for $(x,E(x), T_xD)$ holds for some constants $C = C(x) \geq 1$ and $\lambda = \lambda(x)>0$;
\item $m_D(D') >0$, where $m_D$ is the normalized Lebesgue measure on $D$ induced by the Riemannian metric of $M$. 
\end{enumerate}
Then $f$ has an SRB measure. 
\end{theorem}
Let us briefly explain the idea of the proof of \Cref{thm.CDPmain}. The proof follows the argument of the proof of Theorem $4.2$ in Climenhaga-Dolgopyat-Pesin \cite{CDP}, using the Hadamard-Perron Theorem from Climenhaga-Pesin \cite{Climenhaga_Pesin} for the  splittings
\[
T_{f^n(x)}M=E_n\oplus F_n,
\qquad
E_n=D_xf^nE(x),
\qquad
F_n=D_xf^nT_xD,
\]
instead of working with measurable cone fields. Let us explain their result using the notation in our partially hyperbolic setting. Let us first recall the geometric mechanism in the construction of SRB measures from  \cite{CDP}. Let $D$ be a
$\dim(E^u)+1$-dimensional submanifold and let $m_D$ denote its induced Riemannian
normalized volume. Suppose that the set $S\subset D$ of
effectively hyperbolic points has positive $m_D$-measure. The usual Pesin-Sinai's construction of SRB measures consider the Ces\`aro averages
\[
\mu_n:=\frac{1}{n}\sum_{j=0}^{n-1}f_*^j m_D.
\]
Every weak*- accumulation point of this sequence is an invariant measure. The main issue
is that effective hyperbolicity does not give uniform control of expansion and distortion of the whole
measure $f_*^j m_D$. The central step is therefore to extract from
$\mu_n$ a ``subeasure'' $\nu_n$ of uniformly positive mass whose geometry and
densities are uniformly controlled. 

After restricting to a positive-measure subset of $S$, the constants in the
definition of effective hyperbolicity can be taken uniformly. Every point in
this subset has a positive density of effective hyperbolic times. If $j$ is
an effective hyperbolic time for $x\in S$, the
Hadamard--Perron Theorem from \cite{Climenhaga_Pesin} produces a neighborhood
\[
W_j^x\subset D\cap B\left(x,\bar r e^{-\bar\lambda j}\right)
\]
such that
\begin{equation}\label{eq.djx}
D_j^x:=f^j(W_j^x)
\end{equation}
is a disk of fixed size centered at \(f^j(x)\) that contracts uniformly for $j$-iterates in the past: there are constants \(C>0\) and \(\bar\lambda>0\), independent of
\(x\) and \(k\), such that
\[
d\bigl(f^{-k}(y),f^{-k}(z)\bigr)
\leq
C e^{-\bar\lambda k}d(y,z)
\]
for every $y,z\in D_j^x$ and \(0\leq k\leq j\). In addition, there is a good control on  curvature. Thus, $D_j^x$ behaves like a local unstable manifold for finitely many backward iterates.

Using the geometrical control of these finite time unstable disks, one obtains bounded distortion, in the sense that
\[
 f^j_* \left(m_D|_{W_j^x}\right)(\cdot) = \alpha_{x,j}\rho_j^x(\cdot) \der m_{D_j^x}(\cdot), 
\]
where $\alpha_{x,j}= m_D(W^x_j)>0$,  $m_{D_j^x}$ is the normalized volume on that disk and  the density $\rho_j^x$ has a uniform $\log$-H\"older norm and satisfies uniform estimates
\begin{equation}\label{eq.density}
L^{-1}\leq \rho_j^x\leq L,
\end{equation}
for some constant $L\geq 1$.
This implies that each pushed-forward piece is represented by a standard pair:
leaf volume on a finite-time unstable disk, with a density bounded above and
below and having uniformly bounded $\log$-H\"older norm. To obtain contraction estimates
in the transverse direction as well, the disks are selected so that a fixed
proportion of each of them consists of points for which the linear dynamics along a
complementary direction is controlled for a certain number of iterates quantified by a parameter $q\in \mathbb{N}$. Writing $\mathbf J = (C, \overline{\lambda})$, for each $n\in \Nat$, $0 \leq q \leq n$, they consider the set $S_{\mathbf J, q, n}$, which is the set of points in $D$, where 
\[
n \in \Gamma_{\overline{\lambda}}^e(x,(E^s \oplus E^{-c})(x), T_xD) \cap \Gamma^s_{\mathbf J, q}(x, (E^s \oplus E^{-c})(x)).
\]
 These are the set of points that have good control on expansion estimates, and contraction estimates from time $n$ to $n-q$. Define $S_{\mathbf J , q, n} = \emptyset$ if $q >n$. They show the existence of a $\delta>0$ such that for each $q\in \Nat$, there exists $N(q)$ such that if $n\geq N(q)$ then 
 \[
 \frac{1}{n} \sum_{j=0}^{n-1} m_D(S_{\mathbf J,q,j}) > \delta,
 \]
 see equation $(5.19)$ from \cite{CDP}. 

There exist  $\beta'',\beta'>0$ such that for $q$ and $n \geq N(q)$ as above, there exists a finite set $A_n \subset S_{\mathbf J, q, n}$ with the following properties:
\begin{enumerate}
    \item For $x\in A_n$, let $W_n^x$ be as in \eqref{eq.djx}. Then,
    \[
    m_D(W_n^x \cap S_{\mathbf J, q, n}) \geq \beta' m_D(W_n^x).
    \]
    \item For $x,y \in A_n$ with $x\neq y$, $W_n^x \cap W_n^y = \emptyset$.
    \item Set $W_{\mathbf J, q, n} = \bigsqcup_{x\in A_n} W_n^x$. Then
    \[
    \frac{1}{n} \sum_{j=0}^{n-1}m_D(W_{\mathbf J, q ,j}) \geq \beta''>0.
    \]
\end{enumerate}
This is established in the proof of Proposition $4.8$ from \cite{CDP} and it uses a covering argument (Lemma $4.6$). Essentially, for each $q$ fixed, $j \geq N(q)$, one will decompose the measure $f^j_*m_D$ as a convex combination of measures supported on $D_j^x$, for $x\in W_{\mathbf J, q,j}$ and the rest. The measures coming from $D_j^x$ will have uniform control on density and are supported on disks of uniform size and are finite time unstable disks with good contraction control, up to time q, transversaly. This is where the SRB measures are going to come from. 

They choose a sequence of  times $q_n \to +\infty$ that increases slowly, with  $n\geq N(q_n)$, such that if we set
\begin{equation}\label{eq.nun}
\nu_n = \displaystyle \frac{1}{n} \sum_{j=q_n}^{n-1}f^j_* \left(m_D|_{W_{\mathbf J,q_n,j}}\right), 
\end{equation}
then the total mass of $\nu_n$ is bounded from below by $\frac{1}{2} \beta''.$

The construction gives
\[
\nu_n\leq\mu_n.
\]
 Therefore, $\nu_n$ is a portion of $\mu_n$ which is decomposed into the convex combination of absolutely continuous measures on finite time unstable disks, which are pairwise disjoints for each $j$, with controlled densities. No comparable control
is required for the complementary part $\mu_n-\nu_n$.

Up to considering a subsequence, we have that
\[
\mu_n\longrightarrow\mu
\qquad\text{and}\qquad
\nu_n\longrightarrow\nu
\]
in the weak* topology. The measure $\mu$ is invariant, and
\[
0\neq\nu\leq\mu.
\]
The measure $\nu$ does not have to be invariant, but it has positive mass and the SRB property (its disintegration along unstable disks is absolutely continuous), and it is part of an invariant measure. So $\mu$ has a portion with the SRB property and, in particular, there are SRB measures. The fact that $q_n\to\infty$ is used to show that these SRB measures are hyperbolic, they have a contracting transverse directions. 

Climenhaga-Dolgopyat-Pesin formulate their argument using measurable
invariant cone families. However, the cone families are only used to invoke the
effective Hadamard-Perron Theorem from \cite{Climenhaga_Pesin} and obtain uniformly controlled admissible
disks at effective hyperbolic times. Since the analogous Hadamard-Perron
result for invariant splittings is established in \cite{Climenhaga_Pesin}(Theorems A and D), the same proof applies to the splittings $E_n\oplus F_n$ considered above and gives \Cref{thm.CDPmain}.  The goal of the rest of the section is to prove \Cref{thm.existenceSRB} by verifying the conditions of \Cref{thm.CDPmain}. 

\subsection{Verifying effective hyperbolicity}

For the rest of this section, let $f\in \PHb[2]$ verify the expanding on average property and $\Gap>0$. Let $G$ be the set obtained in \Cref{lem.projconvgotherdirections}. Let $D$ be any $\dim(E^u) + 1$ smooth disk  such that there exists a set $D' \subset (D\cap G)$ with positive $m_D$-measure, such that for $x\in D'$, $T_xD \cap \left( E^s \oplus E^{-c}\right)(x) = \{0\}$. Such disks always exists, as we could fix a Lusin set where $x\mapsto E^{-c}(x)$ is continous, and then using a Fubini argument with a smooth family of $\dim(E^u) +1$-dimensional disks, we can build such disks in a neighborhood of a point density point $p$ of this Lusin set. 

For $x\in D \cap G$, write $F'(x) = T_xD$, and for each $l\in \Nat$, write $F_l'(x) = D_xf^l F'(x)$. We are interested in understanding the expanding behavior of $F'(x)$. For points in $D',$ we know that any vector $v\in F'(x) - (E^s \oplus E^{c})(x)$ converges to the direction $E^u$, they will see uniform expansion after a few iterates. So, to understand the expanding behavior of the derivative acting on $F'(x)$ we actually need to understand the behavior of $F(x) = F'(x) \cap E^c(x)$, which is a one dimensional space in $E^c(x)$. For any $n, N\in \mathbb{N}$ and $x\in D \cap G$, define  $\varphi_N^u(f^{nN}(x))$ and $\varphi_N^s(f^{nN}(x))$ defined by 
\begin{align*}
\varphi_N^u(f^{nN}(x)) &\coloneqq \log m\left(D_{f^{nN}(x)}f^N|_{F_{nN(x)}}\right) \textrm{ and } \varphi_N^s(f^{nN}(x))\\
& = -\log\|D_{f^{nN}(x)}f^{N}|_{(E^s \oplus E^{-c})(f^{nN}(x))}\|.
\end{align*}

Observe that 
\[
-\log\|D_{f^{nN}(x)}f^{N}|_{(E^s \oplus E^{-c})(f^{nN}(x))}\| = -\log\|D_{f^{nN}(x)}f^{N}|_{E^{-c}(f^{nN}(x))}\|.
\]

Define
\[
\tilde{\Delta}_N(x):= \max\{0, -\varphi_N^s(x) - \varphi^u_N(x)\}. 
\]
Define also
\[
\displaystyle \widehat{\Gamma}^e_{\lambda,N}(x):= \left\{n\in \mathbb{N}: \sum_{j=k}^{n-1} (\varphi^u_N - \tilde{\Delta}_N)(f^{jN}(x)) \geq \lambda(n-k), \textrm{ for all }0 \leq k < n\right\}, 
\]
and for each $q\in \mathbb{N}$, define
\[
\displaystyle \widehat{\Gamma}^s_{\lambda,q,N}(x):= \left\{ n\in \mathbb{N}: \sum_{j=n-k}^{n-1} \varphi^s_N(f^{jN}(x)) \geq \lambda (n-k),  \textrm{ for all } 0 \leq k \leq q\right\}. 
\]

Our goal is the prove the following proposition. 

\begin{proposition}\label{prop.goodestimatesEH}
Let $f\in \PHb[2]$ have the expanding on average property and $\Gap>0$.  There exist constants $\delta_0>0$, $\lambda>0$ and $N\in \mathbb{N}$,  such that for any  $x\in D'$ and  any $q\in \mathbb{N}$ large enough
\[
\displaystyle \underline{\delta}\left(\widehat{\Gamma}^e_{\lambda, N}(x) \cap \widehat{\Gamma}^s_{\lambda, q, N}(x) \right) > \delta_0. 
\]
\end{proposition}

In preparation to prove \Cref{prop.goodestimatesEH}, we will need the following version of Pliss Lemma.

\begin{lemma}\label{lem.plissfinite}
Given $L \geq \alpha_2 > \alpha_1 >0$, let $\kappa = \frac{\alpha_2 - \alpha_1}{L- \alpha_1}$. Then, given any real numbers $a_1, \cdots, a_Q$ such that
\[
\displaystyle \sum_{j=1}^M a_j \geq \alpha_2 Q \textrm{ and } a_j \leq L \textrm{ for every }1\leq j \leq Q,
\]
there are $l > \kappa Q$ and $1 < n_1 < \cdots< n_l \leq Q$ so that
\[
\displaystyle \sum_{j=n+1}^{n_i} a_j \geq \alpha_1(n_i - n) \textrm{ for every } 0 \leq n < n_i \textrm{ and }i=1, \cdots, l.
\]
\end{lemma}

Let us proceed to the proof of \Cref{prop.goodestimatesEH}.

\begin{proof}[Proof of \Cref{prop.goodestimatesEH}]
Fix $0< \chi < \min \{\CEA, \Gap\}$ and fix
\[
0<\rho < \frac{1}{10} \cdot \frac{\chi}{4 ( \chi + 6\log H)},\] where $H$ is given by \eqref{eq.H}. Fix  $N\in \mathbb{N}$ large enough so that for any $u$-Gibbs measure $\mu$, $\mu(K^*_N)>1-\rho$,  for $*=-,+$, where
\[
\widehat{K}^+_{N} := \{ (y,[v])\in \hM :  \|D_yf^{N}|_{E}\| \geq e^{\chi N}\},
\]
and 
\[
\widehat K^-_N:= \{(y,[v]) \in \hM: \|D_yf^{N}|_{E}\| \leq e^{-\chi N} \}.
\]
This is possible because of the continuity of Lyapunov exponents among $u$-Gibbs measures as we saw in \Cref{Sec:uniformpesinblock}. It is easy to see that these sets are compact. Let $\mathcal{U}^+$ be small open neighborhoods of $\widehat{K}_N^+$ with the property that for any $(y,[v]) \in \mathcal{U}^+$, we have
\[
\|D_yf^N|_{[v]}\| > e^{\frac{\chi}{2} N}.
\]
Define in an analogous way $\mathcal{U}^-$ an open neighborhood of $\widehat{K}_N^-$.  For each $n\in \mathbb{N}$ consider the sequence of projective measures given by
\[
\displaystyle \nu_n^+(x):= \frac{1}{n} \sum_{j=0}^{n-1} \delta_{\hf^{jN}(x,[F(x)])} \textrm{ and } \nu_n^-(x):= \frac{1}{n} \sum_{j=0}^{n-1} \delta_{\hf^{jN}(x, E^{-c}(x))}.
\]

Since in $D'$, $F(x)$ is not tangent to $E^{-c}$, by \Cref{lem.projconvgotherdirections}, for any $x \in D'$, any accumulation measure of $\nu_n^+(x)$ is a projective $u$-Gibbs, in particular, supported on unstable directions. Moreover, we know that any accumulation measure of $\nu_n^-(x)$ is supported on stable directions. Therefore, for $n$ large, we have that
\[
\nu^*_n(x)(\mathcal{U}^*) > 1-\rho, \textrm{ for }*=s,u. 
\]

Let us first work with $\nu_n^+(x)$.  Take $L = 1$, $a_j = \mathds{1}_{\mathcal{U}^+}(\hf^{jN}(x,[F(x)]))$, $\alpha_2 = 1-\rho$ and $\alpha_1 = 1-\beta$, for $\beta=10 \rho$.  Observe that 
\[
\kappa = \frac{\alpha_2 - \alpha_1}{L-\alpha_1} = \frac{\beta - \rho}{\beta} = 1-\frac{\rho}{\beta} = \frac{9}{10}.
\]
By \Cref{lem.plissfinite}, we obtain times $1< n_1^+ < \cdots< n_l^+ \leq n$, with $l > \kappa n$ such that
\[
\displaystyle \frac{1}{n^+_i-k}\sum_{j=k}^{n_i-1}\mathds{1}_{\mathcal{U}^+}(\hf^{jN}(x,[F(x)])) \geq 1-\beta. 
\]
Write $I^+_n(x):= \{n_1^+, \cdots, n_l^+\}$.  Similarly,  by \Cref{lem.plissfinite} applied to $\nu_n^-(x)$, we find Pliss times $1 < n_1^- < \cdots < n_{l'}^- \leq n$ such that $l' > \kappa n$ and 
\[
\displaystyle \frac{1}{n^-_i-k}\sum_{j=k}^{n_i-1}\mathds{1}_{\mathcal{U}^-}(\hf^{jN}(x,E^{-c}(x))) \geq 1-\beta. 
\]
Write $I^-_n(x) := \{n_1^-, \cdots, n_{l'}^-\}$.  Define $I_n(x) := I_n^+(x) \cap I_n^-(x)$. Observe that 
\[
\# I_n \geq \left(1-\frac{2\rho}{\beta}\right) n = \frac{8}{10} n.
\]
Take  $\beta = 10\rho$.

\begin{claim}\label{claim.gammae}
Each $n_i\in I_n(x)$ belongs to $\widehat{\Gamma}^e_{\frac{\chi}{4}N,N}(x)$.
\end{claim}
\begin{proof}
Fix $0 \leq k < n_i$ and write 
\[G_{n_i}(k):= \{j\in [k,n_i-1]: \hf^{jN}(x,F(x)) \in \mathcal{U}^+ \textrm{ and } \hf^{jN}(x,E^{-c}(x)) \in \mathcal{U}^-\}.\]
By the definition of the Pliss times we are considering, we have that
\[
\# G_{n_i}(k) \geq (1-2\beta)(n_i - k).
\]
Observe that for any $j\in G_{n_i}(k)$, we have that $\tilde{\Delta}_N(f^{jN}(x)) = 0$.  For any other  $j$, we have $\tilde{\Delta}_N(f^{jN}(x)) \leq  2 N \log H$.  Therefore,
\[
\begin{array}{lll}
\displaystyle \sum_{j=k}^{n_i-1} \tilde{\Delta}_N(f^{jN}(x)) & =& \displaystyle  \left( \sum_{j\in G_{n_i}(k)} \tilde{\Delta}_N(f^{jN}(x)) \right) + \left(\sum_{j\in [k,n_i-1] - G_{n_i}(k)} \tilde{\Delta}_N(f^{jN}(x)) \right) \\
& \leq & \displaystyle 4 \beta (n_i-k)N \log H. 
\end{array} 
\]
We have
\[
\begin{array}{lll}
\displaystyle \sum_{j=k}^{n_i-1} (\varphi_N^u - \tilde{\Delta}_N) (f^{jN}(x))& = & \displaystyle  \sum_{j=k}^{n_i-1} \varphi_N^u(f^{jN}(x)) - \sum_{j=k}^{n_i-1} \tilde{\Delta}_N(f^{jN}(x))\\
& \geq & \displaystyle  \left( \sum_{j\in G_{n_i}(k)} \varphi^u_N(f^{jN}(x)) \right) + \left(\sum_{j\in [k,n_i-1] - G_{n_i}(k)} \varphi^u_N(f^{jN}(x)) \right) \\
& & \displaystyle - 4\beta (n_i - k)N \log H \\
&\geq & \displaystyle \# G_{n_i}(k) N \frac{\chi}{2} - (n_i-k -\# G_{n_i}(k)) N \log H - 4\beta(n_i-k) N \log H \\
& \geq & \displaystyle \left((1-2\beta) \frac{\chi}{2} - 6\beta \log H \right) N (n_i - k) 
\end{array}
\]
By our choice of $\beta$, we have that 
\[
 \left((1-2\beta) \frac{\chi}{2} - 6\beta \log H \right) > \frac{\chi}{4}.
\]
We conclude that for any $0 \leq k < n_i$
\[
\displaystyle \sum_{j=k}^{n_i-1} (\varphi_N^u - \tilde{\Delta}_N) (f^{jN}(x)) \geq \frac{\chi}{4} N (n_i - k). \qedhere
\]

\end{proof}

Fix $q \in \mathbb{N}$ and consider $\mathcal{I}_n(x) = I^-_n(x) \cap (q+1, n]$.  Observe that, 
\[
\#\left( I_n^-(x)\cap [0, q] \right)  \leq q. 
\]
Therefore,
\[
\frac{\#\mathcal{I}_n(x)}{n} \geq (1-\frac{\rho}{\beta})  - \frac{q}{n}.
\]
In particular, for each $q\in \mathbb{N}$, if $n$ is large enough so that $\frac{q}{n} < \frac{1}{10}$, we obtain
\[
\frac{\# \mathcal{I}_n(x)}{n} \geq \frac{8}{10}.
\]

\begin{claim}
Each $n_i \in \mathcal{I}_n(x)$ belongs to $\widehat{\Gamma}^s_{\frac{\chi}{4}N, q, N}(x)$.
\end{claim}
\begin{proof}
The proof is similar to the computations that appeared in the proof of Claim \ref{claim.gammae}. 
\end{proof}

In particular, for each $q \in \mathbb{N}$ and for each $n$ large enough, we have that
\[
\#\left(\widehat{\Gamma}^e_{\frac{\chi}{4}N, N}(x) \cap \widehat{\Gamma}^s_{\frac{\chi}{4}N, q, N}(x) \right) \geq \frac{6}{10}n. 
\]
Hence,
\[
\underline{\delta}\left(\widehat{\Gamma}^e_{\frac{\chi}{4}N, N}(x) \cap \widehat{\Gamma}^s_{\frac{\chi}{4}N, q, N}(x) \right) \geq \frac{6}{10}.
\]

\end{proof}

\begin{proof}[Proof of \Cref{thm.existenceSRB}]
Let $D$ be a $\dim(E^u)+1$ smooth disk as before, such that there is a set $D' \subset (D \cap G)$, where $G$ is the set obtained in \Cref{lem.projconvgotherdirections}, with the property that for any $x\in D'$, $T_xD \cap (E^s \oplus E^{-c})(x) = \{0\}$.

Let $\delta_0, \lambda>0$ and $N\in \mathbb{N}$ be given by \Cref{prop.goodestimatesEH}.  Let $g= f^N$ and consider the sets $\Gamma^e_\lambda(x,(E^s \oplus E^{-c})(x), F(x))$ and $\Gamma^s_{(1,\lambda,q)}(x,(E^s \oplus E^{-c})(x))$ defined using $g$ instead of $f$. Notice that 
\[
\widehat{\Gamma}^e_{\lambda,N}(x) \subset \Gamma^e_\lambda(x,(E^s \oplus E^{-c})(x), F(x)) \textrm{ and } \widehat{\Gamma}^s_{\lambda,q,N}(x) \subset \Gamma^s_{(1,\lambda,q)}(x, (E^s\oplus E^{-c})(x)).
\]
\Cref{prop.goodestimatesEH} implies that for any $x\in D'$, 
\[
\underline{\delta} \left( \Gamma^e_\lambda(x,(E^s \oplus E^{-c})(x),F(x)) \cap \Gamma^s_{(1,\lambda,q)}(x,(E^s \oplus E^{-c})(x)) \right) > \delta_0,
\]
for all $q$ sufficiently large.  Therefore, 
\[
\displaystyle \lim_{q\to +\infty}\underline{\delta} \left( \Gamma^e_\lambda(x,(E^s\oplus E^{-c})(x),F(x)) \cap \Gamma^s_{(1,\lambda,q)}(x,(E^s \oplus E^{-c})(x)) \right)  \geq \delta_0.
\]

Recall the definition of a Pesin block given in \Cref{Pesin.block}. For each $\delta>0$, by \Cref{prop.uniform.pesin} there is a Pesin block $\Delta_{C,\chi, \varepsilon}$ with the property that $\mu(\Delta_{C,\chi,\varepsilon})>1-\delta$ for every $u$-Gibbs measure $\mu$. Take $\overline{\theta}>0$ small such that for any $x\in \Delta_{C,\chi,\varepsilon}$, the angle between $E^{-c}(x)$ and $E^{+c}(x)$ is greater than $2\overline{\theta}$.

Consider $\Delta_{C,\chi,\varepsilon}^* := \{(x,E^{*c}(x)): x\in \Delta_{C,\chi,\varepsilon}\}$ for $* = -,+$. The sets $\Delta_{C,\chi,\varepsilon}^*$ are compact and disjoint. Take small neighborhood $U^*$ of $\Delta_{C,\chi,\varepsilon}^*$ with the property that if $(x,v^*) \in U^*$ for $*=-,+$ then the angle between $v^-$ and $v^+ $ is at least $\overline{\theta}$. 

Consider $\nu_n^-(x) := \frac{1}{n} \sum_{j=0}^{n-1} \delta_{\hat{g}^j(x, [E^{-c}(x)])}$ and $\nu_n^+(x) := \frac{1}{n} \sum_{j=0}^{n-1} \delta_{\hat{g}^j(x, [F(x)] )}$.  By \Cref{lem.projconvgotherdirections}, we have that for $n$ large
\[
\nu^*_n(x)(U^*) > 1-2\delta. 
\]
Therefore, 
\[
\overline{\delta} \left(\{n\in \mathbb{N}: \theta(g^n(x),E^{-c}(g^n(x)), F_{nN}(x)) < \overline{\theta}\}\right) <4\delta. 
\]
This shows that 
\[
\lim_{\overline{\theta} \to 0} \overline{\delta} \left(\{n\in \mathbb{N}: \theta(g^n(x),E^-(g^n(x)), F_{nN}(x))< \overline{\theta}\}\right) =0.
\]
By \Cref{thm.CDPmain}, $g$ has an SRB measure. Observe that if $\mu$ is an SRB measure for $g$, then
\[
\frac{1}{N} \sum_{j=0}^{N-1} f^j_* \mu
\]
is an SRB measure for  $f$.

We are left to show that the union of the basins of the SRB measures covers  a set of full volume. 

Suppose not. Then, there exists a set $\mathcal{B}$ having positive volume, such that $\mathcal{B}$ does not intersect any of the basins of SRB measures. Up to removing a set of zero volume, we may assume that every point in $\mathcal{B}$ is contained in $G$. In the same way as above, we can fix a disk $D$ having dimension $\dim(E^u) +1$, a set $\mathcal{B}' \subset \mathcal{B} \cap D$, where for every $x\in \mathcal{B}'$, the direction $E^{-c}(x)$ is not contained in $T_xD$, and the $m_D$-measure of $\mathcal{B}'$ is positive. 

Recall the outline of the construction of SRB measures given in \Cref{subsetction.effectivehyperbolicity}. For some fixed constants $\mathbf J = (C, \overline{\lambda})$, we can define the set $S_{\mathbf J, q, n}(\mathcal{B}')$ in the same way as $S_{\mathbf J, q ,n}$, but restricting to points in $\mathcal{B}'$. Following the exact same steps, we obtain constants $\beta'', \beta'>0$, a set $A_n \subset S_{\mathbf J, q, n}$ such that
\begin{enumerate}
    \item For $x\in A_n$, let $W_n^x$ be as in \eqref{eq.djx}. Then,
    \begin{equation}\label{eq.betalinha}
    m_D(W_n^x \cap S_{\mathbf J, q, n}(\mathcal{B}')) \geq \beta' m_D(W_n^x).
    \end{equation}
    \item For $x,y \in A_n$ with $x\neq y$, $W_n^x \cap W_n^y = \emptyset$.
    \item Set $W_{\mathbf J, q, n} = \bigsqcup_{x\in A_n} W_n^x$. Then
    \[
    m_D(W_{\mathbf J, q ,n}) \geq \beta''>0.
    \]
\end{enumerate}
In what follows, up to replacing $f$ by $f^N$, we will write the argument for $f$. 
For a slowly increasing sequence $q_n \to +\infty$ one sets 
\[
\displaystyle \nu_n = \frac{1}{n} \sum_{j=q_n}^{n-1} f^j_* \left(m_D|_{W_{\mathbf J, q_n, j}}\right), 
\]
and the total mass of $\nu_n$ is greater than $\frac{1}{2} \beta''>0$. For $\mu_n = \frac{1}{n} \sum_{j=0}^{n-1} f^j_* m_D$, we have that $0 \neq \nu_n \leq \mu_n$. Recall that, up to considering a subsequence, we had that 
\[
\mu_n \longrightarrow \mu \textrm{ and } \nu_n \to \nu,
\]
and the measure $\nu$ has the following properties:
\begin{itemize}
    \item $\nu(M) \geq \frac{1}{2} \beta''$;
    \item Proposition $4.4$ from \cite{CDP} gives that $\nu$ can be written as a convex combination of measures $\rho_\alpha \der m_{D_\alpha}$, where $D_\alpha$ is a genuine unstable disk of uniform size and $\rho_\alpha$ is a density uniformly bounded away from $0$ and $\infty$. 
    \item $\nu$ is part of an invariant measure.
\end{itemize}
In particular, there are ergodic components of $\mu$ with the SRB property. 

Since $S_{\mathbf J, q_n, j} \subset \mathcal{B}'$, by \eqref{eq.betalinha} and \eqref{eq.density}, for $j\geq N(q_n)$ and $x\in A_j$, we have that
\begin{equation}
    \label{controlafteriterate}
    \displaystyle \frac{ f^j_*m_D\left(D_j^x\cap f^j(\mathcal{B}')\right)}{f^j_*m_D(D_j^x)} \geq \frac{\beta'}{L}.
\end{equation}

By \Cref{prop.uniform.pesin} and since $\nu(M) > \frac{1}{2} \beta''$, we can fix $\varepsilon>0$ sufficiently small and find a Pesin block, which we will denote it by $\Lambda$ such that the following holds: for a set of points $\mathcal{R}$ having positive $\nu$-measure, for $x\in \mathcal{R}$ we have:
\begin{itemize}
    \item there exists an SRB measure $\mu_x$ for which $x$ is a typical point;
    \item there exists an unstable disk $D(x)$ of uniform size containing $x$ such that $m_{D(x)}(\Lambda)>1-\varepsilon$;
    \item there exist sequences $\left(n_k\right)_k$, with $n_k \to +\infty$,  $\left(x_k\right)_k$, where $x_k \in A_{n_k}$ such that 
    \[
    D_{n_k}^{x_k} \longrightarrow D(x),
    \]
    where the convergence happens in the $C^1$-topology.
\end{itemize}

Consider the set 
\[
\Lambda^s_x:= \bigcup_{y \in \Lambda \cap D(x)} W^{s-}_{loc}(y),
\]
where $W^{s-}_{loc}(y)$ is the local Pesin stable manifold of $y$, which has uniform size and it forms an absolutely continuous lamination for points in a fixed Pesin block. By \eqref{controlafteriterate}, and using the $C^1$-convergence of $D_{n_k}^{x_k}$ to $D(x)$, we have that for $k$ large enough
\[
\displaystyle \frac{ f^{n_k}_*m_D\left(D_{n_k}^{x_k}\cap \Lambda^s_x\right)}{f^{n_k}_*m_D(D_{n_k}^{x_k})}> \frac{1-2\varepsilon}{L}.
\]
In particular, for $\varepsilon$ small enough, and for $k$ large, we obtain that 
\[
f^{n_k}_*m_D \left(f^{n_k}(\mathcal{B}') \cap \Lambda^s_x\right)>0,
\]
but this implies that, up to removing a set of zero measure, the set $f^{n_k}(\mathcal{B}')  \cap \Lambda^s_x$ belongs to the basin of $\mu_x$, which is a contradiction with our definition of $\mathcal{B}'$. So, $\mathcal{B}$ must have zero volume. \qedhere

\end{proof}

\section{Exponential mixing of the SRB measures} \label{Sec:SPR}

In this section we show the exponential mixing of the SRB measures.  The main tool we will use is the recent work of Buzzi-Crovisier-Sarig \cite{BCSSPR} about the \emph{strong positive recurrence} introduced to study measures of maximal entropy. 

\subsection{Strong positive recurrence} \label{SPRsub}

In this subsection we will introduce the material needed from \cite{BCSSPR}.  We will focus the statements for the geometric potential, which is the one we will be considering in this paper. 

Let $f$ be a $C^{1}$ diffeomorphism of a closed manifold and fix $\chi >0$. An ergodic invariant measure $\mu$ is $\chi$-hyperbolic if the absolute value of all of its Lyapunov exponents are greater than $\chi$.

Given constants $C,\chi, \varepsilon>0$, recall that in \Cref{Pesin.block} we defined a $(C,\chi, \varepsilon)$-Pesin block. In a $(C,\chi,\varepsilon)$-Pesin block $\Delta$, the Oseledets bundles $E^-(x)$ and $E^+(x)$ depend continuously with the choice of $x\in \Delta$.  The \emph{recurrent non-uniformly hyperbolic set} is defined by
\[
\mathrm{NUH}'(f):= \displaystyle \bigcup_{\chi >0} \bigcap_{\varepsilon>0} \bigcup_{C\geq 1} \bigcup_{(C,\chi,\varepsilon)-\textrm{Pesin block }\Lambda} \{x: x\textrm{ is a limit of periodic points in $\Lambda$}\}.
\]

\begin{definition}\label{def:homoclinicrelation}
Given two points $x,y\in \mathrm{NUH}'(f)$, we say that they are \emph{homoclinically related} if the Pesin stable manifold of $x$ has a transverse intersection with the Pesin unstable manifold of $y$, and the Pesin unstable manifold of $x$ has a transverse homoclinic intersection with the Pesin stable manifold of $y$.  In this case, we write $x\sim_h y$. 
\end{definition}

This relation gives an equivalence relation in $\mathrm{NUH}'(f)$ (see Proposition $2.6$ in \cite{BCSSPR}).  The equivalence classes of $\sim_h$ in $\mathrm{NUH}'(f)$ are called the \emph{Borel homoclinic classes}.  The \emph{period} of a Borel homoclinic class $X$ is the greatest common divisor of all the periods of periodic points in the class.  

Let $X$ be an invariant Borel set and $\phi: X \to \mathbb{R} \cup \{-\infty\}$ be a Borel function such that $\sup \phi < +\infty$.  Given an invariant measure carried by $X$, we define $P(\mu, \phi) = h(\mu) + \int \phi d\mu$.  Define the TOP pressure of $f$ restricted to $X$ by
\[
P_{\mathrm{TOP}}(f|_X,\phi) := \sup \{P(\mu, \phi): \mu \textrm{ is an invariant measure supported on $X$}\}.
\]

\begin{definition}[Strong Positive Recurrence, Definition $1.12$ in \cite{BCSSPR}]
A diffeomorphism $f$ is SPR for $\phi$ on an invariant Borel set $X$ if there exists $\chi>0$ such that for each $\varepsilon>0$, there are $C\geq 1$,  a $(C,\chi, \varepsilon)$-Pesin block $\Lambda$ and numbers $P_0< P_{\mathrm{TOP}}(f|_X,\phi)$ and $\tau >0$ such that: for every ergodic measure $\nu$ on $X$, if $P(\nu,\phi) >P_0$ then $\nu(\Lambda) >\tau$. We say that $f$ is SPR for $\phi$ if $X$ can be taken to be the entire manifold. 
\end{definition}

An invariant measure $\nu$ that verifies $P(\nu, \phi) = P_{\mathrm{TOP}}(f|_X, \phi)$ is called an \emph{equilibrium state}. In what follows, $\phi$ is either a H\"older continuous function or the geometric potential, which is defined as 
\[
\phi(x) = -\log |\det D_xf|_{E^+(x)}|
\]
if the full Oseledets unstable direction exists, or $-\infty$ if the direction doesn't exist.
Buzzi-Crovisier-Sarig showed that this condition is translated to the SPR condition for countable Markov shifts. Using the results that were already known for thermodynamical formalism of countable Markov shifts, they obtain a spectral gap property which implies several good statistical properties.

\begin{theorem}[Theorem $G$ in \cite{BCSSPR}]\label{thm.sprexponential}
Let $f$ be a $C^{1+\alpha}$-diffeomorphism on a compact manifold.  Let $X$ be a Borel homoclinic class with period $q$ and let $\phi$ be either a H\"older continuous function or the  geometric potential on $X$.Suppose that $f$ is SPR for $\phi$ on $X$. Then the following properties hold:
\begin{enumerate}
\item \textbf{Existence and Uniqueness:} There exists a unique equilibrium state measure $\mu_{\phi}$ on $X$.
\item \textbf{Exponential mixing and other properties:} $\mu_{\phi}$ is isomorphic to the product of a Bernoulli scheme with a cyclic permutation of $q$ points.  That is, there are probability measures $\mu_1, \cdots , \mu_q$ such that
\[
\mu_\phi = \frac{1}{q} \sum_{i=1}^q \mu_i,
\]
such that $f_*\mu_i = \mu_{i+1}$ with the convention that $q+1=1$. The following exponential mixing property holds: for every $\beta>0$, there are $\theta \in (0,1)$ and $K>1$ such that for all $\beta$-H\"older continuous functions $\eta, \psi:M \to \mathbb{R}$, for any $i\in \{1, \cdots, q\}$,
\[
\displaystyle \left| \int \eta . (\psi \circ f^{nq} ) d\mu_i - \int \eta d\mu_i \int \psi d\mu_i \right| \leq K \|\eta\|_\beta \|\psi\|_\beta \theta^n, \textrm{  for }n\geq 0, 
\]
where 
\[
\|\eta\|_\beta:= \sup |\eta| + \sup \left\{ \frac{|\eta(x) - \eta(y)|}{d(x,y)^\beta}: \textrm{ } x,y \in M, x\neq y\right\}.
\]
The measure $\mu_{\phi}$ also verifies other quantitative statistical properties, such as large deviations, almost sure invariance principle, among others (see Theorem $11.13$ in \cite{BCSSPR}).
\end{enumerate}
\end{theorem}

In what follows, we will work with the geometric potential and denote it by $\phi$. 
  By Ruelle's inequality, $h(\mu) \leq -\int \phi d\mu$, where the right hand side coincides with the sum of the positive Lyapunov exponents.  Define the \emph{pressure} of $\mu$ by
\[
P(\mu) = h(\mu) +\int \phi \dd\mu,
\]
 Observe that $P(\mu) \leq 0$. Moreover, after the work of Ledrappier, Ledrappier Young \cite{Ledhyp,LedYoungI}, $P(\mu) = 0$ if and only if the measure $\mu$ is SRB.  To prove that the SRB measures are exponentially mixing, up to a period, for expanding on average systems, we will show that they are SPR on each Borel homoclinic class supporting an SRB measure. 

\subsection{Unstable entropy and projective measures}\label{subsection.unstableentropy}

 Let $f$ be a partially hyperbolic $\mathcal C^1$-diffeomorphism, and let $\Fu$ be its unstable foliation. For any invariant probability measure $\mu$, we say that a partition $\xi$ is $\mu$-measurable if there exists a family of measurable sets $\{Q_i\}_{i\in \mathbb{N}}$ and a measurable set $R \subset M$ of full $\mu$-measure, such that if $A\in \xi$ is an atom of the partition, then there exists a sequence $(A_i)_i$, where $A_i \in \{Q_i, Q_i^c\}$ such that $A \cap R = \bigcap_i A_i \cap R$. We remark  that if $\xi$ is a measurable partition, then for $\mu-$almost every $x$ the atom of $\xi$ containing $x$ (henceforth, denoted as $\xi(x)$) is a measurable set.  If $\xi, \xi'$ are two (not necessarily measurable) partitions, we say that $\xi>\xi'$ ($\xi$ is finer than $\xi'$) if for $\mu-$ almost every point $x$, $\xi(x) \subset \xi'(x)$. 

\begin{definition}
We say that a $\mu$-measurable partition $\xi$ is $u$-subordinated if for $\mu-$almost every $x$ the following conditions hold:
\begin{itemize}
\item there is $r_0>0$ such that $\xi(x)$ is contained in the ball centered at $x$ of radius $r_0$ inside $\Wu{x}$;
\item $\xi(x)$ contains an open neighborhood of $x$ inside $\Wu{x}$;
\item $\xi$ is increasing, meaning that $\xi>f\xi=\{f(\xi(f^{-1}x))\}$.
\end{itemize}	
 \end{definition}  
The classical construction of Sinai \cite{SinaiMarkov} shows that $\mu$-measurable $u$-subordinated partitions exist.  

The \emph{unstable entropy} of $f$ with respect to $\mu$ is defined by
\begin{equation}
\displaystyle h^u_{\mu}(f) = H_\mu(f^{-1} \xi|\xi),
\end{equation}
where $\xi$ is any $u$-subordinated $\mu$-measurable partition, and $H_\mu(f^{-1}\xi|\xi)$ is the conditional entropy of $f^{-1}\xi$ with respect to $\xi$, which is given by
\[
H_{\mu} (f^{-1}\xi|\xi) = \displaystyle \int - \log \mu^{\xi}_x(f^{-1} \xi(x)) \dd\mu(x).
\]
This value is independent of the choice of $u$-subordinated partition. 

Suppose that $f\in\PHb[2]$, and recall the notation from Section \ref{section.projectiveugibbs}. Given a $\mu$-measurable partition $\xi$, we can use the natural projection $\hpi: \hM \to M$, to define a partition $\hat{\xi}$ on $\widehat{M}$. Given a point $\hat{p}$, we define $\hat{\xi}(\hat{p})$ by 
\begin{equation}\label{eq.unstableprojectivepartition}
	\left(\hat{\pi}|_{\widehat{W}^u(\hat{p})} \right)^{-1}(\xi( \hat{\pi}(\hat{p}))).
\end{equation}
Observe that by construction, the partition $\hat{\xi}$ is finer than the projetive unstable foliation $\widehat{\Fu}$.

\begin{lemma}\label{lemma.hatmumeasurable}
Let $\mu$ be an invariant measure and $\xi$ a $\mu$-measurable $u$-subordinated partition. Let $\hat{\xi}$ be the partition of $\widehat{M}$ induced by $\xi$ defined above. Suppose that $\hat{\mu}$ is any $\hat{f}$-invariant measure that projects to $\mu$. Then, $\hat{\xi}$ is $\hat{\mu}$-measurable. 
\end{lemma}

\begin{proof}
Fix a measurable section $\xi \ni P \mapsto x_P \in P$ and for $(q,[v]) \in \hM$ define 
\[
\Theta\left(q,[w]\right) = (x_{\xi(q)},h^u_{x_{\xi(q)},q}([v])).
\]
The function $\Theta$ is measurable and, observe that, the level sets of $\Theta$ are exactly the atoms of the partition $\hat \xi$.  This shows that $\hat \xi$ forms a measurable partition. 
\end{proof}

Given an $\hat{f}$-invariant measure $\hat{\mu}$ that projects to $\mu$, we can define in a similar way the unstable entropy of $\hf$ with respect to $\hat{\mu}$ by 
\[
h^u_{\hmu}(\hf) = H_{\hat{\mu}}(\hat{f}^{-1} \hat{\xi} | \hat{\xi}).
\]
We will need the following version of the invariance principle
\begin{theorem}\label{thm.invarianceprinciple}
Let $\mu$ and $\hmu$ be as above. Then,
\[
h^u_{\hmu}(\hf) \leq h^u_{\mu}(f).
\]
Moreover, equality holds if and only if for $\hat{\mu}$-almost every $\hp$, and some ( a fortriori, any) $\mu-$measurable partition $\xi$, it holds 
$\hpi_* \hmu^{\hat \xi}_{\hat{p}} = \mu^{\xi}_{\hat{\pi}(\hat{p})}$. 
\end{theorem}

When $\mu$ is a $u$-Gibbs measure, this result says that $\hat{\mu}$ is a projective $u$-Gibbs measure if and only if the unstable entropy of $\hat{\mu}$ coincides with the unstable entropy of $\mu$.  The proof of this result is essentially contained in Tahzibi-Yang's paper \cite{TahzibiYang}. For completeness, we will explain the proof. Before we proceed to the proof, we will be working with the following partitions and measures:
\begin{itemize}
\item $f\in \PHb[2]$ and $\mu$ is an $f$-invariant measure. 

\item $\xi$ a $u$-subordinated $\mu$-measurable partition that we fixed for $f$.  Elements of $\xi$ will be denoted by $P$ or $P_x = \xi(x)$, for $\mu$-almost every $x$.  Conditional measures will be denoted by $\mu^\xi_P$ or $\mu^\xi_x$, meaning, $\mu^\xi_{P_x}$. We write $\mu_Q$ the quotient measure of $\mu$ on $M/\xi$. 
\item $\hat{\xi}$ the partition defined above, whose elements are contained in projective unstable manifolds. Elements of $\hat{\xi}$ will be denoted by $\widehat{P}$ or $\widehat{P}_{(x,[v])} = \hat{\xi}(x,[v])$, for $\hmu$-almost every $\hat{x} =(x,[v])$.  Conditional measures are denoted by $\hmu^{\hat{\xi}}_{\hat{P}}$, or $\hmu^{\hat{\xi}}_{(x,[v])}$. We write $\hmu_Q$ the quotient measure of $\hmu$ on $\widehat{M}/\hat{\xi}$. 
\item The partition $\hat{\xi}^{cu} :=\hat{\pi}^{-1} \xi$ on $\widehat{M}$, whose elements are saturated by fibers.  Elements of this partition will be denoted by $\widehat{P}^{cu}$ or  $\widehat{P}^{cu}_P= \hat{\pi}^{-1}(P)$.  Conditional measures are denoted by $\hat{\mu}^{\hat{\xi}^{cu}}_{\widehat{P}^{cu}}$ or $\hat{\mu}^{\hat{\xi}^{cu}}_{(x,[v])}$.  The quotient measure is denoted by $\hmu^{cu}_Q$ on $\widehat{M}/\hat{\xi}$. 
\item We will also work with the partitions $f^{-1} \xi$, $\hat{f}^{-1} \hat{\xi}$ and $\hat{f}^{-1} \hat{\xi}^{cu}$.  We will denote the elements of these partitions by $P^{-1}$, $\widehat{P}^{-1}$ and $\widehat{P}^{cu,-1}$, respectively.  
\end{itemize}

\begin{lemma}\label{lem.equivalentquotientmeasures}
With the notation above,  for $\hmu$-almost every $(x,[v])$, $\hat{\pi}_* \hmu^{\hat{\xi}^{cu}}_{(x,[v])} = \mu^\xi_x$. 
\end{lemma}
\begin{proof}
 Observe that $\widehat{M}/ \hat{\xi}^{cu}$ can be naturally identified with $M/\xi$. We have that 
\[
\hmu =  \int_{M/\xi} \hmu^{\hat{\xi}^{cu}}_{\widehat{P}^{cu}} \dd\hmu_{Q}(\widehat{P}^{cu}) \textrm{ and } \mu = \int_{M/\xi} \mu^{\xi}_P \dd\mu_{Q}(P),
\]
Given a measurable set $B\subset M/\xi$, we have that
\[
\mu_Q(B) = \mu\left(\bigcup_{P\in B} P \right)  = \hmu\left(\hpi^{-1} \left(\bigcup_{P\in B} P \right) \right) = \hmu_{Q}(B).
\]
Therefore, $\hmu_{Q} = \mu_{Q}$.  Thus,
\[
\int_{M/\xi} \mu^{\xi}_Pd\mu_{Q}(P) = \mu = \hpi_* \hmu = \hpi_*\int_{M/\xi} \hmu^{\hat{\xi}^{cu}}_P \dd\hmu_{Q}(P) = \int_{M/\xi} \hpi_* \hmu^{\hat{\xi}^{cu}}_P \dd\mu_{Q}(P),
\]
and, by the essential uniqueness of the disintegration of a measure,  we have that 
\begin{equation}\label{eq.equalquotientmeasures}
\mu^{\xi}_{x} = \hpi_*\hmu^{\hat{\xi}^{cu}}_{(x,[v])}
\end{equation}
 for $\hmu$-almost every $(x,[v])$. 
 \end{proof}
 
 \begin{proof}[Proof of \Cref{thm.invarianceprinciple}]
 
For $\mu_{Q}$-almost every $P$, fix a point $x_P \in P$. Write $\widehat{\xi}_P$ the partition $\widehat{\xi}$ restricted to $\widehat{P}^{cu}_P$. Observe that $D_P:= \hpi^{-1}(x_P)$ can be naturally identified with $\widehat{P}^{cu}_P / \hat{\xi}_P$. Write $\hat{\mu}^P_Q$ the quotient measure of $\hat{\mu}^{cu}_P$ on $\widehat{P}^{cu}_P/\hat{\xi}_P$.  So, 
\begin{equation}\label{eq.furtherdecomposition}
\hmu^{\hat{\xi}^{cu}}_P = \displaystyle \int_{D_P} \hmu^{\widehat{\xi}}_t \dd\hmu^P_{Q}(t), 
\end{equation}
where  $\hmu^{\widehat{\xi}}_t$ is the conditional measure on $\hat{\xi}_P(x,t)$. 

Since $\xi$ is an increasing $u$-subordinated partition,  for each $P$ as above, there a subset of the naturals $I_P$, which is at most countable, and sets $P^{-1}_i$, where $i\in I_P$, such that
\[
P = \displaystyle \bigsqcup_{i\in I_P} P^{-1}_i.
\]
This gives a partition of $\widehat{P}^{cu}_P$ by
\[
\widehat{P}^{cu}_P = \displaystyle \bigsqcup_{i\in I_p} \widehat{P}^{cu,-1}_i,
\]
where $\widehat{P}^{cu,-1}_i = \hpi^{-1}(P^{-1}_i)$. This follows from the fact that $f^{-1}\xi$ is $u$-subordinated and refines $\xi$, so every element of the partition contains an open ball inside a strong unstable manifold. In particular, an element of $\xi$ contains at most countably many elements of $f^{-1}\xi$.  Therefore,  by \Cref{lem.equivalentquotientmeasures} and by \eqref{eq.furtherdecomposition}
\[
\begin{array}{lll}
h^u_\mu(f) & = & \displaystyle \int - \log \mu^\xi_x(f^{-1}\xi(x)) \dd\mu(x) = \int_{M/\xi} \sum_{i\in I_P} -\mu^\xi_P(P^{-1}_i) \log \left(\mu^\xi_P(P^{-1}_i)\right) \dd\mu_Q(P) \\
& = & \displaystyle \int_{M/\xi} \sum_{i\in I_P} - \hat{\mu}^{\xi^{cu}}_P (\widehat{P}^{cu,-1}_i) \log \left(\hat{\mu}^{{\xi}^{cu}}_P(\widehat{P}^{cu,-1}_i)\right) \dd\mu_Q(P)\\
& = & \displaystyle \int_{M/\xi} \sum_{i\in I_P} -  \left(\int_{D_P} \hmu^{\widehat{\xi}}_t(\widehat{P}^{cu,-1}_i) \dd\hmu^P_{Q}(t)\right) \log \left(\int_{D_P} \hmu^{\widehat{\xi}}_t(\widehat{P}^{cu,-1}_i) \dd\hmu^P_{Q}(t)\right) \dd\mu_Q(P).
\end{array}
\]
Consider $\psi(t) = -t\log t$, for $t>0$. Observe that $\psi$ is a concave function. By Jensen inequality,
\[
\begin{array}{l}
 h^u_\mu(f) = \displaystyle \int_{M/\xi} \sum_{i\in I_P} -  \left(\int_{D_P} \hmu^{\widehat{\xi}}_t(\widehat{P}^{cu,-1}_i) \dd\hmu^P_{Q}(t)\right) \log \left(\int_{D_P} \hmu^{\widehat{\xi}}_t(\widehat{P}^{cu,-1}_i) \dd\hmu^P_{Q}(t)\right) d\mu_Q(P)\\
 =\displaystyle  \int_{M/\xi} \sum_{i\in I_P} \psi \left(\int_{D_P} \hmu^{\widehat{\xi}}_t(\widehat{P}^{cu,-1}_i) \dd\hmu^P_{Q}(t)\right) \dd\mu_Q(P)\\
\displaystyle \geq  \int_{M/\xi} \sum_{i\in I_P} \int_{D_P}\psi \left( \hmu^{\widehat{\xi}}_t(\widehat{P}^{cu,-1}_i) \right) \dd\hmu^P_{Q}(t) \dd\mu_Q(P) = \\
 \displaystyle  \int_{M/\xi} \sum_{i\in I_P} \int_{D_P}- \hmu^{\widehat{\xi}}_t(\widehat{P}^{cu,-1}_i)\log \left( \hmu^{\widehat{\xi}}_t(\widehat{P}^{cu,-1}_i)\right) \dd\hmu^P_{Q}(t) \dd\mu_Q(P)\\
 = \displaystyle  \int_{M/\xi} \int_{D_P} \sum_{i\in I_P} - \hmu^{\widehat{\xi}}_t(\widehat{P}^{cu,-1}_i)\log \left( \hmu^{\widehat{\xi}}_t(\widehat{P}^{cu,-1}_i)\right) \dd\hmu^P_{Q}(t) \dd\mu_Q(P)\\
 = \displaystyle \int_{M/\xi} \int_{\widehat{P}^{cu}_P} - \log \left( \hmu^{\widehat{\xi}}_t\left(\hat{f}^{-1}\hat{\xi}(x,[v])\right)\right) \dd\hmu^{\hat{\xi}^{cu}}_{P}(x,[v]) \dd\mu_Q(P)\\
 = \displaystyle \int_{\widehat{M}} -\log \left(\hat{\mu}^{\hat{\xi}}_{(x,[v])} \left( \hat{f}^{-1} \hat{\xi}(x,[v])\right) \right) \dd\hat{\mu}(x,[v]) = h^u_{\hat{\mu}}(\hat{f}).
\end{array}
\]
This concludes the proof of $h^u_{\hat{\mu}}(\hat{f}) \leq h^u_\mu(f)$. Moreover, equality occurs if and only if for $\mu_Q$-almost every $P$, $\hat{\mu}^P_Q$-almost every $t$ and any $i\in I_P$,
\[
\hat{\mu}^{\hat{\xi}}_t(\widehat{P}_i^{cu,-1}) = \mu^{\xi}_P(P^{-1}_i).
\]
This is translated to the measures $\hat{\pi}_* \mu^{\hat{\xi}}_{(x,[v])} = \mu^\xi_x$ on the subalgebra generated by $f^{-1} \xi$.  Since $h^u_{\hat{\mu}}(\hat{f}^n) = n h^u_{\hat{\mu}}(\hat{f}) = n h^u_\mu(f) = h^u_\mu(f^n)$,  by repeating the same argument, we obtain that $\hat{\pi}_* \mu^{\hat{\xi}}_{(x,[v])} = \mu^\xi_x$ on the subalgebra generated by $\{f^{-n}\xi\}_{n\in \mathbb{N}}$. Observe that the size of each atom of  the partition $f^{-n}\xi $ is going to zero as $n$ increases, and that $f^{-n}\xi$ refines $\xi$. Therefore,  $\bigvee_n f^{-n} \xi$ generates the Borel $\sigma$-algebra on each atom of $\xi$. Thus, $\hat{\pi}_* \mu^{\hat{\xi}}_{(x,[v])} = \mu^\xi_x$.  \qedhere

 \end{proof}

\subsection{Semicontinuity of unstable projective entropy}

It will be important for us to obtain some type of semi continuity of the unstable entropy of projective measures.  This is an adaptation of the main theorem of \cite{Yangunstableentropy}.

\begin{theorem}\label{thm.semicontinuityuentropy}
Suppose that $\hat{\mu}_n$ is a sequence of $\hat{f}$-invariant measures converging to $\hat{\mu}$ in the weak*-topology. Then,
\[
\displaystyle \limsup_{n\to +\infty} h^u(\hat{\mu}_n) \leq h^u(\hat{\mu}).
\]
\end{theorem}

The proof follows \cite{Yangunstableentropy} almost verbatim. Let us give a sketch of the proof explaining the small changes needed.

\begin{proof}
    Let us first explain the setting in \cite{Yangunstableentropy}.  Let $f$ be a partially hyperbolic $C^1$-diffeomorphism, let $\mu_n$ be a sequence of $f$-invariant measures converging to $\mu$. Then the main result in \cite{Yangunstableentropy} proves that
    \[
\displaystyle \limsup_{n\to +\infty} h^u(\mu_n) \leq h^u(\mu).
    \]
    Let $\{\left(B_i,\Phi_i,D_i\right)\}_{i=1}^k$ be a finite sequence of foliated charts for the strong unstable foliation, where $(B_i,\Phi_i)$ is a  chart and $D_i$ is a choice of transversal to the foliation in the chart, such that the union of $B_i$ cover $M$.  
    
    \textbf{Step 1: Building the partitions $\mathcal{A}^{W^u}$ and $\mathcal{A}^m$.} One constructs a finite partition $\mathcal{A}$ of $M$ whose elements are contained in the foliated charts $B_i$ and with certain decay of measure property of neighborhoods of the boundary of the partition (see the beginning of Section 3 from \cite{Yangunstableentropy}). Once this partition $\mathcal{A}$ is fixed, one defines $\mathcal{A}^{W^u}$ (which is denoted by $\mathcal{A}^{\mathcal{F}}$ or $\mathcal{A}^{\mathcal{F}_n}$ in \cite{Yangunstableentropy}) to be the partition whose elements are the intersection of an element of $\mathcal{A}$ with the local plaque of $W^u$ of the corresponding   foliation box $B_i$.  For each $m\in \mathbb{N}$ one defined $\mathcal{A}^m = \bigvee_{j=1}^m f^{-j}(\mathcal{A})$.  

    \textbf{Step 2: Building some new partitions.} For each $i$, let $C_{i,1} \leq C_{i,2} \leq \cdots$ be a sequence of finite partitions of $D_i$ with the following properties:
    \begin{itemize}
\item $\mathrm{diam}(C_{i,t}) \to 0$ as $t\to +\infty$;
\item for any $i,t \geq 0$, and any element $C$ of $C_{i,t}$, 
\begin{equation}\label{eq.zeromeasure} \mu\left( \bigcup_{x\in \partial C} \mathcal{A}^{W^u}(x) \right) = 0.\end{equation}
    \end{itemize}
The partitions $C_{i,t}$ induces a new partition
\[
\widetilde{C}_{i,t}= \{ \bigcup_{x\in C} \mathcal{A}^{W^u}(x): \textrm{ $C$ is an element of $C_{i,t}$}\}.
\]
For $m\in \mathbb{N}$ and an element $P$ of $\mathcal{A}^m$, observe that there is $i$ such that $P$ is contained in $B_i$, the partition $\widetilde{C}_{i,t}$ induces a partition on $P$ given by $P_t= \{P \cap \widetilde{C}: \textrm{ $\widetilde{C}$ is an element of $\widetilde{C}_{i,t}$}\}$. One then defines
\[
\mathcal{A}^m_t= \{P_t: \textrm{ $P$ is an element of $\mathcal{A}^m$.}\}.
\]
These are built in Section $4.2.1$ in \cite{Yangunstableentropy}. 

\textbf{Step 3: Proving the semicontinuity.} Given $\varepsilon>0$, Proposition $4.1$ in  \cite{Yangunstableentropy} gives that for $m_0$ large
\begin{equation}
    \label{eq.entropyeq1}
    \frac{1}{m_0} H_{\mu} \left( \bigvee_{j=0}^{m_0} f^{-j}\mathcal{A}^{W^u}| \mathcal{A}^{W^u} \right) - \frac{\varepsilon}{3} \leq h^u(\mu).
\end{equation}
The partition $\mathcal{A}^{W^u}$ is not $u$-subordinated (the main problem is the increasing property), however, this  first inequality states that unstable entropy can be approximated by the finite time refinement $\bigvee_{j=0}^{m_0} f^{-j} \mathcal{A}^{W^u}$. 

Proposition $4.5$ in \cite{Yangunstableentropy} then gives that for $t_0$ large
\begin{equation}
    \label{eq.entropyeq2}
    H_\mu \left( \mathcal{A}^{m_0}_{t_0} | \mathcal{A}^0_{t_0}\right) -\frac{\varepsilon}{3} \leq  H_{\mu} \left( \bigvee_{j=0}^{m_0} f^{-j}\mathcal{A}^{W^u}| \mathcal{A}^{W^u} \right).
\end{equation}
One can think of  elements of $\mathcal{A}^m_t$ as a ``thickening'' of elements of $\bigvee_{j=0}^{m_0} f^{-j}\mathcal{A}^{W^u}$, this is good for approximating it by measures nearby. Indeed, the property about zero measure from \eqref{eq.zeromeasure} implies that for any $P$ in $\mathcal{A}^{m_0}_{t_0}$ we have $\mu(\partial P) = 0$, among some other properties. In particular,
\[
\displaystyle\lim_{n\to +\infty} \mu_n(P) = \mu(P).
\]
Therefore,
\[
\displaystyle\lim_{n\to +\infty} H_{\mu_n}\left(\mathcal{A}^{m_0}_{t_0} |\mathcal{A}^0_{t_0} \right) =  H_{\mu}\left(\mathcal{A}^{m_0}_{t_0} |\mathcal{A}^0_{t_0} \right).
\]
For $n$ sufficiently large,
\begin{equation}
    \label{eq.entropyeq3}
     H_{\mu_n}\left(\mathcal{A}^{m_0}_{t_0} |\mathcal{A}^0_{t_0} \right) -\frac{\varepsilon}{3} \leq  H_{\mu}\left(\mathcal{A}^{m_0}_{t_0} |\mathcal{A}^0_{t_0} \right).
\end{equation}
Combining \eqref{eq.entropyeq1}, \eqref{eq.entropyeq2} and \eqref{eq.entropyeq3}, we obtain that 
\[
\frac{1}{m_0}  H_{\mu_n}\left(\mathcal{A}^{m_0}_{t_0} |\mathcal{A}^0_{t_0} \right) -\varepsilon \leq h^u(\mu).
\]
Corollary $4.6$ in \cite{Yangunstableentropy} gives that for any $m$ and $t$, 
\[
\frac{1}{m} H_{\mu_n}\left(\mathcal{A}^{m}_{t} |\mathcal{A}^0_{t} \right) \geq h^u(\mu_n).
\]
Hence,
\[
h^u(\mu_n) - \varepsilon \leq h^u(\mu).
\]

\textbf{The proof of \Cref{thm.semicontinuityuentropy}} The only difference of the proof in our setting is that we use the foliated boxes built in \Cref{subsection.projectivedynamics} for $\widehat{\Fu}$ instead of the foliated boxes in $B_i$. The proof then follows verbatim the scheme above. 
\end{proof}

\subsection{Proof of the strong positive recurrence property}

In this subsection, we show that if $f$ has the expanding on average property and $\Gap>0$, then $f$ is SPR for every Borel homoclinic class supporting an SRB measure.  We will prove the following theorem.

\begin{theorem}\label{thm.sprueap}
Let $f\in \PHb[2]$ have the expanding on average property, and satisfying $\Gap>0$. Let $\mu$ be an ergodic SRB measure for $f$, and let $X$ be the Borel homoclinic class carrying $\mu$. Then $f$ is SPR on $X$. 
\end{theorem}

Recall that to show that $f$ is SPR on $X$, we need to show the existence of a Pesin block such that  any measure with large enough pressure gives uniform mass to this Pesin block.  To prove this, we will need the following proposition.

\begin{proposition}\label{prop.convergenceofexponents}
Let $f\in \PHb[2]$ have the expanding on average property, and satisfying $\Gap>0$.  Let $X$ be a Borel homoclinic class.  Suppose that $\mu_n$ is a sequence of ergodic measures carried by $X$  such that $P(\mu_n)$ converges to $0$. Take any limit measure $\mu$ of the sequence $\mu_n$.  Then,
\begin{enumerate}
\item $\mu$ is a $u$-Gibbs measure,
\item $\displaystyle \lim_{n\to +\infty} \lambda^{+c}(\mu_n) = \lambda^{+c}(\mu) \textrm{ and } \lim_{n\to +\infty} \lambda^{-c}(\mu_n) = \lambda^{-c}(\mu). $
\end{enumerate}
\end{proposition}

\begin{proof}[Proof of \Cref{prop.convergenceofexponents}]
Recall that $\phi$ is the geometric potential. Let $\nu$ be an invariant measure such that $\nu$-almost every points has one positive and one negative center Lyapunov exponent. Define the strong unstable geometric potential by 
\[
\phi^u(x) := -\log |\det D_xf |_{E^u}|,
\]
and for an invariant measure $\nu$, define the unstable pressure of $\nu$ by 
\[
P^u(\nu) := h^u(\nu) + \int \phi^u(x) \dd\nu(x).
\]
For $\nu$-almost every $x$, there is $e(x) \in \{1, \cdots, \dim E^u\}$ and Lyapunov exponents $\lambda_1^{u}(x)>\lambda_2^u(x)>\cdots \lambda_{e(x)}^u(x)>0$ associated to the directions subordinated to $E^u$, and let $E^u_{i}(x)$ denote the corresponding Oseledets bundle., for $i\in \{1, \cdots, e(x)\}$. As a consequence of the Ledrappier-Young entropy formula \cite{LedYoungII}, 
\begin{align*}
h^u_{\nu}(f)&= \displaystyle\int \sum_{i=1}^{e(x)} \gamma_i^u(x)\lambda_i^u(x)\der\nu(x)\\
h_{\nu}(f)&= \displaystyle h^u_{\nu}(f)+\int \gamma^{+c}(x)\lambda^{+c}(x)\der\nu(x)
\end{align*}
where $0\leq \gamma_i^u(x)\leq \dim E^u_i(x), 0\leq \gamma^{+c}(x) \leq 1$. Ruelle's inequality implies that
\[
h_\nu(f) \leq \displaystyle \int \sum_{i=1}^{e(x)} \left( \dim E^u_i(x) \right) \lambda^u_i(x) d\nu(x) + \int \lambda^{+c}(x) d\nu(x).
\]
Moreover, Oseledets Theorem also gives that 
\[
\displaystyle\int \phi^u(x) d\nu(x) = \displaystyle - \int \sum_{i=1}^{e(x)} \left(\dim E^u_i(x) \right) \lambda^u_i(x) \dd\nu(x).
\]
and 
\[
\displaystyle \int \phi(x) \dd\nu(x) = - \int \sum_{i=1}^{e(x)} \left( \dim E^u_i(x) \right) \lambda^u_i(x) \dd\nu(x) - \int \lambda^{+c}(x) \dd\nu(x),
\]
Therefore,
\[
\displaystyle P^u(\nu) = h^u_\nu(f) +\int \phi^u(x) d\nu(x) =\int \sum_{i=1}^{e(x)} (\gamma_i^u(x) - \dim E^u_i(x))\lambda_i^u \leq 0,
\]
since $\gamma^u_i(x) \leq \dim E^u_i(x)$ for $\nu$-almost every $x$. 
Thus,
\begin{align*}
P(\nu)&= \displaystyle h_{\nu}(f)- \int \sum_{i=1}^{e(x)} \left(\dim E^u_i(x) \right) \lambda_i^u(x)\der\nu(x) -\int \lambda^{+c}(x)\der\nu(x)\\
&=\displaystyle h^u_\nu(f) + \int \gamma^{+c}(x) \lambda^{+c}(x)d\nu(x)+ \int \phi^u(x)\der\nu(x) -\int \lambda^{+c}(x)\der\nu(x)\\
& = \displaystyle h^u_\nu(f)+\int \phi^u(x)\der\nu(x) + \int \left(\gamma^{+c}(x) - 1 \right) \lambda^{+c}(x) d\nu(x)\\
& \leq P^u(\nu) \leq 0,
\end{align*}
where we used that $\gamma^{+c}(x) -1 \leq 0$. We showed that
\begin{align}\label{eq:pressures}
0 \geq P^u(\nu) \geq P(\nu).
\end{align}

Suppose that $\mu_n$ is a sequence of measures such that $P(\mu_n)$ converges to $0$ and, up to a subsequence, suppose that they converge to a measure $\mu$.  Since 
\[
\displaystyle \lim_{n\to +\infty} h^u_{\mu_n}(f) + \int \phi^u \dd\mu_n = 0,
\]
 we have that
\[
\displaystyle \limsup_{n\to +\infty} h_{\mu_n}^u(f) = \limsup_{n\to +\infty} -\int \phi^u (x) \dd\mu_n(x).
\]
By the semi-continuity of unstable entropy \cite{Yangunstableentropy},
\[
\displaystyle \limsup_{n \to +\infty} h^u_{\mu_n}(f) \leq h^u_{\mu}(f).
\]
Therefore, 
\[
 \displaystyle \limsup_{n\to +\infty} -\int \phi^u (x) \dd\mu_n(x) =  \limsup_{n\to +\infty} h_{\mu_n}^u(f) \leq h_{\mu}^u(f) \leq -\int \phi^u(x) \dd\mu(x).
\]
Since the function $\phi^u$ is continuous and $\mu_n$ converges to $\mu$ in the weak*-topology, we conclude that 
\[
h^u_{\mu}(f) = -\int \phi^u(x) \dd\mu(x)
\]
and by Proposition $5.2$ in \cite{Yangunstableentropy}, $\mu$ is an $u$-Gibbs measure. 

For an invariant measure $\nu$, given a projective lift $\hat{\nu}$, we define its unstable pressure as
\[
\displaystyle P^u(\hat{\nu}) := h_{\hat{\nu}}^u(\hat{f}) + \int \phi^u(x) \dd\nu(x). 
\]

As explained in \Cref{sub:projetive_dynamics_continuity_of_exponents}, and since the measures $\mu_n$ and $\mu$ are hyperbolic, each has exactly two ergodic lifts to $\widehat{M}$ (denoted $\hat{\mu}_n^+, \hat{\mu}_n^-$ and  $\hat{\mu}^+, \hat{\mu}^-$, respectively).

\begin{claim}
$P^u(\hat{\mu}_n^+) = P^u(\mu_n).$
\end{claim}
\begin{proof}
Fix $n\in \Nat$ and let $\xi$ be an $u$-subordinated $\mu_n$-measurable partition and let $\widehat{\xi}$ and $\widehat{\xi}^{cu}$be the partitions given in \Cref{subsection.unstableentropy}.  Since $E^{+c}$ is invariant by unstable holonomies \Cref{lem:invariantEc}, for $\mu_n$-almost every point, we have that the measure $\widehat{\mu_n}_{(p,[v])}^{\widehat{\xi}^{cu}}$ is supported on a single piece of projective unstable manifold, and, therefore, it coincides with $\widehat{\mu_n}_{(p,[v])}^{\widehat{\xi}^u}$, for $\widehat{\mu_n}^+$-almost every $(p,[v])$.  Therefore,
\[
\widehat{\pi}_* \widehat{\mu_n}_{(p,[v])}^{\widehat{\xi_n}^u} = \widehat{\pi}_* \widehat{\mu_n}_{(p,[v])}^{\widehat{\xi}^{cu}} = \mu_{n,p}^{\xi}.
\]
By \Cref{thm.invarianceprinciple}, $h^u_{\widehat{\mu_n}}(\widehat{f}) = h^u_{\mu_n}(f)$.  Hence,
\[
\displaystyle P^u(\hmu_n^+) = h^u_{\widehat{\mu_n}^+}(\widehat{f}) + \int \phi^u \dd\mu_n = h^u_{\mu_n}(f) + \int \phi^u \dd\mu_n = P^u(\mu_n). \qedhere
\]
\end{proof}

Take any limit measure $\hat{\mu}$ of the sequence $\hat{\mu}_n^+$.  We will show that $\hat{\mu} = \hat{\mu}^+$. 

By Theorem \ref{thm.semicontinuityuentropy}, we have that
\[
\displaystyle 0 \geq h^u(\hat{\mu}) + \int \phi^u \dd\mu \geq \limsup_{n\to +\infty} h^u(\hat{\mu}_n^+) + \int \limsup_{n\to +\infty} \phi^u \dd\mu_n = \limsup_{n\to +\infty} P^u(\hat{\mu}_n^+) = 0.
\]
Hence, 
\[
P^u(\hat{\mu}) = 0.
\]
Which implies that $h^u_{\hat{\mu}}(\hf) =- \int \phi^u \dd\mu = h^u_{\mu}(f)$, since $\mu$ is a $u$-Gibbs measure.  By \Cref{thm.invarianceprinciple}, we have that $\hat{\mu}$ is a projective $u$-Gibbs measure. Since $f$ is expanding on average and $\Gap >0$, by \Cref{mainthm:TheoremD},  we have that the only projective $u$-Gibbs measure that projects to $\mu$ is given by $\hat{\mu}^+$.  We conclude that $\lim_n \hat{\mu}_n^+ = \hat{\mu}^+$.

Since the top center exponent is given by the integral of the function $\psi(p,E) = \log \|D_pf|_E\|$, which is continuous in $\widehat{M}$, we obtain that
\[
\displaystyle \lim_{n \to +\infty} \lambda^{+c}(\mu_n) = \lim_{n\to +\infty} \int \psi \dd\hat{\mu}_n^+ = \int \psi \dd\hat{\mu}^+ = \lambda^{+c}(\mu).
\]
Which gives the continuity of the top exponent. Since for an invariant measure $\nu$ we have $\lambda^{+c}(\nu) + \lambda^{-c}(\nu) = \int \log |\det Df|_{E^c}| \dd\nu$, which is continuous with respect to $\nu$, we conclude that
\[
\lim \lambda^{-c}(\mu_n) = \lambda^{-c}(\mu).
\]
This finishes the proof of \Cref{prop.convergenceofexponents}.\qedhere
\end{proof}

\begin{proof}[Proof of \Cref{thm.sprueap}]

Let $\mu$ be an SRB measure carried by $X$, the corresponding Borel homoclinic class. Since $\mu$ is an $SRB$ measure, we have that $P(\mu) = 0$.  Fix $\chi>0$ small such that $4\chi \in (0, \min\{\CEA, \Gap\}$, and fix $\varepsilon>0$ small. Suppose that SPR fails for $(\chi,\varepsilon)$, then, there exist a sequence of ergodic measures $\mu_n$ carried by $X$ such that $P(\mu_n) \to 0$ and $\mu_n (\Delta_{n,\chi, \varepsilon}) \to 0$. Up to a subsequence, we may assume that $\mu_n$ converges to a measure $\mu'$ carried by $X$. 
\Cref{prop.convergenceofexponents} implies that $\mu'$ is $u$-Gibbs and $\lambda^{\pm c}(\mu_n) \to \lambda^{\pm c}(\mu')$. Applying \Cref{prop.generaluniformpesin} to $\mathcal{K} = \{\mu_n\} \cup \{\mu'\}$, implies that there exists a constant $C$ such that $\mu_n(\Delta_{C,\chi,\varepsilon})>\frac{1}{2}$ for all $n$, which is a contradiction. 
\end{proof}

\begin{proof}[Proof of \Cref{mainthm:TheoremC}]

    Suppose that $f\in \PHb[2]$ is expanding on average and satisfies $\Gap>0$. Let $\mu$ be an ergodic SRB measure. Let $X$ be its Borel homoclinic class carrying $\mu$. By \Cref{thm.sprueap}, $f$ is SPR on $X$. \Cref{mainthm:TheoremC} follows from \Cref{thm.sprexponential}. 
\end{proof}

\section{Examples of expanding on average partially hyperbolic systems}
\label{sec:Examples}

In this section, we will present several examples of partially hyperbolic diffeomorphisms with the expanding on average property, and we will make some remarks about their abundance. Recall the discussion on the finitary versions of the expanding on average properties (see \Cref{def:finitaryutillings} and the text after).

\subsection{Skew products with fiber maps small perturbations of the identity and conservative H\'enon-type maps.}\label{subsec:henonlike}

Our first examples are skew products over hyperbolic maps such that the maps on the center fibers are arbitrarily close to the rotation by $-\frac{\pi}2$, which makes the fourth iterates arbitrarily close to the identity. In the same time, the center fiber maps of our examples are conservative H\'enon-like maps.

In order to simplify the computations, we will consider the case when the base map is a linear multiplication by the natural number $\lambda\geq 4$ on the circle. Equivalently, one can take the base map the solenoid induced by the multiplication by $\lambda$ on the circle: $U_{\lambda}:T_3\to T_3,$ 
\[
  U_{\lambda}(u,v,t)=(\eta_1\cos t+\eta_2u,\eta_1\sin t+\eta_2v,\lambda t).
\]
where $T_3=\mathbb{D}\times\mathbb{S}^1$ is the solid torus. For $0<\eta_2<<\eta_1$ sufficiently small, its maximal invariant set $\Lambda_{\lambda}$ is a hyperbolic attractor, and $U_{\lambda}|\Lambda_{\lambda}$ factors over the multiplication by $\lambda$ on $\mathbb T^1$.

For parameters $\ep>0, t\in(0,2\pi)$ we consider the conservative H\'enon-like map $g_{\ep,t}:\Tor^2\rightarrow\Tor^2$, 
\[
g_{\ep,t}(x,y)=(y+\ep\sin(x+t),-x).  
\]

We define the skew-product $f_{\ep}:\Tor^3\rightarrow\Tor^3$ over the multiplication by $\lambda$ on $\mathbb T^1$: 
\begin{align}
    f_{\ep,\lambda}(x,y,t)=(g_{\ep,t}(x,y),\lambda t)=(y+\ep\sin(x+t),-x,\lambda t).\label{eq.henonlike}
\end{align}

Observe that $f_{\ep,\lambda}$ lifts to an invertible skew product $F_{\ep,\lambda}$ on $T_3\times\mathbb T^2$, and for $\lambda$ large enough we have the partially hyperbolic attractor $\Lambda_{\lambda}\times\mathbb T^2$. By a mild abuse of language, we will say that $f_{\ep,\lambda}$ verifies the expansion on average property, provided that its lift $F_{\ep,\lambda}$ does (of course restricted to its partially hyperbolic attractor).

We have the following result.

\begin{proposition}\label{prop:Henon}
    There exists $\ep_0>0$ that satisfies the following: for any $0<\ep<\ep_0$, there exists a $\lambda_{\ep}\in\mathbb N$ such that for any $\lambda>\lambda_{\ep}$, the map $f_{\ep,\lambda}$ satisfies is expanding on average.
\end{proposition}

\begin{remark}
  Note that $f^4_{\ep,\lambda}$ is a perturbation of $f^4_{0,\lambda}$ which acts by identity on the center fibers. If $f_{\ep,\lambda}$ has the expanding on average property, the same holds for $f^4_{\ep,\lambda}$.

  The existence of non-uniformly perturbations of partially hyperbolic systems acting as the identity in the center were studied by K. Marin in \cite{Marin2016}. The present method is different from hers, and permits to obtain stronger statistical consequences.
\end{remark}

\begin{proof}
Let us use some notations. For $\theta\in \mathbb R$, let $c(\theta)=\cos\theta, s(\theta)=\sin\theta$. Define the quantities
\begin{align*}
    &C=C(y,\lambda,t)=y+\lambda t\\
    &B=B(x,t)=x+t\\
    &A=A(x,y,\lambda,\ep,t)=C+\ep\sin B=C+\ep s(B).
 \end{align*}

Consider a direction $v_a=(1,a)$. We are interested in the action of the derivative of $f_{\ep,\lambda}$ on $v_a$ (since the center direction corresponds to the first two coordinates, we will make an abuse of notation and consider that $Df_{\ep,\lambda}$ acts on $v_a$).
\begin{align*}
&D_{(x,y)}g_{\ep,t}=\begin{bmatrix}
    \ep c(B) & 1\\
    -1 & 0 
    \end{bmatrix}\Rightarrow D_{(x,y,t)}f_{\ep,\lambda}v_a=\begin{bmatrix}
    a+\ep c(B)\\
    -1  
    \end{bmatrix}=\begin{bmatrix}
    u_{\ep,a}\\
    -1  
    \end{bmatrix}\\
    \intertext{thus}
    &D_{(x,y,t)}f^2_{\ep,\lambda}v_a=\begin{bmatrix}
    \ep c(A) & 1\\
    -1  & 0 
    \end{bmatrix}\begin{bmatrix}
    u_{\ep,a}\\
    -1  
    \end{bmatrix}=\begin{bmatrix}
    \ep c(A)u_{\ep,a}-1\\
    -u_{\ep,a}  
    \end{bmatrix}=\begin{bmatrix}
    \ep c(A)(a+\ep c(B))-1\\
    -a-\ep c(B)  
    \end{bmatrix}.
\end{align*}

\begin{lemma}\label{lem:linearunstable}
    There exists $K>0$, $\ep_0>0$ so that for every $0<\ep<\ep_0$, for all $(x,y)\in \Tor^2,\ a\in\mathbb R,$ and all $\lambda\in\mathbb N$, $\lambda\geq 4$, we have
\[ 
E(x,y,\ep,\lambda,a)=\int_0^{2\pi}\log\frac{\|D_{(x,y,t)}f^2_{\ep,\lambda}v_a\|}{\|v_a\|}\dd t \geq K\ep^4.
\]
\end{lemma}

\begin{proof}
We remark first that
\[
  E(x,y,0,\lambda,a)=0,\ \forall x,y,\lambda,a.
\]

We start writing

\begin{align*}
N_{\ep,a} &:= \norm{D_{(x,y,t)}f^2_{\ep,\lambda}v_a}^2=(1-\ep c(A)u_{\ep,a})^2+u_{\ep,a}^2=[1-\ep c(A)(a+\ep c(B))]^2+(a+\ep c(B))^2\\
&=1+a^2+2a\ep(c(B)-c(A))+\ep^2 (c(B)^2+a^2 c(A)^2-2c(B)c(A))+\\
&\ \ \ \ +2a\ep^3 c(B)c(A)^2+\ep^4 c(B)^2c(A)^2 
\end{align*}

Note that we can expand
\begin{align*}
c(A)&=\cos(C+\ep s(B))=c(C)-\ep s(B)s(C)-\frac{\ep^2}2s(B)^2c(C)+\frac{\ep^3}6s(B)^3s(C) + O(\ep^4)\\
c(A)^2&=c(C)^2-2\ep s(B)s(C)c(C)+\ep^2s(B)^2(s(C)^2-c(C)^2)+O(\ep^3).
\end{align*}
and therefore (if $\rho_a=\norm{v_a}^2=1+a^2$)
\begin{align*}
\frac{N_{\ep,a}}{\rho_a}&=1+\frac{2a\ep}{\rho_a}\left[c(B)-c(C)+\ep s(B)s(C)+\frac{\ep^2}2s(B)^2c(C)-\frac{\ep^3}6s(B)^3s(C)\right]+\\
&+\frac{\ep^2}{\rho_a}\left[c(B)^2+a^2c(C)^2-2a^2\ep s(B)s(C)c(C)+a^2\ep^2s(B)^2(s(C)^2-c(C)^2)-\right.\\
&\left.-2c(B)c(C)+2\ep s(B)c(B)s(C)+\ep^2s(B)^2c(B)c(C)\right]+\\
&+\frac{2a\ep^3}{\rho_a}c(B)[c(C)^2-2\ep s(B)s(C)c(C)]+\frac{\ep^4}{\rho_a}c(B)^2c(C)^2+O(\ep^5)\\
&=1+\alpha_1\ep+\alpha_2\ep^2+\alpha_3\ep^3+\alpha_4\ep^4+O(\ep^5);
\end{align*}
where
\begin{align*}
    \al_1&=\frac{2a(c(B)-c(C))}{\rho_a}\\
    \al_2&=\frac{2as(B)s(C)+c(B)^2+a^2c(C)^2-2c(B)c(C)}{\rho_a}\\   \al_3&=\frac{as(B)^2c(C)-2a^2s(B)s(C)c(C)+2s(B)c(B)s(C)+2ac(B)c(C)^2}{\rho_a}\\
    \al_4&=\frac{-2as(B)^3s(C)+6a^2s(B)^2(s(C)^2-c(C)^2)+6s(B)^2c(B)c(C)-12as(B)c(B)s(C)c(C)+}{6\rho_a}\\
    &\ \ \ \ \ \frac{+6c(B)^2c(C)^2}{}
\end{align*}

Recall that
$$\log(1+x)=x-\frac{x^2}2+\frac{x^3}3-\frac{x^4}4+O(x^5),$$
then
\begin{align*}
\log\frac{N_{\ep,a}}{\rho_a}&=\log(1+\al_1\ep+\al_2\ep^2+\al_3\ep^3+\al_4\ep^4+O(\ep^5))\\
&=(\ep\al_1+\ep^2\al_2+\ep^3\al_3+\ep^4\al_4)-\frac{\ep^2\al_1^2+2\ep^3\al_1\al_2+\ep^4\al_2^2+2\ep^4\al_1\al_3}2+\frac{\ep^3\al_1^3+3\ep^4\al_1^2\al_2}3-\\
&\ \ \ \ \ -\frac{\ep^4\al_1^4}4+O(\ep^5)\\
&=\ep\al_1+\ep^2\left( \al_2-\frac{\al_1^2}2\right)+\ep^3\left( \al_3-\al_1\al_2+\frac{\al_1^3}3\right)+\\
&\ \ \ \ \ \ +\ep^4\left( \al_4-\frac{\al_2^2+2\al_1\al_3}2+\al_1^2\al_2-\frac{\al_1^4}4\right)+O(\ep^5).
\end{align*}
Observe that 
\begin{align*}
E(x,y,\ep,\lambda,a)=\int_0^{2\pi} \frac{1}{2}\log \frac{N_{a,\ep}}{\rho_a}\dd t.
\end{align*}
Then the derivatives with respect to $\ep$ in 0 are
\begin{align*}
\frac{\partial E}{\partial \ep}(x,y,0,\lambda,a)&=\int_0^{2\pi}\frac 12\al_1\dd t=\int_0^{2\pi}\frac{a(c(B)-c(C))}{1+a^2}\dd t=0\\
\frac{\partial^2 E}{\partial \ep^2}(x,y,0,\lambda,a)&=\int_0^{2\pi}\al_2-\frac{\al_1^2}2\dd t\\
&=\int_0^{2\pi}\frac{2as(B)s(C)+c(B)^2+a^2c(C)^2-2c(B)c(C)}{1+a^2}-\frac{2a^2(c(B)-c(C))^2}{(1+a^2)^2}\dd t\\
&=\int_0^{2\pi}c(B)^2\left( \frac{1+a^2}{1+a^2}-\frac{4a^2}{(1+a^2)^2} \right)\dd t=\pi\frac{(1-a^2)^2}{(1+a^2)^2}\geq 0,
\end{align*}
Since 
\[
\int_0^{2\pi}c(B)dt=\int_0^{2\pi}c(C)\dd t=0,
\]
\[
    \int_0^{2\pi} c(B)^2 \dd t=\int_0^{2\pi} c(C)^2 \dd t=\pi,
\]
and by orthogonality of the Fourier polynomials,
\[
    \int_0^{2\pi} c(B)c(C) \dd t=\int_0^{2\pi} s(B)s(C) \dd t=0.
\]

In particular,
\[
    \frac{\partial^2 E}{\partial \ep^2}(x,y,0,\lambda,a)>0\quad a\neq \pm 1
\]

Let us remark here that the formula works also for $a=\pm\infty$, corresponding to the vertical vector $v_a$; in this case the second derivative is $\pi$ (the limit when $a\to\infty$).

Now we address the special directions $a=\sigma$, where $\sigma=\pm 1$. 

First we see that the third derivative is
\[
   \frac{\partial^3 E}{\partial \ep^3}(x,y,0,\lambda,a)=\int_0^{2\pi}3\al_3-3\al_1\al_2+\al_1^3\dd t 
\]
is an integral of linear combinations of products of 3 trigonometric functions belonging to the set $\{\sin(x+t),\cos(x+t),\sin(y+\lambda t),\cos(y+\lambda t)\}$. It is not difficult to see that any integral of the product of such 3 functions is zero (for example by symmetry, the positive part is equal to the negative part), so we have
\[
\frac{\partial^3 E}{\partial \ep^3}(x,y,0,\lambda,a)=0.
\]

For the fourth derivative, for $\sigma=\pm1$ we compute
\begin{align*}
\frac{\partial^4 E}{\partial^4\ep}(x,y,0,\lambda,\sigma)= \int_0^{2\pi}12\al_4-12\al_1\al_3-6\al^2_2+12\al_1^2\al_2-3\al_1^4 \dd t
\end{align*}
We will use the relations
\begin{align*}
&\int_0^{2\pi} c(B)^4 \dd t = \int_0^{2\pi} c(C)^4 \dd t=\frac{3\pi}{4}\\
&\int_0^{2\pi} c(B)^2 c(C)^2 \dd t = \int_0^{2\pi} s(B)^2c(C)^2 \dd t= \int_0^{2\pi} s(B)^2s(C)^2 \dd t= \int_0^{2\pi} c(B)^2s(C)^2 \dd t=\frac {\pi}2\\
\end{align*}
while all the integrals where some trigonometric function has odd exponent are zero (same symmetry argument for example).

From here it follows that
\begin{align*}
\int_0^{2\pi} 12\al_4 \dd t &= \int_0^{2\pi} 6s(B)^2s(C)^2-6s(B)^2c(C)^2+6c(B)^2c(C)^2dt=3\pi.
\end{align*}

Next, 

\[
   \int_0^{2\pi}  12\al_1\al_3 \dd t =\int_0^{2\pi}12c(B)^2c(C)^2-6s(B)^2c(C)^2\dd t=3\pi.
\]

For the next term, 

\begin{align*}
\int_0^{2\pi}6\al^2_2\dd t=\int_0^{2\pi}\left(6s(B)^2s(C)^2+\frac 32c(B)^4+\frac 32c(C)^4+9c(B)^2c(C)^2\right)\dd t=\frac{39\pi}{4}.
\end{align*}

We continue,
\[
    \int_0^{2\pi}12\al_1^2\al_2 \dd t=\int_0^{2\pi}\left(6c(B)^4+6c(C)^4+36c(B)^2c(C)^2\right)\dd t =27\pi.
\]

Finally,
\[
\int_0^{2\pi} 3\al_1^4 \dd t=\int_0^{2\pi}\left(3c(B)^4+3c(C)^4+18c(B)^2c(C)^2\right)\dd t=\frac{27\pi}{2}
\]
Putting everything together we get
\begin{align*}
\frac{\partial^4 E}{\partial\ep^4}(x,y,0,\lambda,\sigma)=3\pi-3\pi-\frac{39\pi}4+27\pi-\frac{27\pi}2=\frac{15\pi}4>0
\end{align*}

We observe that $\frac{\partial^4 E}{\partial\ep^4}(x,y,0,\lambda,a)$ is a function which depends only on $a$ and is continuous, so there is $\delta>0$ such that if $\frac{|a^2-1|}{a^2+1}<\delta$ then $\frac{\partial^4 E}{\partial\ep^4}(x,y,0,\lambda,a)>3\pi$.

We also observe that $\frac{\partial^k E}{\partial\ep^k}(x,y,\ep,\lambda,a)$ is uniformly bounded by some constant $D_k$ for all $\ep<1$, independent of $x,y,a,\lambda$.

We deduce that for every $x,y,a,\lambda$ and $\ep<1$, there exists $\ep'$ such that
\[
E(x,y,\ep,\lambda,a)=\frac{(a^2-1)^2}{2(a^2+1)^2}\pi\ep^2+\frac 1{24}\frac{\partial^4 E}{\partial^4\ep}(x,y,0,\lambda,a)\ep^4+\frac 1{120}\frac{\partial^5 E}{\partial^4\ep}(x,y,\ep',\lambda,a)\ep^5.
\]
If $\frac{|a^2-1|}{a^2+1}<\delta$ we have
\[
E(x,y,\ep,\lambda,a)\geq\frac 1{24}3\pi\ep^4+\frac 1{120}\frac{\partial^5 E}{\partial^4\ep}(x,y,\ep',\lambda,a)\ep^5\geq\frac{\pi}8\ep^4-\frac{D_5}{120}\ep^5,
\]
while if $\frac{|a^2-1|}{a^2+1}>\delta$ we have
\[
E(x,y,\ep,\lambda,a)\geq\frac{\delta^2}2\pi\ep^2-\frac{D_4}{24}\ep^4-\frac{D_5}{120}\ep^5
\]
From here we can find an $\ep_0$ such that for any $\ep<\ep_0$ and any $x,y,a,\lambda$ we have
\[
  E (x,y,\ep,\lambda,a)\geq \frac{\ep^4}4.
\]
This proves the lemma.
\end{proof}

Now let us continue with the proof of the Proposition. We will consider the lift $F_{\ep,\lambda}$ of $f_{\ep,\lambda}$ to $T_3\times\mathbb T^2$ and its partially hyperbolic attractor $\Lambda_{\lambda}\times\mathbb T^2$. We fix the increasing partition $\xi$ subordinated to the unstable foliation of $F_{\ep,\lambda}$ such that each element $P$ of the partition projects homeomorphically to the interval $(0,2\pi)$. If $P\in\xi$, $p\in P\in\Lambda_{\lambda}\times\mathbb T^2$, $v_a\in E^c_p$, $v_a^*=\frac{v_a}{\|v_a\|}$ its normalization, recall that
\[
I_2(p,v_a^*,F_{\ep,\lambda})=\int_P\log\|D_qF^2_{\ep,\lambda}\chi_{q,p}(v_a^*)\|\dd \mu_P^u(q).
\]
One can check that if we parametrize $P$ with respect to $t$, $q=q(t)$, then $d\mu_P^u(q)=\frac 1{2\pi}\dd t$, so the integral can be rewritten
\[
I_2(p,v_a^*,F_{\ep,\lambda})=\frac 1{2\pi}\int_0^{2\pi}\log\|D_{q(t)}F_{\ep,\lambda}^2\chi_{q(t),p}(v_a^*)\|\dd t.
\]

Assume that $p=(u,v,t_0,x,y)$ with $(u,v,t_0)\in\Lambda_{\lambda}$ in the base and $(x,y)\in\mathbb T^2$ in the central fiber. Consider $v_a\in E^c_p$. Denote $(x(t),y(t))=Pr_{\mathbb T^2}q(t)$ the projection of the unstable leaf $q(t)$ on the center fiber $\mathbb T^2$. Then $D_{q(t)}F_{\ep,\lambda}^c=D_{(x(t),y(t),t)}f_{\ep,\lambda}^c$.

When $\ep=0$ we have
\[
I_2(p,v_a^*,F_{\ep,\lambda})=\frac 1{2\pi}E(x,y,0,\lambda,a)=0.
\]
In general we have
\begin{equation}\label{eq:integrals}
I_2(p,v_a^*,F_{\ep,\lambda})-\frac 1{2\pi}E(x,y,\ep,\lambda,a)=\frac 1{2\pi}\int_0^{2\pi}\log\frac{\|D_{(x(t),y(t),t)}f^2_{\ep,\lambda}\chi_{q(t),p}(v_a^*)\|}{\|D_{(x,y,t)}f^2_{\ep,\lambda}(v_a^*)\|}\dd t.
\end{equation}

\begin{lemma}\label{lem:approximation}
    When $\lambda$ tends to infinity, $x(t)$ converges uniformly to $x$, $y(t)$ converges uniformly to $y$ and $\chi_{p,q(t)}$ converges uniformly to the identity.
\end{lemma}

\begin{proof}
To see that $(x(t),y(t))$ converges uniformly to $(x,y)$ it is enough to remark that the unstable manifold is inside an horizontal unstable cone which gets smaller when $\lambda$ goes to infinity.

The fact that $\chi_{p,q(t)}$ converges uniformly to the identity follows from the formula of $\chi_{p.q(t)}$ as the projectivization of an infinite product of matices (derivatives).

\end{proof}

The proof of the Proposition \ref{prop:Henon} follows from the formula \ref{eq:integrals} and Lemmas \ref{lem:linearunstable} and \ref{lem:approximation}

\end{proof}

\subsection{Random Standard (or standard-like) Maps}

Here we present other examples of skew products over Anosov with the expanding on average property. Unlike the previous section, now the fiber maps in the centers are far away from identity (they exhibit strong shears or hyperbolic regions).

\subsubsection{The Berger-Carrasco example}\label{subsec:bcexample}

Consider the standard map
 \[
     s_N(x,y)=\bigl(2x-y+N\sin x,x\bigr),\quad N\in\mathbb N.
 \]
 Fix a hyperbolic matrix \(A\in\mathtt{SL}(2,\mathbb Z)\) with eigenvalues $0<\lambda<1<\lambda^{-1}$ and let
\[
    P_1\colon\Tor^2\longrightarrow\Tor^2,
    \qquad
    P_1(z,w)=(z,0),
\]
denote projection onto the first coordinate. For \(N\geq 1\), define
\[
    f_N\colon\Tor^2\times\Tor^2 \longrightarrow \Tor^2\times\Tor^2
\]
by
\begin{equation}\label{eq.bcexample}
    f_N(x,y,z,w)=\left(s_N(x,y)+P_1\bigl(A^N(z,w)\bigr),
        A^{2N}(z,w)\right).
\end{equation}
Thus, \(f_N\) is obtained by coupling the standard map with the fast Anosov automorphism \(A^{2N}\), the coupling being introduced
through the first coordinate of \(A^N(z,w)\).

The map \(f_N\) is real analytic and preserves Lebesgue measure. Its derivative has the block form
\[
    Df_N(x,y,z,w) =
    \begin{bmatrix}
        Ds_N(x,y) & P_1\circ A^N \\
        0          & A^{2N}
    \end{bmatrix}.
\]
In particular,
\[
    E^c=\mathbb R^2\times\{0\}
\]
is invariant, and the restriction of \(Df_N\) to \(E^c\) coincides with \(Ds_N\). This allows to think of $f_N$ as a ``random standard map'', the randomness driven by the dynamics in the base. Observe that for $N$ large, the expansion and contraction rates of \(A^{2N}\) dominate the derivative of \(s_N\), and the map \(f_N\) is partially hyperbolic.

\begin{proposition}\label{proposition:bcisea}
    There exists $N_0\in\mathbb N$ such that, for any $N\geq N_0$, the map $f_N$ is expanding on average.
\end{proposition}

\begin{proof}
    We will give a sketch of the proof. Let $e^u$ be the unstable unit vector of $A$ and $P_1e^u$ its projection to the first coordinate. Following \cite{Berger2014} we say that $\gamma$ is a u-curve if it is contained in an unstable manifold of $f_N$ and projects to an unstable segment of $A^{2N}$ of length $\frac 1{\lambda^N\|P_1e^u\|}$. In \cite{Berger2014} the authors prove the following fact: assuming the $N$ is sufficiently large, then fixing any u-curve $\gamma$ of $f_N$, and fixing any unit vector field $X$ along $\gamma$ and with values in $E^c$, which is $(C_X,\frac 12)$-H\"older, with $C_x=20N^2\lambda^N$, for all large enough $n$ we have
\[
\frac 1{n}\int_{\gamma}\log\|D_qf^nX(q)\|\dd\mu_{\gamma}^u(q)>\frac 1{20}\log N.
\]    
(The authors state that there exists a vector field $X$ with this expanding property, but they prove, in fact, that every vector field $X$ satisfying the required properties has the expanding property).

The maps $f_N$ are center-bunched for $N$ large, so we have the holonomy $\chi_{q,p}$ along unstable leaves. The center bundle $E^c$ of $f_N$ is smooth, so in this case \cite{SAGHIN_2024} proves, in fact, that $\chi_{q,p}$ is Lipschitz with respect to $p,q$: $\|\chi_{q,p}v-v\|\leq Cd(p,q)$ for some constant $C>0$. Given an u-curve $\gamma$, there is some preimage $f_N^{-k}\gamma$ small enough such that the vector fields given by the unstable holonomy are $(C_X,\frac 12)$-H\"older. \cite{Berger2014} shows that the $(C_X,\frac 12)$-H\"older property of unit vector fields along unstable curves is invariant under $Df_N$, and since the unstable holonomy is also invariant under $Df$, we get that all the vector fields $\chi_{q,p}v$ are $(C_X,\frac 12)$-H\"older along all u-curves.

We obtain that for any u-curve $\gamma$, any $p\in\gamma$, $v\in E^c(p)$, we have that for all $n$ large enough
\[
\frac 1{n}\int_{\gamma}\log\|D_qf^n\chi_{q,p}v\|\dd\mu_{\gamma}^u(q)>\frac 1{20}\log N.
\]

Then $f_N$ has the expanding on average property.
    
\end{proof}

In \cite{Obata2023}, SRB measures for dissipative perturbations of $f_N$ were studied. The author shows that in a $C^2$-neighborhood of $f_N$ there is at most one SRB measure and using measure rigidity argument inspired in \cite{Brown_2017}, he shows that ``most fibered'' perturbations admits an SRB measure. Using our results from this paper, we are able to extend that to the entire neighborhood. 

\begin{corollary}\label{cor:bchasitall}
    Let $N_0$ be as in \Cref{proposition:bcisea}. Then, for every $N\geq N_0$, there exists $\mathcal{U} \subset \PHb[2]$ a $C^2$-neighborhood of $f_N$ such that any $g\in \mathcal{U}$ has a unique SRB measure $\mu_g$. Moreover, $\mu_g$ is Bernoulli, it has full basin, it is exponentially mixing and verifies all of the strong statistical properties given by \Cref{thm.sprexponential}.
\end{corollary}
\begin{proof}
   By \cite{Obata2018}, $f_N$ is Bernoulli. The result follows from \Cref{thm:Theorem E}.
\end{proof}

\subsubsection{Other related examples}

The main idea used in \cite{Berger2014} is the following. Assume that $f:M\rightarrow M$ is a partially hyperbolic diffeomorphism. Assume that in $M$ we have a large (non-invariant) hyperbolic set $H$, where $Df|_{E^c}$ is strongly hyperbolic, and a small critical set $C$ where the map may not be hyperbolic, but we have some uniform control from above and below for $Df|_{E^c}$. Assume that any large enough unstable piece crossing the manifold $M$ spends much more time inside the hyperbolic part $H$ that in the critical part $C$. 

Assume that $\gamma$ is a large unstable piece and $\chi_{q,p}(v)$ is a unit vector field on $\gamma$ given by the unstable holonomy. If the vector field $\chi_{q,p}(v)$ is inside the expanding cone $C^u$ of $Df_{E^c}$ in the hyperbolic part $H$, then it gets (strongly) expanded. This together with the control on the critical part gives the expanding on average of $\chi_{q,p}v$ along $\gamma$ in one iterate:
\[
\int_{\gamma}\log\|D_qf\chi_{q,p}(v)\|\dd\mu_{\gamma}^u(q)\gg 0.
\]
If $\chi_{q,p}(v)$ is not in the unstable cone $C^u$, then it may get contracted. We need a mechanism to move the vectors inside $C^u$, and this can be done in the following way:
\begin{itemize}
    \item Choose different regions $H_1,H_2,\dots H_k\subset H$ with (considerably) different stable directions in each region. Thus, for any given central vector $v$, maybe we have contraction in some $H_i$, but in others we have expansion, or at least the pushed forward of the vector, $Dfv$, moves inside the expanding cone $C^u$.
    \item Obtain good control on the variation of the vector fields $\chi_{q,p}(v)$ along large enough unstable pieces, in order to apply the previous argument. This can be done by choosing a strong enough expansion on the base, thus making the holonomy $\chi_{q,p}$ close to the identity.
\end{itemize} 

This method was used in various papers like \cite{Carrasco2020}, \cite{extendedflex}, \cite{Andersson2022}, etc., in order to obtain non-uniformly hyperbolic maps. It is flexible enough and can be easily adapted in order to obtain various families of skew products with the expanding on average property, inside various homotopy classes, or exhibiting various $u$-Gibbs measures.

\subsection{Non-compact center leaves}\label{subsec:noncompactcenter}

Here we will provide examples of partially hyperbolic diffeomorphisms with the expanding on average property and with non-compact center leaves. The examples are obtained perturbing 4-dimensional linear Anosov diffeomorphisms with a 2-dimensional weak unstable bundle (considered as a center bundle, and which is clearly not compact).

\begin{proposition}
    Let $L$ be a hyperbolic automorphism of $\mathbb T^4$ with the eigenvalues $\lambda_u>1>\lambda_{ws}>\lambda_{ms}>\lambda_{ss}$. Assume that $\lambda_{ws}\lambda_{ms}>\lambda_{ss}$. Then there exists a $C^{\infty}$ diffeomorphism $f$ of $\mathbb T^4$ and $N\in\mathbb N$ such that $L^N\circ f$ is a partially hyperbolic diffeomorphism with the expanding on average property.
\end{proposition}

\begin{proof}
    The proof of this result is basically contained in the proof of Theorem B of \cite{extendedflex}. The conclusion of that Theorem B is that $L^N\circ f$ is non-uniformly hyperbolic for the volume, however one in fact proves that $f$ is expanding on average.

    A sketch of the proof is the following. $f_t$ is a family of maps which preserve the center planes corresponding to the eigenvalues $\lambda_{ws},\lambda_{ms}$. Using the semi-conjugacy between $L^N\circ f_t$ and $L^N$, one can obtain an increasing partition $\xi$ subordinated to $W^u_{L^N\circ f_t}$ by pulling back a Markov partition for $L$ (and also $L^N$). Iterating forward if necessary, we can assume that the atoms of $\xi$ are large enough so they wrap around the torus several times (but still uniformly bounded from above).

    Each map $f_t$ is similar to a standard map, in the sense that it creates a large set on $\mathbb T^4$ with very strong hyperbolicity, and a small critical set where there is uniform control on the derivative. The authors in \cite{extendedflex} prove the following result. If $N$ is sufficiently large, and $t$ satisfies some specific conditions (for example $t=\lambda_{ws}^{-N(1+\ep)}$ with $\ep$ small enough) then we have the following: For any $P$ an element of $\xi$, and any $X$ an unit vector field on $P$ with values in $E^c$ and $l$-Lipschitz with $l=l_Lt^2\lambda_{ws}^{2n}\lambda_{ss}^n$ ($l_L$ depends only on $L$), we have 
\[
\liminf_{n\rightarrow\infty}\frac 1{n}\int_{P}\log\|D_q(L^N\circ  f_t)^nX(q)\|\dd\mu_P^u(q)>C.
\]
for some constant $C$ independent of $P,X$. The fact the $X$ is $l$-Lipschitz is used for two things: the $l$-Lipschitz unit vector fields are invariant under $D(L^N\circ f_t)$, and the $l$-Lipschitz conditions implies that $X$ has a small enough variation along the elements $P$ of the partition $\xi$. Both these conditions are satisfied by unit vector fields given by the unstable holonomy $\chi_{q,p}$ if $N$ is large enough, so we get a similar result for $\chi_{q,p}$: for any $P\in\xi$ and any $v\in E^c(p),\ p\in P$ we have
\[
\liminf_{n\rightarrow\infty}\frac 1{n}\int_{P}\log\|D_q(L^N\circ f_t)^n\chi_{q,p}v\|\dd\mu_P^u(q)>C.
\]

    This implies that $L^N\circ f_t$ is expanding on average.
\end{proof}

\section*{Acknowledgements}
\phantomsection
\addcontentsline{toc}{section}{Acknowledgements}

PDC was partially supported by CNPq (Produtividade em Pesquisa) 305197/2025-8 and Projeto Universal CNPq 401737/2025-0. 

DO was partially supported by the National Science Foundation under Grant No. DMS-2349380. 

PDC and DO were partially supported by FAPEMIG project APQ-05983-25. 

RS was partially supported by the ANID project Fondecyt Regular 1230632. 

JY was partially supported by CAPES–Finance Code 001, CNPq, FAPERJ, PRONEX. 

PDC and JY were partially supported by CAPES MATH-AMSUD project ``Entropy and Dimension in Non-uniformly Hyperbolic Dynamics''.

DO benefited from conversations with Sylvain Crovisier especially about showing the SPR property. 

AI Use: Both ChatGPT (model Astra) and Claude (model Opus) were used to revise the manuscript and improve the writing. These models were not used for assisting with the mathematical content.


\appendix

\section{Uniform partitions and \texorpdfstring{$u$}{u}-tilings}
\label{sec:appendix}

\subsection{A measurable partition adapted to every \texorpdfstring{$u$-}{u}Gibbs measure.}
\label{ssec:uniformmeasurable}
It will useful to have measurable partitions with respect to every $u$-Gibbs measure. Let $f\in\PH[1+\alpha], \alpha>0$.

\begin{theorem}\label{thm:totaluGibbspartition}
  Given $R>0$ there exists a measurable partition $\eta$ of $M$, $S$ a compact submanifold of codimension-one, a Borel function $r:M\to [0,1]$, and a constant $C>0$ which verify the following.
  \begin{enumerate}
    \item $\eta$ is subordinated to the unstable partition of $f$, and for every $E\in\eta, \partial E \subset S$.
    \item $\diam_n\eta(x)\leq R$, for every $x\in M$.
    
    \item If $\xi=\bigvee_{n\geq 0} f^n\xi$, then $r(x)>0$ implies $\Wu{x,r(x)} \subset \xi(x)$.
    \item For every $\mu\in \Gibbs$,
        \begin{enumerate}
          \item $r>0$ for $\mu$-almost every point, and
          \item for $t\in[0,1]$
          \begin{align}\label{eq:sublinearr}
          \mu(r<t)\leq Ct.
          \end{align}
        \end{enumerate}
    \end{enumerate}
\end{theorem}

\begin{definition}
  We say that a Borel set $G$ is of total $u$-Gibbs probability if $\mu\in\Gibbs$ implies $\mu(G)=1$.
\end{definition}

The theorem above implies $G=\{r>0\}$ is of total $u$-Gibbs probability, and that if $\mu\in \Gibbs$, then $\xi$ is an increasing measurable partition with respect to $\mu$, which contains relatively open sets for $\mu$-almost every point. The partition $\eta$ has the advantage that the boundary of its atom are contained in a regular submanifold. 

\begin{remark}
  The same theorem holds for projective $u$-Gibbs measures. One simply lifts the partition using $\hat\pi:\hat M\to M$.
\end{remark}

We assume that $f$ is equipped with an adapted metric, and let $\lambda=\max_x\{\norm{D^u_xf^{-1}}\}$. We work with foliation boxes $U$ for the unstable foliation, whose plaques are disks $D$ with $\diam_u(D)\geq d_0$, hence in particular $\Leb^u(D)\geq v_0$, for some uniform constants $d_0. v_0$.

If $\mu\in \Gibbs$, then disintegrating $\mu|U$ along $\Fu|U$ we get that for $\mu$-almost every $u$-disk $D \subset U$, the densities of the conditional probabilities verify
\[
  y,x\in D\Rightarrow \frac{\rho_D(y)}{\rho_D(x)}=\Delta(x,y) \coloneq \prod_{k=1}^{\oo} \frac{J^uf(y)}{J^uf(x)}.
\]
See \eqref{eq:Delta}: here we omit the dependence on $U$, if clear from the context. By uniform contraction to the past one gets that $\rho_D$ is Lipschitz, bounded away from zero and $\oo$, and $D\mapsto \rho_D$ is continuous (\Cref{lem:regularrho}). In particular there exists some constant $K_U>0$ so that for every $\mu\in \Gibbs$,
\begin{align}\label{eq:comparisonconditionalvsleb}
\frac{1}{K_U}\leq \frac{\mu_D}{\Leb^u_D}\leq K_U,
\end{align}
provided that the conditional $\mu^D$ is defined on $D$.

Next we construct a partition subordinated to the unstable foliation. Given $p\in M$ it is easy to construct coordinates $z:Q\to \R^{\dim M}$ around $p$, with $Q$ compactly contained in a foliation box $U$, and so that $\partial Q$ consists of sets of the form 
\[
  S_i=\{z_i=a\}\quad a=\pm 1
\]
which are transverse to $E^u$. Observe that if $N_{\eps}(S_i)$ denotes the $\eps$-neighborhood of $S_i$, then its volume is comparable with $C_i\eps$, and hence by \eqref{eq:comparisonconditionalvsleb} we get a constant $C_Q>0$ and $\eps_Q>0$ so that for every $\mu\in \Gibbs$,
\[
  0<\eps<\eps_Q\Rightarrow \mu(N_{\eps}(\partial Q))\leq C_Q \eps.
\]
Cover $M$ with finitely many $Q_1,\cdots, Q_s$. With no loss of generality we assume that if $Q_i \subset U_i, Q_j \subset U_j$ and $Q_i \cap Q_j\neq \emptyset$, then $U_i\cup U_j$ is contained in a foliation box of diameter less than $R$, for a fixed prescribed $R>0$. Define
\[
   S=\bigcup_j \partial Q_j\qquad C_S=\max_j C_{Q_j}\qquad \eps_S=\max_j \eps_{Q_j}
\]

Consider the partition
\begin{align*}
A_1= Q_1,\qquad A_i=Q_i\setminus \bigcup_{j<i} Q_j.
\end{align*}
Then $\mathcal{A}=\{A_1,A_2,\ldots\}$ is a finite partition, and the frontier of their atoms is piecewise differentiable. 

Let $\eta$ be the partition obtained refining $\mathcal A$ with the local unstable plaques inside $U_1,\cdots, U_s$. It is direct to check that $\eta$ is measurable, subordinated to the unstable foliation of $f$.

Each $Q_i$ is away from the boundary of $U_i$: we let $\delta=\min_i d(\partial Q_i,\partial U)$, and define
\[
  b(x)\coloneq \min\{\delta, \frac{1}{2}d(x,S)\}.
\]
Note that if $\gamma$ is a curve starting at $x\in U_i$ has length less than $b(x)$, then it stays in $U_i$ and cannot intersect $S$. It follows that if $x\in Q_i$, the curve remains inside, and $\Im(\gamma) \subset \mathcal{A}(x)$. Therefore
\[
  b(x)>0\Rightarrow \Wu{x,b(x)}\subset \eta(x).
\] 
Next, we replace $\eta$ by an increasing partition. Let
\[
  \xi=\bigvee_{n=0}^{\oo} f^n\eta;
\]
recall that $f^n\eta$ is the partition with atoms $f^n\eta(x)=f^{n}(\eta(f^{-n}x))$. By definition, $f\xi\leq \xi$, so $\xi$ is increasing, and countably generated.

Define 
\[
  r(x)=\inf\{\frac{b(f^{-n}x)}{\lam^n}\}.
\]
Then $r:M\to[0,1]$ is Borel. If $r(x)>0$ and $y\in \Wu{x,r(x)}$, it follows that for every $n\geq 0$,
\begin{align*}
\disu{f^{-n}x}{f^{-n}y}\leq \lam^n r(x)\leq b(f^{-n}x)\Rightarrow f^{-n}(y)\in\eta(f^{-n}x) 
\end{align*}
hence\footnote{Here $\disu{x}{y}$ denotes the intrinsic distance between two points in the same $u$-leaf} $y\in \xi(x)$. That is,
\[
  \Wu{x,r(x)} \subset \xi(x)\qquad\text{if } r(x)>0.
\]
Take $t<\min\{\delta,\frac{\eps_S}{2}\}$. Then, if $r(x)<t$ there exists some $n\geq 0$ so that
\[
b(f^{-n}x)\leq t\lam^n\Rightarrow d(f^{-n}x,S)<2t\lam^n<\eps_S \lam^n   
\]
and therefore
\[
  \{r<t\} \subset \bigcup_{n\geq 0} f^n(N_{2t\lam^n}(S)).
\]
If $\mu\in \Gibbs$, 
\begin{align*}
\mu(r<t)\leq \sum_{n\geq 0}\mu(N_{2t\lam^n}(S))\leq t\frac{2C_S}{1-\lam}.
\end{align*}
which in turn implies
\[
  \mu\{r=0\}=0.
\]
Since $r\leq 1$, we can find $C\geq \frac{2C_S}{1-\lam}$ so that for every $0\leq t\leq 1, \mu(r<t)\leq C t$.

\begin{remark}
  Note that since $\diam_u(f^{-n}\xi)\leq  \lam^n R$, it follows that
  \[
  \bigvee_{n\geq 0} f^{-n}\xi=\text{ partition by points.} 
  \]
\end{remark}

We note the following consequence of \eqref{eq:sublinearr}. 

\begin{corollary}
There exists an $f$-invariant subset $G_* \subset G$ so that
\[
  x\in G_*\Rightarrow \bigwedge_{n\geq 0} f^n\xi(x)=\Wu{x}. 
\]  
\end{corollary}

\begin{proof}
Denote, for $x\in G$,
\[
  g(x)=-\log r(x):
\]
observe that
\[
  \{r<e^{-t}\}=\{g>t\}\Rightarrow \forall\mu\in \Gibbs,\ \mu\{g>t\}\leq Ce^{-t}.
\]
Since $\xi(f^{-n}x) \supset \Wu{f^{-n}x.r(f^{-n}x)}$, for every $x$, $f^n\xi(f^{-n})(x) \supset \Wu{x,\frac{r(f^{-n}x)}{\lam^n}}$. Fix $\eps>0$ so that $\kappa=\frac{e^{-\eps}}{\lam}>1$, and consider $X_n=\{x:r(f^{-n}x)\geq e^{-\eps n}\}=\{x:g(f^{-n}x)\leq \eps n\}$. By \eqref{eq:sublinearr}, for any $\mu\in \Gibbs$
  \[
  \sum_{n\geq 0}\mu(x:g(f^{-n}x)>\eps n)<\oo
  \]
  therefore, by Borel--Cantelli, $X=\liminf_n X_n$ has total $u$-Gibbs probability. Consider $G_*=X\cap G$ and note that for $x\in G_*$,
  \[
    f^n\xi(f^{-n}x)\supset \Wu{x,\kappa^n}, 
  \]
  which, as $n\mapsto\oo$ increases to the whole unstable leaf $\Wu{x}$.
\end{proof}

It follows that $\xi$ verifies all properties needed for the arguments in \cite{LedYoungI,LedYoungII}, with the advantage of not depending on a reference $u$-Gibbs measure.

\subsection{\texorpdfstring{$u$}{u}-tilings}
\label{ssec:utilings}

In this part we prove \Cref{lem:goodtiling}.

We start by constructing a Borel set $Z \subset M$ so that for every $x$, $Z\cap \Wu{x}$ is $4R$ separated: this set will provide the center of the cells.  

Take $T_1,\cdots, T_q$ transverse disks to $\Fu$ so that $M=\bigcap_i \Wu{T_i,R}$, and, reducing $R$ if necessary, so that $x\in T_i$ implies $\Wu{x,4R}\cap T_i=\{x\}$. We start with $T_1$, and remove all points in $\bigcup_{i>1} T_i\cap \Wu{T_1,4R}$; then we proceed with $T_2$, and so on. This gives us a Borel set $Z \subset \bigcup_i T_i$ so that
\begin{align*}
&z\neq z'\in Z\Rightarrow d_u(z,z')\geq 4R\\
&\Wu{Z,5R}=M.
\end{align*}

Now we use the partition $\eta$ of the previous part and make a measurable selection to group atoms near the elements in $Z$. We forget about the atoms of $\eta$ contained in $S$. For the others $E\in \eta$ with $E\setminus S\neq\emptyset$, we claim the following

\begin{claim}
There exists a measurable function $p:M/\eta\to M$ so that $p_E\in E\setminus S$, whenever this set is non-empty.
\end{claim}

\begin{proof}
Recall that $\eta$ is obtained by refining $\mc A$ by local unstable plaques. It suffices then to fix some $U$ foliation box with chart
$\Psi: D^u\times T\to U$. Denote $D_t=\Psi(D^u\times \{t\})$ the corresponding plaque. We fix $A\in\mc A$: again, it suffices to make the measurable choice on each of these finitely many atoms. The atoms of $\eta$ inside $A\cap U$ are of the form $E_r=A\cap D_t$

If $t_0\in T$ is so that $E_{t_0}\setminus S\neq\emptyset$, then we can choose $z\in D^u$ so that $\Psi(z,t_0)\in E_{t_0}\setminus S$, and since $A\setminus S$ is open, we can find a neighborhood $T_{t_0} \subset T$ so that $t\mapsto s_0(t)=\Psi(z,t_0)\in E_{t_0}\setminus S$. Cover $T$ by countably many of such neighborhoods $T_0, T_1,\ldots$ and define $p_{E_t}=s_j(t)$, where $j$ is the first index so that $t\in T_j$. It follows that $p_E$ is measurable.
\end{proof}

Now we do an assignment rule for the atoms of $\eta$: given $E\in\eta$ we choose $z\in Z$ which is closest to $p_E$ and assign $E$ to $z$. Note that there are finitely many elements of $Z$ inside $\Wu{p_E,5R}$, and we select one (measurably, for example, using the previous construction of centers). We denote
\[
  O_z=\bigcup_{E\text{ assigned to }z} (E\setminus S)\qquad Q_z=\overline{O_z}^u
\]
in the relative $u$-closure. Note that $O_z$ is leafwise open.

Take $x\in \Wu{z,R}\setminus S$ and let $E=\eta(x)$. Then $d^u(p_E,z)<2R$, and if $z'\neq z$, 
\[
\disu{p_E}{z'}\geq 4R- \disu{p_E}{z}>2R.  
\]
It follows that $E$ is assigned to $z$, and $\Wu{z,R}\setminus S \subset O_z$. Conversely, if $x\in O_z$, then $\disu{x}{p_E}\leq R$ and $\disu{p_E}{z}\leq 5R$ so
\[
  \Wu{z,R}\setminus S \subset O_z \subset \Wu{z,6R}
\]
Taking closure, and using that $S$ has empty interior on every leaf, we get 
\[
  \Wu{z,R} \subset Q_z \subset \Wu{z,6R}
\]
This also gives the uniform diameter and volume estimates for $Q_z$: $\diam_u(Q_z)\leq D\coloneq 12R$ and $v\leq \Leb^u(Q_z)\leq V$, for some $v,V>0$.

Denote $\mc Q=\{Q_z:z\in Z\}$: this is locally finite, since the centers of the cells are separated, and the union of all cells cover the manifold, since the complement of $S$ is dense.  Note that if $x\not\in S$, then it has a relative neighborhood inside one and only one $Q_z$. This shows that
\begin{align*}
Q_z\setminus O_z \subset S\Rightarrow \partial_u Q_z \subset S 
\end{align*} 
and
\begin{align*}
z\neq z'\Rightarrow Q_z\cap Q_{z'} \subset S.
\end{align*}

This finishes the proof.


\printbibliography[heading=bibintoc]

\bigskip

\begingroup
\small
\setlength{\parindent}{0pt}
\setlength{\parskip}{1em}

\textsc{Pablo D. Carrasco}\\
Departamento de Matemática, ICEx,
Universidade Federal de Minas Gerais,
Belo Horizonte, MG, Brazil\\
Email: \href{mailto:pdcarrasco@mat.ufmg.br}
{\texttt{pdcarrasco@mat.ufmg.br}}

\textsc{Davi Obata}\\
Department of Mathematics,
Brigham Young University,
Provo, UT 84602, USA\\
Email: \href{mailto:davi.obata@mathematics.byu.edu}
{\texttt{davi.obata@mathematics.byu.edu}}

\textsc{Radu Saghin}\\
Instituto de Matemáticas,
Pontificia Universidad Católica de Valparaíso,
Valparaíso, Chile\\
Email: \href{mailto:radu.saghin@pucv.cl}
{\texttt{radu.saghin@pucv.cl}}

\textsc{Jiagang Yang}\\
Departamento de Geometria,
Instituto de Matemática e Estatística,
Universidade Federal Fluminense,
Niterói, RJ, Brazil\\
Email: \href{mailto:jiagangyang@id.uff.br}
{\texttt{jiagangyang@id.uff.br}}

\par
\endgroup
\end{document}